\documentclass[12pt,a4paper]{amsart}

\usepackage{amssymb}
\usepackage{amsfonts}
\usepackage{amsmath}
\usepackage[mathscr]{eucal}
\usepackage{mathrsfs}
\usepackage{comment}
\usepackage{bm}

\usepackage[top=30truemm,bottom=30truemm,left=30truemm,right=30truemm]{geometry}

\usepackage{amsthm}

\newtheorem{thm}{Theorem}[section]
\newtheorem{prop}[thm]{Proposition}
\newtheorem{lem}[thm]{Lemma}
\newtheorem{cor}[thm]{Corollary}
\newtheorem{defn}{Definition}

\numberwithin{equation}{section}

\def\Z{{\mathbb Z}}
\def\Q{{\mathbb Q}}
\def\R{{\mathbb R}}
\def\C{{\mathbb C}}

\def\II{{\mathbb I}}
\def\A{{\mathbb A}}

\def\emp{\varnothing}
\def\fa{{\mathfrak a}}

\def\fd{{\mathfrak d}}

\def\fg{{\mathfrak g}}
\def\fh{{\mathfrak h}}

\def\fm{{\mathfrak m}}
\def\fn{{\mathfrak n}}
\def\fo{{\mathfrak o}}
\def\fp{{\mathfrak p}}

\def\ft{{\mathfrak t}}

\def\sa{{\mathsf a}}
\def\sG{\mathsf G}

\def\sM{\mathsf M}
\def\sN{\mathsf N}
\def\sP{\mathsf P}

\def\sm{\mathsf m}
\def\sn{\mathsf n}

\def\sb{\mathsf b}
\def\sf{\mathsf f}

\def\sc{\mathsf c}

\def\sv{\mathsf v}

\def\sT{{\mathsf T}}
\def\cO{\mathfrak o}

\def\cB{{\mathscr B}}

\def\cH{{\mathscr H}}
\def\cM{{\mathcal M}}
\def\cN{{\mathcal N}}

\def\cU{{\mathcal U}}

\def\Re{{\operatorname {Re}}}
\def\Im{{\operatorname {Im}}}
\def\tr{{\operatorname{tr}}}
\def\nr{{\operatorname{N}}}
\def\GL{{\operatorname {GL}}}

\def\End{{\operatorname{End}}}

\def\Hom{{\rm{Hom}}}
\def\diag{{\operatorname {diag}}}

\def\vol{{\operatorname{vol}}}
\def\vV{{\rm{Sym^2}}}
\def\leq{\leqslant}
\def\geq{\geqslant}
\def\bsl{\backslash}

\def\cQ{{\mathcal Q}}

\def\d{{\rm{d}}}

\def\fin{{\rm{\bf {f}}}}
\def\bK{{\bf K}}

\def\b1{{\bold 1}}

\def\sB{{\mathsf B}}

\def\Mcal{{\mathcal M}}

\def\ss{{\mathsf s}}
\newcommand{\sslash}{\mathbin{/\mkern-6mu/}}

\def\sd{{\mathsf d}}

\def\sZ{{\mathsf Z}}

\newcommand{\cchi}{\operatorname{\mbox{{\rm 1}}\hspace{-0.25em}\mbox{{\rm l}}}}
\def\sV{{\mathsf V}}
\def\le{{\leq}}
\def\ge{{\geq}}
\def\GL{{\bf GL}}

\title{The Ranking-Selberg integral on ${\bf GSp}(2)$ for square free levels} 

\author{Seiji Kuga and Masao Tsuzuki}
\subjclass[2020]{primary 11F46, secondary 11F67}

\address{Faculty of Science and Technology, Sophia University, Kioi-cho 7-1 Chiyoda-ku Tokyo, 102-8554, Japan}
\email{s-kuga-2g7@sophia.ac.jp}

\address{Faculty of Science and Technology, Sophia University, Kioi-cho 7-1 Chiyoda-ku Tokyo, 102-8554, Japan}
\email{m-tsuduk@sophia.ac.jp}

\begin{document}
\date{}

\begin{abstract} 
We explicitly compute the Rankin-Selberg type integral introduced by Piatetski-Shapiro over adeles for vector-valued Siegel cusp forms of square-free levels $\Gamma_0(N)$. On the way, for particular test functions in the Bessel models of irreducible admissible representations, exact evaluations of the local zeta-integrals are given.  
\end{abstract}
\maketitle

\section{Introduction} \label{sec:Intro}
Lately, the Bessel periods on the symplectic similitude group $\sG:={\bf GSp}_2$ of rank $2$ progressively gain its importance in the arithmetic of Siegel modular forms. The notion of Bessel period and the associated local model of admissible representations of $\sG$ were  originally introduced by Novodvorski and Piatetski-Shaprio in their work (\cite{NovodvorskiPS}, \cite{Novodvorski}, \cite{PS97}), where a modern treatment of Andrianov's integral representation (\cite{And1}, \cite{And2}) of the spinor $L$-function for Siegel modular forms was developed. Related to this, we should also mention independent work by Sugano (\cite{Sugano84}), which handles vector-valued Hilbert-Siegel cusp forms on arithmetic groups defined by maximal orders of indefinite division algebras. The Bessel period also plays an essential role in a formulation of B\"{o}cherer's conjecture posed by Liu \cite{Liu} in a style of the refined Gan-Gross-Prasad conjecture. This version of B\"{o}cherer’s conjecture has been completely solved by Furusawa and Morimoto (\cite{FurusawaMorimoto}, \cite{FurusawaMorimoto3}). 
In this paper, we compute the Rankin-Selberg type integral a la Piatetski-Shapiro for holomorphic vector-valued Siegel cusp forms of square-free levels explicitly, using  results by Pitale-Schmidt \cite{PitaleSchmidt} on local new vectors on $\sG(\Q_p)\,(p<\infty)$; since we rely on the local-global method, our result is immediately applied to non-holomorphic Siegel cusp forms of discrete series type ({\it cf}. \cite{Lemma}, \cite{LPSZ}). 
We explain the moral of our investigation for scalar-valued Siegel modular forms of weight $l$. Our focus is an asymptotic formula as $N+l\rightarrow \infty$ of the average 
\begin{align}
\frac{1}{[{\bf Sp}_2(\Z):\Gamma_0(N)]} \sum_{\varphi \in \cB(l,N)}\langle E(s,\Lambda,\mu)\mid \varphi\rangle\, B^{\Lambda}(\varphi),
 \label{Intro-f0}
\end{align}  
where $\sG^\#$ is a subgroup of $\sG$ defined by a quadratic field $E$  of discriminant $D<0$ (see \S\ref{sec:Embed}), $E(s,\Lambda,\mu)\,(s\in \C)$ is the Eisenstein series on $\sG^\#(\A)$ attached to a ray class character $\Lambda$ of $E$ and a finite order character $\mu:\A^\times /\Q^\times\rightarrow \C^1$ (see \S\ref{sec:RSint}), $B^\Lambda(\varphi)$ is the Bessel period of $\varphi$ in an orthonormal basis $\cB(l,N)$ of Siegel cusp forms of weight $l$ and square free level $\Gamma_0(N)$, and $\langle E(s,\Lambda,\mu)\mid \varphi\rangle$ is the $L^2$-inner product of $E(s,\Lambda,\mu)$ and $\varphi|\sG^\#(\A)$. 
We suppose the ramifications of $E/\Q$, $\mu$, $\Lambda$, $\pi$ are mutually disjoint. Actually, for application (\cite{KugaTsuzuki2}), we consider a more general average than \eqref{Intro-f0} by requiring $\cB(l,N)$ to consist of joint eigenforms by Hecke operators and by putting the eigenvalue of such a Hecke operator at each $\varphi$ in the formation of the sum \eqref{Intro-f0}. With such a generality, we obtain a completely explicit expression of \eqref{Intro-f0} in terms of 
\begin{itemize}
\item The spinor $L$-function $L(s,\pi,\mu)$ (see \S\ref{sec:SpinorFEQ}) attached to an irreducible cupsidal representation $\pi\cong \otimes_{p}\pi_p$ twisted by $\mu$. 
\item An average of Fourier coefficient of a newform $\varphi_\pi^0$ in $\pi$, denoted by $R(\varphi_\pi^0,E,\Lambda)$ (see \eqref{def-phiEomega}) introduced in \cite{KST}. 
\item An average of local Bessel periods of $\pi_p$ (see \S\ref{sec: SpecAvRSI}).   
\end{itemize}
\vspace{-0.1cm}
For the precise statement including the vector-valued case, we refer to Proposition \ref{Main-Prp1} and Theorem \ref{Main-T1}. Owing to the explicit nature of our result, we are able to  show that, asymptotically as the prime level $N$ grows to infinity  with the weight $l$ being fixed, the contribution to \eqref{Intro-f0} of the old forms, the Yoshida lifts and the Saito-Kurokawa lifts are $0$ (see Theorem \ref{AsymptoticFormula}). 
These results are necessary in our forthcoming paper \cite{KugaTsuzuki2}, where we shall work out the geometric side of a relative trace formula whose spectral side is essentially \eqref{Intro-f0}. We should remark that an average similar to our computed average \eqref{SpectralAveage-f} without the factors ${\widehat{L}}(s+1/2,\pi,\mu)$ and $\widehat{f_S}(\pi)$ is studied in \cite[\S3]{DPSS}. When $E=\Q(i)$, our average \eqref{SpectralAveage-f} is almost the same as the one in \cite[\S8]{Blomer} except the appearance of the local periods. 

Let us explain the structure of this paper. In \S\ref{sec:Prel}, we introduce basic objects related to algebraic subgroups of the symplectic group $\sG$ and Haar measures on their $p$-adic, real and adelic points. In \S\ref{sec:RSint}, the Rankin-Selberg type integral and the global Bessel function for Siegel cusp forms are recalled; after this preparation, the basic identity (Lemma~\ref{RSint-L1}) is stated, which computes the pairing $\langle E(s,\Lambda,\mu),\varphi\rangle$ in \eqref{Intro-f0} in terms of the Mellin transform of the global Bessel function attached to $\varphi$. We review the proof (\cite{PS97}) to determine a constant depending on Haar measures exactly. In \S\ref{sec:FC}, we extend the definition of the quantity $R(\varphi, E,\Lambda)$ for a ray class character $\Lambda$, which was originally considered in \cite{KST} for class group character; we describe its behavior under the Galois conjugation $\Lambda\mapsto \Lambda^\dagger$ (Lemma \ref{BesselConj}). The local multiplicity one theorem for Bessel models of $\pi_p$'s allows us to split $B^\Lambda(\varphi)$ for $\varphi \in \pi$ corresponding to a pure tensor in $\otimes_{p}\pi_p$ to an Euler product of local Bessel functions for $\pi_p$ up to a constant. To define our newform $\varphi_\pi^0$ and identify this constant with $R(\varphi_\pi^0,E,\Lambda)$ (see Lemma \ref{BesselMod-L1}), we fix pairs of a local Bessel functional $\ell_p$ on $\pi_p$ and a local new vector $\xi_p$ of $\pi_p$ such that $\ell_p(\xi_p)=1$; in \S\ref{sec:BesselMod}, we explain that the main result of \cite{PitaleSchmidt} ensures the existence of such $\{(\ell_p,\xi_p)\}_p$ when $\pi$ corresponds to a new forms on $\Gamma_0(N)$. In \S\ref{sec:CompLZ}, we compute the local zeta-integrals by using results of \cite{PitaleSchmidt}. The unramified computations are known (see \S\ref{sec:Ur1}). Other cases, as well as the explicit determination of the local periods of prime level (see \S\ref{sec:Local period}), seem to be new, so that some details are given in \S\ref{sec:IIIa}, \S\ref{sec:IVb}, and \S\ref{sec:Iold}. In \S\ref{sec:SpinorFEQ}, the functional equation of the $L$-function is deduced from the basic identity by using Lemma \ref{BesselConj}. Although a ``nice'' functional equation of the $L$-function itself is known in a broader generality (\cite[Lemma 1.2]{Schmidt2018}) owing to Arthur's classification of the discrete spectrum, we include this section for the sake of completeness and to provide a proof free from Arthur's result, confirming that the local $L$-factor and the local $\varepsilon$-factor defined by \cite{PS97} coincides with the ones proposed in \cite[Tables 2,3]{Schmidt}. In section \S\ref{sec: SpecAvRSI}, we prove Proposition \ref{Main-Prp1} and Theorem \ref{Main-T1}; by invoking a result by Furusawa-Morimoto (\cite[Theorem 8.1]{FurusawaMorimoto}), we obtain an explicit expression of \eqref{Intro-f0} solely in terms of $L$-functions (Theorem \ref{Main-T2}). To extend the result of Furusawa-Morimoto to a case when $\Lambda$ ramifies, we need a formula of a local period of the unramified matrix coefficient, we include some detail in \S\ref{sec:CompGP} because this case is not handled in \cite{DPSS}. 

\begin{comment}
this hypothesis is true for $\pi\in\Pi_{\rm cusp}^{(T_\theta,\Lambda)}(\lambda,N)$ when $\Lambda$ is an unramified character, i.e., $\cM=1$ which is the case that $\Lambda\in\widehat{{\rm Cl}(E)}$.

The refined version of the Gan-Gross-Prasad congecture for $({\bf SO}(5),{\bf SO}(2))$ in (\cite{Liu}) states an explicit equality between the Gross-Prasad period twisted by a character on a unipotent subgroup and the special central $L$-values. Thanks to the great work done by Furusawa and Morimoto(\cite{FurusawaMorimoto3}) 
\end{comment}

\section{Preliminaries}\label{sec:Prel}
Let $\sV$ denotes the space of symmetric matrices of degree $2$; we consider $\sV$ as a $\Z$-scheme by defining
$$
\sV(R)=\left\{ \left[\begin{smallmatrix} a & b \\ b & c \end{smallmatrix} \right]\mid a,b,c\in R\right \}
$$
for any commutative ring $R$. The group ${\bf GL}_2(R)$ acts on this space from the right as 
\begin{align}
\sV(R) \times {\bf GL}_2(R)\ni (T,h)\quad \longmapsto \quad  T\ss(h):=\tfrac{1}{\det h
}{}^t h T h \in \sV(R).
\label{GLaction}
\end{align}
Let $\sG$ denote the symplectic similitude group  
$$
\sG:=\{g\in {\bf GL}_4\mid {}^t g \left[\begin{smallmatrix} O & 1_2 \\ -1_2 & O\end{smallmatrix} \right] g=\nu(g)\,\left[\begin{smallmatrix} O & 1_2 \\ -1_2 & O\end{smallmatrix}\right] \quad (\exists \nu(g) \in {\bf GL}_1 \}
$$
(usually denoted by ${\bf GSp}_2$), where $\nu: \sG \rightarrow {\bf GL}_1$ is the similitude character. The center of $\sG$ is denoted by $\sZ$, which coincides with the set of all the scalar matrices in $\sG$. Let $\sP=\sM\sN$ be the Siegel parabolic subgroup of $\sG$, where 
\begin{align*}
\sM&:=\{\sm(A,c):=\left[\begin{smallmatrix} A & O \\ O &{}^t A^{-1}c \end{smallmatrix} \right] \mid A\in {\bf GL}_2\,,c \in {\bf GL}_1\}, \\
\sN&:=\{\sn(X):=\left[\begin{smallmatrix} 1_2 & X \\ O & 1_2 \end{smallmatrix} \right] \mid X \in \sV\}.
\end{align*}
Note that $\nu(\sm(A,c))=c$ and $\nu(\sn(X))=1$. Let
\begin{align*}
\sT:=\left\{\sm( \left[\begin{smallmatrix} a & 0 \\
0 & b \end{smallmatrix} \right],c)  \mid a,b,c \in \GL_1\right\}.
\end{align*}
The Weyl group $W:=N_\sG(\sT)/\sT$ of $\sG$ is generated by the images of
$$s_1:=\left[\begin{smallmatrix} 0 & 1 & 0 & 0 \\ 1 & 0 & 0 & 0 \\ 0 & 0 & 0 & 1 \\ 0 & 0 & 1 & 0 \end{smallmatrix}\right],\quad s_2:=\left[\begin{smallmatrix} 0 & 0 & 1 & 0 \\ 0 & 1 & 0 & 0 \\ -1 & 0 & 0 & 0 \\ 0 & 0 & 0 & 1 \end{smallmatrix}\right].
$$

\subsection{Embedding of groups}\label{sec:Embed}
Let $E=\Q(\sqrt{D})\subset \C$ be an imaginary quadratic field of discriminant $D<0$, so that $\sqrt{D}=i|D|^{1/2}$ with $i$ being the imaginary unit. Let $\cO_{E}$ denote the ring of integers of $E$. Choose a $\Z$-basis $\{1,\theta\}$ of $\cO_E$, so that 
\begin{align}
 \theta-\bar\theta=-\sqrt{D},
 \label{thetaD}
\end{align}
where $\tau \mapsto \bar\tau$ denotes the non-trivial automorphism of $E/\Q$.

The different $\fd_{E/\Q}$, defined as $\fd_{E/\Q}^{-1}:=\{\tau \in E\mid \tr_{E/\Q}(\tau \cO_{E})\subset \Z\}$, is ${\sqrt{D}}\cO_E$; thus a symplectic $\Z$-bi-linear form $ \langle \cdot, \cdot \rangle: \cO_E^2 \times \cO_E^2 \longrightarrow \Z$ is defined by 
\begin{align*} 
\langle x,y \rangle:=\tr_{E/\Q}\left(\tfrac{-1}{\sqrt{D}}\det \left[ \begin{smallmatrix} x_1 & y_1 \\ x_2 & y_2 \end{smallmatrix} \right] \right), \quad x=\left[\begin{smallmatrix} x_1 \\ x_2 \end{smallmatrix} \right],\, y=\left[\begin{smallmatrix} x_1 \\ x_2 \end{smallmatrix} \right] \in \cO_E^2.
\end{align*}
Note that the following vectors of $\cO_E^2$ form a symplectic $\Z$-basis, i.e., $\langle v_i^{+},v_j^{-}\rangle=\delta_{ij}$ and $\langle v_{i}^{+},v_j^{+}\rangle=\langle v_i^{-},v_j^{-}\rangle=0$ for all $i,j$: 
\begin{align*}
v_{1}^{+}:=\left[\begin{smallmatrix} 1 \\ 0 \end{smallmatrix} \right], \quad 
v_{2}^{+}:=\left[\begin{smallmatrix} \theta \\ 0 \end{smallmatrix} \right], \quad 
v_{1}^{-}:=\left[\begin{smallmatrix} 0 \\ -\bar \theta  \end{smallmatrix} \right], \quad 
v_{2}^{-}:=\left[\begin{smallmatrix} 0  \\ 1 \end{smallmatrix} \right]. 
\end{align*} 
For any ring $R$, set 
$$
\sG^{\#}(R):=\{h\in {\bf GL}_2(\cO_E \otimes_\Z R)\mid \det(h) \in R^\times\}.
$$
For $h\in \sG^\#(R)$, viewed as an $R$-linear automorphism of $(\cO_E\otimes_\Z R)^2$, let $\iota_\theta(h)$ denote the $4\times 4$-matrix representing $h$ by the $R$-basis $\{v_1^{+},v_2^{+},v_1^{-},v_2^{-}\}$ of $(\cO_E\otimes_\Z R)^2$, i.e., 
\begin{align}
[h(v_1^{+}), h(v_2^{+}), h(v_1^{-}), h(v_2^{-})]=[v_1^{+},v_2^{+}, v_1^{-},v_2^{-}]\,\iota_{\theta}(h). 
\label{Embed-f0'}
\end{align}
Then $\iota_{\theta}(h) \in \sG(R)$ and $\nu(\iota_\theta(h))=\det h$. Thus, we have an embedding 
$$
\iota_{\theta}:\sG^\#(R) \longrightarrow \sG(R). 
$$ 
Let $\sB^\#$ be the Borel subgroup of $\sG^\#$ such that $\sB^\#(R)$ coincides with the set of all points 
\begin{align}
\left[\begin{smallmatrix}a \tau & \beta \\ 0& \tau^{-1}\end{smallmatrix}\right], \quad \tau \in (E\otimes_\Q R)^{\times}, \,a\in R^\times,\,\beta \in R.  
 \label{Embed-f0}
\end{align}
Let $\sN^\#$ denote the unipotent radical of $\sB^\#$, i.e., $\sN^\#(R)$ is the set of points with $\tau=a=1$. Set $\sZ^\#:=\{a 1_2\mid a\in {\bf GL}_1\}$; then $\sZ^\#$ is a subgroup of the center of $\sG^\#$ of index $2$. 

Let $I_{\theta}:E\longrightarrow {\bf M}_2(\Q)$ denote the matrix representation of the regular representation of the $\Q$-algebra $E$ with respect to the basis $\{1,\theta\}$, i.e., 
\begin{align}
[\tau,\tau\theta]=[1,\theta]\,I_{\theta}(\tau), \quad \tau\in E, 
\label{Embed-DefI}
\end{align}
or explicitly, 
\begin{align}
I_{\theta}(a+b\theta)=\left[\begin{matrix} a & -b\nr_{E/\Q}(\theta) \\ b & a+b\tr_{E/\Q}(\theta) \end{matrix} \right], \quad a,b\in \Q.
 \label{Embed-I}
\end{align}
For $\beta=b_2+b_3\theta \in E_R$ with $b_2,b_3\in R$, define an element $X_\beta$ of $\sV(R)$ as 
\begin{align}
X_\beta:=\left[\begin{smallmatrix} b_1 & b_2 \\ b_2 & b_3
\end{smallmatrix} \right] \quad \text{with $b_1:=-b_2\tr_{E/\Q}(\theta)-b_3 \nr_{E/\Q}(\theta)$}. 
\label{Def-Xbeta}
\end{align}
We have 
\begin{align}
\iota_{\theta}\left(\left[\begin{smallmatrix} \tau & 0 \\ 0& a \tau^{-1}\end{smallmatrix}\right] \right)&=\sm(I_{\theta}(\tau), a), \quad a\in R^\times,\,\tau\in E_{R}^\times, 
 \label{Embed-f00}
\\
\iota_{\theta}\left(\left[\begin{smallmatrix}1 & \beta  \\ 0& 1\end{smallmatrix}\right]\right)&=\sn (X_\beta), \quad \beta \in E_{R}.
 \label{Embed-f000}
\end{align}
By formulas \eqref{Embed-I} and \eqref{Embed-f00}, we have $\iota_{\theta}(a 1_2)=a 1_4$ for $(a\in R^\times)$, which implies 
\begin{align}
\iota_\theta(\sZ^{\#})=\sZ. 
\label{Embed-center}
\end{align}
We then have $\nr_{E/\Q}(x+\theta y)=\left[x,y\right] T\left[\begin{smallmatrix} x \\ y \end{smallmatrix}\right]$ with
\begin{align}
T_{\theta}&:=\left[\begin{matrix}1 & 2^{-1}\tr_{E/\Q}(\theta) \\ 2^{-1}\tr_{E/\Q}(\theta) & \nr_{E/\Q}(\theta) \end{matrix} \right].
 \label{Def-Ttheta}
\end{align}
Let $\sV^{T_\theta}(\Q)$ be the orthogonal of $T_\theta$ in $\sV(\Q)$ with represent to the non-degenerate quadratic form $\tr(XY)$ on $\sV(\Q)$, i.e.,
\begin{align*}
\sV^{T_\theta}(\Q)&
:=\{X\in \sV(\Q)\mid \tr(T_\theta X)=0\}.
\end{align*}  
Note that $\tr(T_\theta^2)=1+(\tr_{E/\Q}(\theta)/2)^2+\nr_{E/\Q}(\theta)^2>0$, $\det(T_\theta)=-D/4>0$ by \eqref{thetaD}, and $T_\theta\in \sV(\Z)$. We have $\sV^{T_\theta}(\Q)=\{X_\beta\mid \beta \in E\}$, and
\begin{align}
\sV(\Q)=\Q T_\theta \oplus \sV^{T_\theta}(\Q), \qquad  \iota_{\theta}(\sN^\#)=\sn(\sV^{T_\theta})
\label{Embed-ff0}
\end{align}
by \eqref{Embed-f000}. The group $\sM$ acts on the space of rational homomomorhisms $\Hom(\sN,{\mathbb G}_a)$ by the rule ${\rm Ad}^*(m)\chi(n)=\chi(m^{-1}nm)$ for $m\in \sM$, $n\in \sN$ and $\chi \in {\Hom}(\sN,{\mathbb G}_a)$. For $T\in \sV$, let $\sM_{T}$ denote the stabilizer of $\chi_{T}: \sn(X)\mapsto \tr(TX)$, and $\sM_{T}^\circ$ the identity component with respect to the Zariski topology. Then, for any $\Q$-algebra $R$, 
$$
\sM_{T}^\circ(R)=\{\sm(I_{\theta}(\tau),\nr_{E/\Q}(\tau))\mid \tau \in E_{R}\}.
$$
Define a bilinear form on $\sV(R)$ as $(X,Y):=-\tr(XY^{\dagger})$ and ${\bf SO}(\sV(R))$ as the special orthogonal group of $(\cdot,\cdot)$. 

In the proof of Lemma \ref{BessePerL2}, we use the following lemma:
\begin{lem}\label{acci.isom.}
Let $F$ be a field with characteristic $0$. 
\begin{itemize}
\item[(i)] The map $\ss:{\bf GL}_2(F)\to{\bf GL}(\sV(F))$ induces an isomorphism
$$
{\bf PGL}_2(F) \stackrel{\sim}{\longrightarrow} {\bf SO}(\sV(F)). 
$$
\item[(ii)] The map $\ss\circ I_\theta: (E\otimes_\Q F)^\times/F^\times \to {\rm Stab}_{{\bf SO}(\sV(F))}(T_\theta)$ is an isomorphism.
\end{itemize}
\end{lem}

\subsection{Open compact subgroups at finite places} \label{sec:OCSGP}
Let $p$ be a prime number. Set $E_p:=E\otimes_\Q \Q_p$ and $\cO_{E,p}:=\cO_{E} \otimes_\Z \Z_p$. 
Then, $\bK^\#_p:=\sG^\#(\Z_p)$, which coincides with the stabilizer in $\sG^\#(\Q_p)$ of the $\cO_{E,p}$-lattice $\cO_{E,p}^2 \subset E_p^2$, is a maximal compact subgroup of $\sG^\#(\Q_p)$, and $\bK_p:=\sG(\Z_p)$ is the standard maximal compact subgroup of $\sG(\Q_p)$. For a non-zero ideal $\fn\subset \Z_p$, set 
$$
\bK_0(\fn):=\{\left[\begin{smallmatrix} A & B \\ C & D \end{smallmatrix}\right] \in \sG(\Z_p)\mid C\in \fn\}.
$$
For a non-zero ideal $\fa\subset \cO_{E,p}$, set 
$$
\bK^{\#}_0(\fa):=\{\left[\begin{smallmatrix} a & b  \\ c & d \end{smallmatrix} \right]\in \sG^\#(\Z_p) \mid c\in \fa\}.
$$
Thus, $\bK_0(\Z_p)=\bK_p$ and $\bK_0^\#(\cO_{E,p})=\bK_p^\#$.  
\begin{lem} \label{OPSGP-L1}
 For a non-zero ideal $\fn \subset \Z_p$, we have $\bK_0^\#(\fn\cO_{E,p})=\iota_{\theta}^{-1}(\bK_0(\fn))$.  
\end{lem}
\begin{proof} Indeed, both $\iota_\theta^{-1}(\bK_p)$ and $\bK^\#_p$ coincides with the stabilizer of $\cO_{E,p}^2=\langle v_1^{+},v_2^{+},v_1^{-},v_2^{-}\rangle_{\Z_p}$ in $\sG^\#(\Q_p)$. Hence $\iota_\theta^{-1}(\bK_p)=\bK^\#_p$. In the remaining part of the proof, we suppose $\fn\subset p\Z_p$. Then for $g=\left[\begin{smallmatrix} A & B \\ C & D \end{smallmatrix}\right]\in \bK_p$, we have $g\in \bK_0(\fn)$ if and only if $g(v_1^{+}), g(v_2^+)\in \langle v_1^{+},v_2^{+}\rangle_{\Z_p}+\fn \langle v_1^{-},v_2^{-}\rangle _{\Z_p}$. For $g=\iota_{\theta}(h)$ with $h=\left[\begin{smallmatrix} a & b  \\ c & d \end{smallmatrix} \right]\in \bK_p^\#$, this last condition becomes $
\left[\begin{smallmatrix} a \\ c \end{smallmatrix} \right], \, 
\left[\begin{smallmatrix} a \theta \\ c \theta  \end{smallmatrix} \right] \in \left[\begin{smallmatrix} \cO_{E,p} \\ \fn \cO_{E,p} \end{smallmatrix}\right],
$ or equivalently, $c \in \fn\cO_{E,p}$. 
\end{proof}
For any $N\in \Z_{>0}$, define open compact subgroups $\bK_0(N)\subset \sG(\A_\fin)$ and $\bK^\#(N\cO_E)\subset \sG^\#(\A_\fin)$ as 
$$
\bK_0(N):=\prod_{p<\infty}\bK_0(N\Z_p), \quad \bK^\#_0(N\cO_E):=\prod_{p<\infty} \bK_0^\#(N\cO_{E,p}).
$$

\subsection{Maximal compact subgroup at the archimedean place} \label{sec:MCPtArc}
The identity connected component of $\sG(\R)$ is $\sG(\R)^0=\{g\in \sG(\R)\mid \nu(g)>0\}$. Set $\bK_\infty:=\sG(\R)^0\cap {\bf O}(4)$, which is a maximal compact subgroup of $\sG(\R)^0$ given as 
\begin{align*}
\bK_\infty=\{\left[\begin{smallmatrix} A & B\\ -B & A \end{smallmatrix} \right]\mid A,B\in {\bf M}_2(\R),\,A+iB\in {\bf U}(2)\}. 
\end{align*}
Note that $\bK_\infty \subset {\bf Sp}_2(\R)$. The action of $\sG(\R)^0$ on the Siegel upper-half space $\fh_2:=\{Z\in {\bf M}_2(\C)\mid {}^t Z=Z,\,\Im(Z)\gg 0\}$, denoted by $(g,Z)\mapsto g\langle Z\rangle$
is defined by the usual formula, i.e., 
$$
g\langle Z\rangle:=(AZ+B)(CZ+D)^{-1}, \quad 
 g=\left[\begin{smallmatrix} A & B \\ C & D \end{smallmatrix}\right] \in \sG(\R)^0, \quad Z\in \fh_2.
 $$
 The stabilizer of $i 1_2\in \fh_2$ coincides with $\bK_\infty$. Set $\gamma_\infty:=\sm(-1_2,-1)=\left[\begin{smallmatrix} -1_2 & 0 \\ 0 & 1_2\end{smallmatrix} \right]\in \sG(\R)$; then $\bK_\infty\cup \bK_\infty\gamma_\infty$ is a maximal compact subgroup of $\sG(\R)$. 
The four vectors
\begin{align*}
u_1:=\left[\begin{smallmatrix} 1 \\ 0 \end{smallmatrix}\right], \quad
u_2:=\left[\begin{smallmatrix}i \\ 0 \end{smallmatrix}\right], \quad
u_3:=\left[\begin{smallmatrix} 0 \\ -i \end{smallmatrix}\right], \quad
u_4:=\left[\begin{smallmatrix} 0 \\ -1 \end{smallmatrix}\right].
\end{align*}
in $E_\R^2$ form an $\R$-basis such that
\begin{align}
[v_1^{+},v_2^{+},v_1^{-},v_2^{-}]=[u_1,u_2,u_3,u_4] \,(b_\R^{\theta})^{-1}
\label{Embed-f2}
\end{align}
with 
\begin{align}
b_\R^{\theta}:=\sm(A_\theta,2^{-1}\sqrt{|D|})^{-1}, \quad A_\theta:=\left[\begin{matrix} 1 & 2^{-1}\tr_{E/\Q}(\theta) \\ 0 & -2^{-1}\sqrt{|D|} \end{matrix} \right]\, \in {\bf GL}_2(\R).
\label{Def-b}
\end{align}
Note that $b_{\R}^\theta\in \sG(\R)^0$ because $\nu(b_{\R}^\theta)^{-1}=2^{-1}\sqrt{|D|}>0$. Set 
$$\bK^{\#}_\infty:={\bf U}(2)\cap \sG^\#(\R)=\{k^\#\in {\bf U}(2)\mid \det(k^\#)=\pm 1\},
$$
which is a maximal compact subgroup of $\sG^\#(\R)$. The (topological identity connected component $(\bK^{\#}_\infty)^0$ of $\bK^\#_\infty$ is ${\bf SU}(2)$, and $\bK^\#_\infty=(\bK^\#_\infty)^0 \cup (\bK_\infty^\#)^0\delta_\infty$ with $\delta_\infty:=\left[\begin{smallmatrix} 1 & 0 \\ 0 & -1\end{smallmatrix}\right]$. A general element of ${\bf SU}(2)$ is written in the form 
\begin{align}
h=\left[\begin{smallmatrix} a & -\bar b \\ b & \bar a\end{smallmatrix} \right]\quad \text{with $a=a'+ia'',b=b'+ib''\in \C$, $|a|^2+|b|^2=1$}.
\label{ElementSU2}
\end{align}
For such an $h$, a computation reveals the relation 
\begin{align}
[hu_1,hu_2,hu_3,hu_4]=[u_1,u_2,u_3,u_4]\left[\begin{smallmatrix} A & B \\ -B & A \end{smallmatrix} \right] \quad \text{with $A:=\left[\begin{smallmatrix}a' & -a'' \\ a'' & a'\end{smallmatrix} \right],\,B:=\left[\begin{smallmatrix}b'' & b' \\ b' & -b'' \end{smallmatrix} \right]$}. 
 \label{Embed-f1}
\end{align}
 \begin{lem} \label{MCPtArc-L1}
We have $\iota_{\theta}^{-1}(\bK_\infty)=(\bK_\infty^{\#})^0$. For $k_\infty^\#\in (\bK^\#_\infty)^0$ as in \eqref{ElementSU2}, and define $A,B\in {\bf M}_2(\R)$ as in \eqref{Embed-f1}. Then,  
\begin{align}
\iota_\theta(k_\infty^\#)=b_{\R}^\theta \left[\begin{smallmatrix} A & B\\ -B & A \end{smallmatrix} \right]\,(b_{\R}^\theta)^{-1}.
 \label{MCPtArc-L1-f1}
\end{align}
\end{lem}
\begin{proof} The formula in \eqref{MCPtArc-L1-f1} follows from \eqref{Embed-f0}, \eqref{Embed-f1} and \eqref{Embed-f2} immediately. From 
$$
[hu_1,hu_2,hu_3,hu_4]=[u_1,u_2,u_3,u_4](-\gamma)
$$
and $-\gamma\not\in \bK_\infty$, the assertion follows. 
\end{proof}

\subsection{Haar measures}
For locally compact unimodular topological groups $H$ relevant to us, we fix Haar measures $\eta_{H}$ on $H$ in the following manner. Let $\A$ be the adele ring of $\Q$ and $\psi:\A/\Q\rightarrow \C^1$ the basic character.
\subsubsection{} \label{sec:HaarEQ}
On the additive group $\C$, define $\d \eta_{\C}(\tau):=\d x\d y=2^{-1}|\d \tau \wedge \d\bar\tau|$ for $\tau=x+iy\,(x,y\in \R)$ with $\d x\,\d y$ being the Lebesgue measure on $\R^2$. 

The Haar measure $\eta_{\C^\times}$ on $\C^\times$ (resp. $\R^\times)$ is defined by $\d \eta_{\C^\times}(\tau)=|\tau|_{\C}^{-1}\d\tau$ for $\tau=x+iy\in \C$ (resp .$|x|_\R^{-1}\d x$). For each $p<\infty$, $E_p^\times$ (resp. $\Q_p^\times$) is endowed with the Haar measure $\eta_{E_p^\times}$ such that $\eta_{E_p^\times}(\cO_{E,p}^\times)=1$ (resp. $\eta_{\Q_p^\times}(\Z_p^\times)=1$). Then, viewing $\A_E^\times$ as the restricted product of $E_p^\times$, we define $\eta_{\A_{E}^\times}=\prod_{p\leq \infty}\eta_{E_p^\times}$. Similarly, we set $\eta_{\A^\times}=\prod_{p\leq \infty} \eta_{\Q_p^\times}$. 

\subsubsection{} \label{sec:MeasUnpt}
For a matrix $X=\left[\begin{smallmatrix} a & b \\ c & d \end{smallmatrix} \right]$, set $X^\dagger:=\left[\begin{smallmatrix} d & -b \\ -c & a \end{smallmatrix} \right]$, so that $XX^\dagger =(\det X)\,1_2$, and $X\mapsto X^\dagger$ induces an involution of the $\Q$-vector space $\sV(\Q)$, which is viewed as a quadratic space with the symmetric bilinear form $\tr(XY^\dagger)$. We endow $\sV(\A)$ with the self-dual Haar measure  with respect to the self-duality defined by the bi-character $(X,Y)\mapsto \psi(\tr(XY^\dagger))$. Let $\eta_{\A}$ (resp. $\eta_{\A_E}$) be the self-dual Haar measure on $\A$ (resp. $\A_{E}$) with respect to the bi-character $(x,y)\mapsto \psi(xy)$ (resp. $(\alpha,\beta)\mapsto \psi(\tr_{E/\Q}(\alpha \bar \beta))$). A computation shows the relation $-\tr(X_\alpha X_\beta^\dagger)=\tr_{E/\Q}(\alpha \bar \beta)$ for $\alpha,\beta \in \A_E$, where $X_\alpha,X_\beta\in V^{T_\theta}(\A)$ are defined by \eqref{Def-Xbeta}. Thus, by the isomorphism $\A_{E}\ni \beta \mapsto X_\beta\in \sV^{T_\theta}(\A)$, the Haar measure $\eta_{\A_E}$ is transferred to a Haar measure on $\sN^\#(\A)$. Since $\tr(T_\theta T_\theta^\dagger)=-D/2\not=0$, we have the orthogonal direct sum decomposition $\sV(\Q)=\Q T_\theta^\dagger\oplus \sV^{T_\theta}(\Q)$. By this and \eqref{Embed-ff0}, write $X\in V(\A)$ as $X=xT_\theta^\dagger+X_\beta=y T_\theta+X_\alpha$ $(x,y\in \A,\,\beta,\alpha \in \A_{E})$. By definition of the Haar measures, we have $\d\eta_{\sV(\A)}(X=\d\eta_{\A}(x) \otimes \d\eta_{\A_E}(\beta)$. Since the change of variables $(x,\beta)\rightarrow (y,\alpha)$ is given as 
$$
y=\tfrac{-D}{2}\,x, \qquad \alpha=\beta+2^{-1}\tr_{E/\Q}(\theta)(x+y)+(\nr_{E/\Q}(\theta)x-y)\,\theta 
$$
and since $|\tfrac{-D}{2}|_\A=1$, we get $\d\eta_{\sV(\A)}(X)=\d\eta_{\A}(y)\otimes \d\eta_{\A_E}(\alpha)$. Through the identification $\sV(\A)\ni X\mapsto \sn(X)\in \sN(\A)$ and $\beta\ni \A_{E}\mapsto \left[\begin{smallmatrix}1 & \beta \\ 0 & 1 \end{smallmatrix} \right]\in \sN^\#(\A)$, the groups $\sN(\A)$ and $\sN^\#(\A)$ acquire Haar measures. Thus, we have the integration formula: 
\begin{align}
\int_{\sN(\A)}f(n)\, \d n=\int_{\A} \int_{\sN^\#(\A)}f(\sn(x T_\theta)\,\iota_{\theta}(n^\#))\,\d x\,\d n^\#
 \label{UnptInt}
\end{align}
for any $f\in L^2(\sN(\A))$. Note that $\vol(\sN(\Q)\bsl \sN(\A))=\vol(\sN^\#(\Q)\bsl \sN^\#(\A))=1$. 

\subsubsection{}
Let $p\leq \infty$. Any element $g^\#\in \sG^\#(\Q_p)$ is written as 
\begin{align}
g^\#=\left[\begin{smallmatrix} 1 & \beta \\ 0 & 1\end{smallmatrix}\right]
\left[\begin{smallmatrix} \tau & 0 \\ 0 & \bar \tau a \end{smallmatrix}\right] \, k^\#
\label{Def-HaarMeasSGP-f0}
\end{align}
with $a \in \Q_{p}^\times$, $\tau \in E_p^\times$, $\beta\in E_p$ and $k^\# \in \bK_p^\#$. Then, our Haar measure on $\sG^\#(\Q_p)$ is symbolically defined as 
\begin{align}
\d\eta_{\sG^\#(\Q_p)}(g_p^\#)=|a|_p^{2} \times \d\eta_{E_p}(\beta)\, \d\eta_{\Q_p^\times}(a)
\, \d\eta_{E_p^\times}(\tau)\,\d\eta_{\bK_p^\#}(k^\#), 
 \label{Def-HaarMeasSGP-p}
\end{align}
where $\d\eta_{\bK_p^\#}(k^\#)$ is the Haar measure on $\bK_p^\#$ with volume $1$. For $p<\infty$, the Haar measure $\eta_{\sG^\#(\Q_p)}$ is the one with $\vol(\bK_p^\#)=1$. On the adele group $\sG^\#(\A)$, we use the product measure of Haar measures on $\sG^\#(\Q_p)$ for $p\leq \infty$. Then, from \eqref{Def-HaarMeasSGP-p}, we get 
\begin{align}
\int_{\sG^\#(\A)}f(g^\#)\,\d g^\#=
\tfrac{\sqrt{|D|}}{2} \int_{\sN^\#(\A)} \int_{\A^\times} \int_{\A^\times_E} \int_{\bK^\#} 
f&\left(n^\# \left[\begin{smallmatrix}  \tau & 0 \\ 0 & a\bar\tau \end{smallmatrix}\right] \, k^\# \right)|a|_\A^{2}
 \label{HaarMeasSGP-Adele}\\
&\qquad \times \d n^\#\,\d\eta_{\A^\times}(a)\,\d\eta_{\A_E^\times}(\tau)\,\d k^\#,
\notag 
\end{align}
where $\d k^\#$ is the normalized Haar measure on $\bK^\#=\prod_{p\leq \infty} \bK_p^\#$. Note that the measure on $\sN^\#(\A)\cong \A_{E}$ is $\eta_{\A_E}$ defined above, which is $\frac{\sqrt{|D|}}{2}\times \prod_{p} \eta_{E_{p}}\,(p\leq \infty)$ due to the formulas $(\prod_{p}\eta_{E_p})(\A_E/E)=|D|^{1/2}/2$ (\cite[Chap V \S4 Proposition 7]{Weil}) and $\eta_{\A_E}(\A_E/E)=1$.

\subsubsection{}\label{sec:HaarGA}
 We fix $\eta_{\bK_\infty}$ so that $\vol(\bK_\infty)=1$. Then we normalize $\eta_{{\bf Sp}_2(\R))}$ in such a way that the quotient $\eta_{{\bf Sp}_2(\R)}/\eta_{\bK_\infty}$ corresponds to the measure $(\det Z)^{-3}\d X\d Y$ on $\fh_2\cong {\bf Sp}_2(\R)/\bK_\infty$. By $\sG(\R)^0=\sZ(\R)^0\,{\bf Sp}_2(\R) \cong \R_{>0}\times{\bf Sp}_2(\R)$, $\eta_{\sG(\R)}$ is defined by demanding that its restriction to $\sG(\R)^0$ is $\eta_{\R^\times}\otimes \eta_{{\bf Sp}_2(\R)}$. For $p<\infty$, we fix $\eta_{\sG(\Q_p)}$ by $\vol(\bK_p)=1$. Then, $\eta_{\sG(\A)}$ is defined to be the restricted product of $\eta_{\sG(\Q_v)}\,(v\leq \infty)$. 

\subsection{Ray class characters}\label{sec:RyClCh}
 For each place $v$ of $E$, let $\fp_v$ be the corresponding maximal ideal of $\fo_{E}$. For any invertible ideal $\fm$ of $\fo_{E}$, define $U(\fm):=\prod_{v<\infty}(1+\fp_{v}^{m_v}\cO_{E,v})\cap\fo_{E,v}^\times$ with $m_v:={\rm ord}_{v}\fm$, which is an open compact subgroup of $\widehat \cO_{E}^\times$. Let $\Lambda=\otimes_{v}\Lambda_v$ be an idele class character on $\A_{E}^\times$ of finite order such that $\Lambda|\A^\times={\bf 1}$; then $\Lambda|E_{\infty}^\times ={\bf 1}$ and $\Lambda$ factors through a finite quotient
 \begin{align}
{\rm Cl}_{\fm}^{0}(E):=\A_{E}^\times/\A^\times E^\times E_\R^{\times}U(\fm)
\label{RayClassGp}
\end{align}
for some invertible ideal $\fm$; the conductor of $\Lambda$ is the maximal one among such ideals $\fm$.  As is well known, the group \eqref{RayClassGp} is isomorphic to a quotient group of the ray class group of conductor $\fm$. 
%with $m_v:=\min\{m\in\Z_{\ge 0}\ ;\ \Lambda_{v}|_{(1+\fp_{v}^{m_v}\fo_{E,v})\cap\fo_{E,v}^\times}={\bf 1}\}$.
We say that a place $v$ is inert in $E/\Q$, splits in $E/\Q$ or ramifies in $E/\Q$ if $E_v:=E\otimes_\Q\Q_p$ is an unramified field extension of $\Q_p$, is isomorphic to $\Q_p\oplus \Q_p$, or is a ramified field extension of $\Q_p$, respectively. We have the following:
\begin{align}
\text{If $v$ is inert in $E/\Q$ and $\Lambda_v|\cO_{E,v}^\times=1$, then $\Lambda_{v}={\bf 1}$}. 
\label{OmegaTriv-2}
\end{align}
Indeed, let $p$ be a prime below $v$; then $ \Lambda_v(p)=1$ due to $\Lambda|\A^\times={\bf 1}$. Hence, $\Lambda_v$ is trivial on $p^{\Z}\cO_{E,v}^\times=E_p^\times$. Define the Galois conjugate $\Lambda^\dagger$ of $\Lambda$ by setting
\begin{align}
\Lambda^\dagger(\tau):=\Lambda(\bar \tau), \quad \tau \in \A_{E}^\times.
 \label{Def-ConjOmega}
\end{align}
We have $\Lambda(\tau)\Lambda^\dagger(\tau)=\Lambda(\tau  \bar \tau)=1$ for $\tau\in \A_E^\times$ by $\Lambda|\A^\times={\bf 1}$. Hence 
\begin{align}
\Lambda^\dagger=\Lambda^{-1}=\bar \Lambda,
\label{OmegaTriv-3}
\end{align} 
where $\bar \Lambda$ is the complex conjugate of $\Lambda$.

For an idele class character $\xi$ of $E^\times$, let $\widehat L(s,\xi)$ be the completed  Hecke $L$-function of $\xi$, and $L_p(s,\xi)$ its local $p$-factor for $p\leq \infty$. Then, $L_p(s,\xi^{\dagger})=L_p(s,\xi)$ for any $p<\infty$. Indeed, if $p$ is not ramified in $E/\Q$, the equality is trivial. Suppose that $p$ ramifies; if $\xi_p$ is a ramified character of $E_p^\times$, then both $L$-factors are $1$. If $\xi_p$ is unramified, then for any prime element $\varpi$ of $E_p$, we have $(\xi^{\dagger})_p(\varpi)=\xi_p( \bar \varpi)=\xi_p(\varpi)$, which implies the equality between local $p$-factors.

 For any character $\mu:\A^\times/\Q^\times \R_{>0}\rightarrow \C^1$, we define $\mu_{E}:=\mu\circ \nr_{E/\Q}:\A^\times_E/E^\times \rightarrow \C^1$. Then, $\mu_{E}$ is Galois invariant, and $\mu_{E}|E_\infty^\times={\bf 1}$. Hence, $(\Lambda\mu_{E})^{\dagger}=\Lambda^{\dagger}\mu_{E}^{\dagger}=\Lambda^{-1}\mu_{E}$, and   
\begin{align}
\widehat L(s,\Lambda^{-1}\mu_{E})=\widehat L(s,\Lambda^\dagger\mu_{E})=\widehat L(s,\Lambda\mu_{E}), \quad (\Re(s)>1). 
\label{FE-omega}
\end{align}
Note that $\widehat L(s,\Lambda\mu_{E})$ is holomorphic except for possible simple poles at $s=1,0$, which occurs if and only if both $\Lambda$ and $\mu$ are trivial.

\section{Eisenstein series and Rankin-Selberg integral} \label{sec:EisSerRSint}

\subsection{Eisenstein series}\label{sec:RSint}
For details, we refer to \cite[\S19]{Jacquet}; the theory on ${\bf GL}_2(\A_{E})$ developed there carries over into the group $\sG^\#(\A)$ by a minor modification. For a finite-dimensional vector space $V$ over a local field, let $\mathcal{S}(V)$ be the space of all Schwartz-Bruhat functions on $V$. For $p\le\infty$, $\phi\in\mathcal{S}(E_p^2)$ and $s\in\mathbb{C}$ with $\Re(s)>0$, we define a function $f_\phi^{(s,\Lambda_p,\mu_p)}:\sG^\#(\Q_p)\rightarrow \C$ by 
\begin{align}
f_\phi^{(s,\Lambda_p,\mu_p)}(g^\#)=\mu_p(\det g^\#)|\det g^\#|_p^{s+1}\int_{E_p^\times}\phi([\begin{smallmatrix} 0 & 1 \end{smallmatrix}][\begin{smallmatrix} \tau & 0 \\ 0 & \bar \tau \end{smallmatrix}]g^\#) \Lambda_p\mu_{E,p}(\tau) |\tau \bar \tau|_p^{s+1} d\eta_{E_p^\times}(\tau).
\label{RSint-f0}
\end{align}
When $p=\infty$, we assume that $\phi$ is $\bK_\infty^\#$-finite. By the local Tate's theory, the function $s\mapsto f_\phi^{(s,\Lambda_p,\mu_p)}(g^\#)$ is continued meromorphically to $\C$ in such a way that
\begin{align}
\sf_\phi^{(s,\Lambda_p,\mu_p)}(g^\#):=L_p(s+1,\Lambda_p\mu_{E,p})^{-1}\times f_\phi^{(s,\Lambda_p,\mu_p)}(g^\#)
\label{f-sf-rel}
\end{align}
is holomorphic on $\C$. Then $\sf_\phi^{(s,\Lambda_p,\mu_p)}$ belongs to the space $\mathscr{V}^\#(s,\Lambda_p,\mu_p)$ consisting of all smooth functions $f$ on $\sG^\#(\Q_p)$ satisfying
$$
f([\begin{smallmatrix} \tau & \beta \\ 0 & a\bar \tau \end{smallmatrix}]g^\#)=\Lambda_p(\tau)^{-1}\mu_p(a)^{-1}|a|_p^{-(s+1)}f(g^\#)
$$
for any $[\begin{smallmatrix} \tau & \beta \\ 0 & a\bar \tau \end{smallmatrix}]\in\sB^\#(\Q_p)$ and $g^\#\in \sG^\#(\Q_p)$.

The Fourier transform $\widehat{\phi}$ of $\phi\in\mathcal{S}(E_p^2)$ is defined by
\begin{align}\widehat{\phi}(x,y)=\int_{E_p^2}\phi(u,v)\psi(\langle[\begin{smallmatrix}  x  \\  y  \end{smallmatrix}],[\begin{smallmatrix}  u  \\  v  \end{smallmatrix}]\rangle)d\eta_{E_p}(u)d\eta_{E_p}(v).
\label{Def-Local-FT}
\end{align}
Set $w_0:=\left[\begin{smallmatrix}0 & -1\\ 1& 0\end{smallmatrix}\right] \in \sG^\#(\Q_p)$. The standard intertwining operator 
$$
M(s):{\mathscr V}^\#(s,\Lambda_p,\mu_p)\longrightarrow {\mathscr V}^\#(-s,\Lambda_p^{-1},\mu_p^{-1})
$$
is defined on $\Re(s)>0$ as the absolutely convergent integral 
\begin{align*}
 M(s)f(g^\#)=\int_{E_p}f\left( 
w_0  
\left[\begin{smallmatrix}1 & \beta \\ 0& 1 \end{smallmatrix}\right]g^\#\right)\,\d\eta_{E_p}(\beta), \quad g^\#\in \sG^\#(\Q_p).  
\end{align*}
The effect of $M(s)$ on the section $f_{\phi}^{(s,\Lambda_p,\mu_p)}$ is described  by the Fourier transform $\widehat \phi$ as 
\begin{align}
f_{\widehat \phi}^{(-s,\Lambda_p^{-1},\mu_p^{-1})}(g^\#)
=c_p\Lambda_p(\sqrt{D})^{-1}|D|_p^{-s+1/2}\,
\varepsilon_p(s,\Lambda\mu_{E},\psi_{E_p})\frac{L_p(1-s,\Lambda^{-1}\mu_{E}^{-1})}{L_p(s,\Lambda\mu_{E})}\times M(s)f_{\phi}^{(s,\Lambda_p,\mu_p)}(g^\#)
\label{JaqS-f1}
\end{align}
with $c_p=1$ if $p<\infty$ and $c_\infty=\frac{\sqrt{|D|}}{2}$ if $p=\infty$, where $\varepsilon(s,\bullet,\psi_{E_p})$ denote Tate's local epsilon factor defined by the character $\psi_{E_p}$ of $E_p$ and the associated self-dual measure on $E_p$ as usual. Let $\mathcal{S}(\A_E^2)$ be the space of all Schwartz-Buruhat functions on $\A_E^2$. For a decomposable element $\phi=\otimes_{p\le\infty}\phi_p\in\mathcal{S}(\A_{E}^2)$, we define
$$\sf_\phi^{(s,\Lambda,\mu)}(g^\#)=\prod_{p\le\infty}\sf_{\phi_p}^{(s,\Lambda_p,\mu_p)}(g_p^\#), \quad g^\#=(g_p^\#)_p\in \sG^\#(\A).$$
Note that $\sf_\phi^{(s,\Lambda,\mu)}$ is left $\sB^\#(\Q)$- invariant and right $\bK^\#$-finite. The Eisenstein series attached to $\sf_\phi^{(s,\Lambda,\mu)}$ is defined by the absolutely convergent series
\begin{align}
E(\phi,s,\Lambda,\mu;g^\#):=\widehat{L}(s+1,\Lambda\mu_E)\times\sum_{\delta \in \sB^\#(\Q)\bsl \sG^\#(\Q)}\sf_\phi^{(s,\Lambda,\mu)}(\delta g^\#), \quad g^\#\in \sG^\#(\A)
 \label{Def-Eis}
\end{align}
for $\Re(s)>1$. The Fourier transform $\widehat \phi$ of $\phi \in {\mathcal S}(\A_{E}^2)$ is defined by a formula similar to \eqref{Def-Local-FT} with respect to the measure $\eta_{\A_E}\otimes \eta_{\A_E}$. The following property of the Eisenstein series is standard.

\begin{prop}\label{RSBasicL0}
Let $\phi\in\mathcal{S}(\A_E^2)$, $s\in \C$, $\Lambda\in \widehat{{\rm Cl}({E})}$, $\mu\in\widehat{\A^\times/\Q^\times\R^\times}$ and $g^\#\in \sG^{\#}(\A)$.

\begin{itemize}
\item[(i)] The map $s\mapsto E(\phi,s,\Lambda,\mu;g^\#)$ $(\Re(s)>1)$ has a meromorphic continuation to $\C$, holomorphic in $s$ unless $\Lambda\mu_E\not={\bf 1}$ in which case it has possible simple poles only at $s=1,-1$. The function $g^\#\mapsto E(\phi,s,\Lambda,\mu;g^\#)$ is an automorphic form on $\sG^\#(\Q)\sZ^\#(\A)\bsl \sG^\#(\A)$. 
\item[(ii)] We have the functional equation $E(\widehat{\phi},-s,\Lambda^{-1},\mu^{-1};g^\#)=E(\phi,s,\Lambda,\mu;g^\#)$. 
\item[(iii)] For a relatively compact subset $\cN\subset \C$ on which $s\mapsto E(\phi,s,\Lambda,\mu;g^\#)$ is regular and for a compact set $\cU \subset \sG^\#(\A)$, there exists constants $C>0$ and $N>0$ such that
$$
|E(\phi,s,\Lambda,\mu;\left[\begin{smallmatrix} a & 0 \\ 0 & 1 \end{smallmatrix}\right]\,g^\#)|\leq C\,|a|^{N}\quad (a\in \R_{>1},\,g^\#\in \cU,\,s\in \cN). 
$$
\end{itemize}
\end{prop}

\subsection{Rankin-Selberg integral and the basic identity}
For any cusp form $\varphi$ on $\sG(\A)$, $\Lambda\in \widehat{{\rm Cl}(E)}$ and $s\in \C$, the Rankin-Selberg integral is defined by 
\begin{align}
\langle E(\phi,s,\Lambda,\mu), \varphi\rangle:=\int_{\sZ^\#(\A)\sG^\#(\Q)\bsl \sG^\#(\A)} E(\phi,s,\Lambda,\mu;g^\#)\,\varphi(\iota_{\theta}(g^\#))\,\d g^\# .\label{Def-RSint}
\end{align}
By Proposition \ref{RSBasicL0}(iii) and by the Fourier expansion \eqref{AdeleFE-L1-f0}, it is straightforward to show the absolute convergence of the integral for $s\in \C$ where the Eisenstein series is regular. For $T\in \sV(\Q)$, define a character of $\sN(\A)$ by 
$$
\psi_{T}(\sn(X))=\psi(\tr(T X)), \quad X\in \sV(\A).
$$
The $(T_\theta,\Lambda)$-Bessel period of a cusp form $\varphi$ on $\sZ(\A)\sG(\Q)\bsl \sG(\A)$ is defined by the integral\footnote{
Note that $\sm(I_\theta(a),\nr_{E/\Q}(a))=a1_4$ for $a\in \A^\times$. Then, due to $\varphi$ being $\sZ(\A)$-invariant, the integral $B^{T_\theta,\Lambda}(\varphi;g)$ is $0$ if $\Lambda|\A^\times\not=1$. This trivial vanishing is not the case by $\Lambda|\A^\times={\bf 1}$.}
\begin{align}
B^{T_\theta,\Lambda}(\varphi;g):=\int_{\A_E^\times/E^\times \A^\times} \Lambda(\tau)^{-1}\,\int_{\sN(\Q)\bsl \sN(\A)} \psi_{T_\theta}(n)^{-1}\varphi(\sm(I_\theta(\tau),\nr_{E/\Q}(\tau))\,n\,g)\,\d n\,\d^\times \tau
 \label{BesselInt}
\end{align}  
for $g\in \sG(\A)$, where $\d^\times \tau$ denote the quotient measure on $\A_{E}^\times/\A^\times$ (see \S\ref{sec:HaarEQ}).
 Here is the basic identity
 %which is originally due to Andrianov in the classical setting (\cite{And1}, \cite{And2}) and is adelically reformulated by 
 due to Piatetski-Shapiro (\cite{PS97}). To determine constants exactly under our normalization of Haar measures, we reproduce the proof briefly. 

\begin{lem}  \label{RSint-L1}
Let $\varphi:\sZ(\A)\bsl \sG(\Q) \sG(\A)\longrightarrow \C$ be a cusp form. For $\Re(s)>1$, we have
\begin{align}
\langle E(\phi,s,\Lambda,\mu),\varphi\rangle&=\tfrac{\sqrt{|D|}}{2}L(s+1,\Lambda\mu_E)\label{RSint-L1-f00}\\
&\times\int_{\A^\times} \int_{\bK_\infty^\#\,\bK_0^\#(\cO_E)} \sf_{\phi}^{(s,\Lambda, \mu)}(k^\#)\mu(a)|a|_\A^{s-1}\,B^{T_\theta, \Lambda}(\varphi;\sm(a 1_2,a)\iota_\theta(k^\#))\,\d^\times a\,\d k^\#.\notag
\end{align}
\end{lem}
\begin{proof} By substituting \eqref{Def-Eis} to \eqref{Def-RSint}, 
\begin{align}
&L(s+1,\Lambda\mu_E)^{-1}\times \langle E(\phi,s,\Lambda,\mu),\varphi\rangle\notag\\
&=\int_{\sZ^\#(\A)\sG^\#(\Q)\bsl \sG^\#(\A)} \sum_{\gamma\in \sB^\#(\Q)\bsl \sG^\#(\Q)} \sf_{\phi}^{(s,\Lambda,\mu)}(\gamma g^\#)\,\varphi(\iota_{\theta}(g^\#))\,\d g^\#
 \notag
\\
&=\int_{\sZ^\#(\A) \sB^\#(\Q) \bsl \sG^\#(\A)} \sf_{\phi}^{(s,\Lambda,\mu)}(g^\#)\,
\varphi(\iota_{\theta}(g^\#))\,\d g^\# 
\notag
\\
&=\tfrac{\sqrt{|D|}}{2} \int_{\A^\times/\Q^\times}\int_{\A_E^\times/\A^\times E^\times} \int_{\sN^\#(\Q)\bsl \sN^\#(\A)} \int_{\bK_0^\#(\cO_E)\bK_\infty^\#}|a|_\A^{-(s+1)}\Lambda(\tau)^{-1}\mu(a)^{-1} \notag
\\
&\quad \times \varphi(\iota_\theta(n^\#)\,\sm(I_{\theta}(\tau),\nr_{E/\Q}(\tau))\,\left[\begin{smallmatrix} a^{-1} 1_2 & 0 \\ 0 & 1_2 \end{smallmatrix} \right]\iota_{\theta}(k_\fin^\#))
\,|a|_\A^{2}\,\d^\times a\,\d^\times \tau\,\d n^\#\,\d k_\fin^{\#}\,\d k_\infty^\#, 
\notag
\end{align}
where the last equality is proved by \eqref{HaarMeasSGP-Adele}. By Lemma~\ref{RSint-L0}, the last expression becomes
\begin{align*}
&\tfrac{\sqrt{|D|}}{2}\int_{\A^\times/\Q^\times}\int_{\A_E^\times/\A^\times E^\times} \int_{\sN^\#(\Q)\bsl \sN^\#(\A)} \int_{\bK_0^\#(\cO_E)\bK_\infty^\#}|a|_\A^{-(s+1)}\Lambda(\tau)^{-1}\mu(a)^{-1} \\
&\quad \times \sum_{\alpha \in \Q^\times} \varphi(n\,\sm(I_{\theta}(\tau),\nr_{E/\Q}(\tau))\,\left[\begin{smallmatrix} \alpha a^{-1} 1_2 & 0 \\ 0 & 1_2 \end{smallmatrix} \right]\iota_{\theta}(k_\fin^\#))\,\psi_{T_\theta}(n)^{-1}
\,|a|_\A^{2}\,\d^\times a\,\d^\times \tau\,\d n\,\d k_\fin^{\#}.
\end{align*}
The $\alpha$-summation and the $a$-integral over $\A^\times/\Q^\times$ are combined to yields an integral over $\A^\times$. Then, by the change of variables $a\mapsto a^{-1}$, we get the desired formula by \eqref{BesselInt}. \end{proof}

\begin{lem} \label{RSint-L0} 
\begin{align*}
\int_{\sN^\#(\Q)\bsl \sN^\#(\A)}\varphi(\iota_\theta(n^\#)g)\,\d n^\#
=\sum_{\alpha\in \Q^\times}
\int_{\sV(\Q)\bsl \sV(\A)}\varphi(\sn(X) \left[\begin{smallmatrix} \alpha 1_2 & 0 \\ 0 & 1_2 \end{smallmatrix}\right]g)\,\psi_{T_\theta}(\sn(X))^{-1}\,\d X,
\quad g\in \sG(\A).
\end{align*}
\end{lem}
\begin{proof}
 The proof is straightforward.
\end{proof}

\section{Automorphic forms and Fourier coefficients} \label{sec:FC}
Let $(\Z^2)_{\rm dom}:=\{\lambda=(l_1,l_2) \in \Z^2\mid l_1\geq l_2\,\}$. Let $\lambda=(l_1,l_2) \in (\Z^2)_{\rm dom}$ and $\varrho$ the representation of ${\bf GL}_2(\C)$ on the space $V_{\varrho}$ of homogeneous polynomials in $X,Y$ of degree $l_1-l_2$ defined as $\varrho(h)f(X,Y)=(\det h)^{l_2}f(aX +b Y, cX+dY)$ for $h=\left[\begin{smallmatrix} a & b \\c & d \end{smallmatrix}\right]\in {\bf GL}_2(\C)$ and $f \in V_\varrho$. As is well-known, any irreducible rational representation of ${\bf GL}_2(\C)$, up to equivalence, is obtained this way. The space $V_\varrho$ carries an ${\bf SU}(2)$-invariant hermitian inner product product $(\cdot\mid\cdot)_\varrho$ given by \cite[((8.2.6)]{FurusawaMorimoto}. Recall $\bK_\infty=\{k_\infty(u):=\left[\begin{smallmatrix} A & B \\ -B & A 
\end{smallmatrix}\right]\mid u=A+iB\in {\bf U}(2)\}$. For $N\in \Z_{>0}$, let $S_{\varrho}(\bK_0(N))$ denotes the space of smooth functions $\varphi:\sG(\A)\rightarrow V_{\varrho}$ that satisfies the following conditions: 
\begin{itemize}
\item[(i)] $\varphi(z \gamma g)=\varphi(g)$ for $(z,\gamma,g)\in \sZ(\A)\times \sG(\Q)\times \sG(\A)$.
\item[(ii)] $\varphi(gk_\fin k_\infty(u))=\varrho(\bar u)^{-1}\varphi(g)$ for $k_\fin\in \bK_0(N)$ and $k_\infty(u) \in \bK_\infty$. 
\item[(iii)] $R(X)\varphi=0$ for all $X\in \fp^{-}$. 
\item[(iv)] $\varphi$ is bounded on $\sG(\A)$. 
\end{itemize}
For $T\in \sV(\R)$ with $\det(T)\not=0$, let ${\bf B}_{\varrho}^{T}: \sG(\R)^0\longrightarrow {\rm End}_{\C}(V_{\varrho})$ be a function defined by 
\begin{align}
{\bf B}_{\lambda}^T(g):=\nu(g)^{(l_1+l_2)/2}\varrho(Ci+D)^{-1} \exp(2\pi i \tr(T\,g\langle i1_2\rangle), \quad g=\left[\begin{smallmatrix} A & B \\ C & D \end{smallmatrix}\right] \in \sG(\R)^0.
 \label{AdeleFE-W}
\end{align}
This function satisfies the conditions: \begin{align}
{\bf B}^{T}_{\varrho}(\sm(A,c)g)&={\bf B}^T_\varrho(g)\circ \varrho({}^t A^{-1}\,c)^{-1}, \quad \sm(A,c)\in {\sM}_T^\circ(\R),\, g\in \sG(\R)^0 
 \label{GWF-f0} \\
{\bf B}^{T}_\varrho(ngk_\infty(u))&=\psi_T(n)\varrho(\bar u)^{-l}\circ {\bf B}(g)
, \quad (n,g,u)\in \sN(\R)\times \sG(\R)^0\times {\bf U}(2), 
 \label{GWF-f1}
\\
R(\fp^{-}){\bf B}^{T}_\varrho&=0.
 \label{GWF-f2}
\end{align}
The function $B^{T}_{\varrho,v}: g\longmapsto {\bf B}_{\varrho}^{T}(g)(v)$ with $v\in V_{\varrho}-(0)$ is bounded on $\sG(\R)^0$ if and only if $T\in \sV(\R)^{+}$, where $\sV(\R)^+$ denote the set of positive definite elements in $\sV(\R)$. Set 
\begin{align*}
\gamma&:=\sm(-1_2,-1)\in \sM(\Q). 
\end{align*}
Note that $\nu(\gamma)=-1$.  
\begin{lem}\label{AdeleFE-L1}
Let $\varphi\in S_\varrho(\bK_0(N))$. There exists a unique family of vectors $a_\varphi(T;g_\fin)\in V_\varrho$ $(T\in \sV(\Q))$, $g_\fin \in \sG(\A_\fin)$ such that 
\begin{align}
\varphi(g_\infty g_\fin)=\sum_{T \in \sV(\Q)}{\bf B}_{\varrho}^{T}(g_\infty)(a_\varphi(T;g_\fin))\,
, \quad g_\infty\in \sG(\R)^\circ,\,g_\fin\in \sG(\A_\fin).
 \label{AdeleFE-L1-f0}
\end{align}
If $T\not \in \sV(\Q)^{+}$, then $a_\varphi(T;g_\fin)=0$ for all $g_\fin \in \sG(\A_\fin)$. For $g_\fin \in \sG(\A_\fin)$, 
\begin{align}
\int_{\sN(\A)/\sN(\Q)}\varphi(ng_\infty g_\fin)\,\psi_{T}(n)^{-1}\d n=
\begin{cases}{\bf B}_{\varrho}^{T}(g_\infty)(a_\varphi(T;g_\fin)), 
\quad &(g_\infty\in \sG(\R)^\circ), \\
0 ,\quad &(g_\infty \in \sG(\R)-\sG(\R)^\circ).
\end{cases}
 \label{AdeleFE-L1-1}
\end{align}
\end{lem}
\begin{proof}
Fix $g_\fin \in \sG(\A_\fin)$ and examine the integral, say $W(g_\infty)$, on the left-hand side of \eqref{AdeleFE-L1-1}. The function $W$ on $\sG(\R)^0$ satisfies conditions \eqref{GWF-f1} and \eqref{GWF-f2}. As such, there corresponds a function $F:\fh_2\longrightarrow V_\varrho$ determined by the relation $F(g_\infty\langle i1_2\rangle)=\nu(g)^{-(l_1+l_2)/2}\varrho(Ci+D)\,W(g_\infty)$ with $g_\infty=\left[\begin{smallmatrix} A & B \\ C & D \end{smallmatrix}\right]\in \sG(\R)^0$ such that $Z=(Ai+B)(Ci+D)^{-1}$. Since $F(Z)$ satisfies $F(Z+X)=\psi_{T}(X)\,F(Z)$ ($X\in \sV(\R))$, it is of the form $F(Z)=F_1(Y)\exp(2\pi i \tr(TX))$ with a $V_\varrho$-valued $C^\infty$-function $F_1(Y)$. Since $F(Z)$ is holomorphic, the Cauchy-Riemann equations yield $\frac{\d}{\d y_{ij}}F_1(Y)=-2\pi t_{ij}F_1(T)$, which is uniquely solved as $$F_1(Y)=a_{\varphi}(T;g_\fin)\exp(-2\pi \tr(TY))$$ with a vector $a_\varphi(T;g_\fin)\in V_\varrho$, or equivalently $W(g_\infty)={\bf B}_\varrho^{T}(g_\infty)(a_\varphi(T;g_\fin))$. Then, the formula in \eqref{AdeleFE-L1-f0} is a consequence of the Fourier expansion of a $\sN(\Q)$-periodic function on $\sN(\A)$. Since $\varphi$ is bounded on $\sG(\A)$, the function $W(g_\infty)$ should be bounded on $\sG(\R)^0$; since $g_\infty \mapsto {\bf B}_{\varrho}^{T}(g_\infty)(v)$ with $v\in V_\varrho-\{0\}$ is unbounded if $T\not\in \sV(\R)^+$, we have $a_{\varphi}(T;g_\fin)=0$ for all $T\in \sV(\Q)-\sV(\R)^+$. It remains to show the second case in \eqref{AdeleFE-L1-1}. Let $g_\infty\not\in \sG(\R)^\circ$. Decompose $\gamma\in \sM(\Q)$ as $\gamma_{\infty}\gamma_{\fin}$ $(\gamma_\infty\in \sM(\R),\,\gamma_\fin\in \sM(\A_\fin))$ in $\sM(\A)$. By the left $\sG(\Q)$-invariance of $\varphi$, 
$$\varphi(\sn(X)g_\infty g_\fin)=\varphi(\gamma\sn(X)g_\infty g_\fin)=
\varphi(\sn(-X)\gamma_\infty g_\infty \gamma_\fin g_\fin).
$$
Integrate this in $X$ over $\sV(\Q)\bsl \sV(\A)$ and make a change of variables $n\rightarrow n^{-1}$ on the way; we have
\begin{align*}
\int_{\sN(\Q)\bsl \sN(\A)}\varphi(n g_\infty g_\fin)\,\psi_{T}(n)^{-1}\d n
&=\int_{\sN(\Q)\bsl \sN(\A)}\varphi(n\gamma_\infty g_\infty \gamma_\fin g_\fin)\,\psi_{-T}(n)^{-1}\,\d n.
\end{align*}
Since $\gamma_\infty g_\infty \in \sG(\R)^\circ$, the last integral becomes $
{\bf B}_\varrho^{-T}(\gamma_\infty g_\infty)(a_\varphi(-T;\gamma_\fin g_\fin))
$ by the first case. Since $-T\not\in \sV(\R)^+$, we have $a(-T;\gamma_\fin g_\fin)=0$.  \end{proof}
The vectors $a_\varphi(T;g_\fin)$ in $V_\varrho$ are referred to as the adelic Fourier coefficients of $\varphi$. As is well-known, there is a linear bijective correspondence between $\varphi \in S_{\varrho}(\bK_0(N))$ and the classical $V_\varrho$-valued Siegel cusp forms $f(Z)$ on ${\bf \Gamma}_0(N):=\{\left[\begin{smallmatrix} A & B \\ C & D\end{smallmatrix}\right]\in {\bf Sp}_2(\Z)\mid C\equiv 0 \pmod{N}\}$ determined by the relation 
$$
f(g_\infty \langle i1_2\rangle)=\nu(g_\infty)^{-(l_1+l_2)/2}\varrho(Ci+D)\varphi(g_\infty), \quad g_\infty= 
\left[\begin{smallmatrix} A & B \\ C & D \end{smallmatrix}\right]\in \sG(\R)^0.
$$
By the modular transformation law $f((AZ+B)(CZ+D)^{-1})=\varrho(CZ+D)^{-1}f(Z)$ for $\left[\begin{smallmatrix} A & B \\ C & D\end{smallmatrix}\right]\in \Gamma_0(N)$, since $-1_4\in \Gamma_0(N)$, we have that $f(Z)$ is identically $0$, hence $S_{\varrho}(\bK_0(N))=(0)$, unless $l_1\equiv l_2\pmod{2}$. (For more details, see \cite[\S3.2]{Schmidt} and \cite[\S4]{AsgariSchmidt}.) 

Since $\sN(\A)\cap \bK_0(N)=\{\sn(X)\mid X\in \sV(\widehat \Z)\}$, by \eqref{AdeleFE-L1-1}, we have $a_{\varphi}(T;1)=0$ unless $X\in \cQ$, where
$$
\cQ:=\Bigl\{\left[\begin{smallmatrix} a & b/2 \\ b/2 & c \end{smallmatrix}\right]\, \Bigm|\, a,b,c\in \Z \Bigr\}
$$
is the dual lattice of $\sV(\Z)$. Then, \eqref{AdeleFE-L1-f0} with $g_\fin=1_4$ reduces to 
\begin{align*}
f(Z)=\sum_{T\in \cQ^{+}}a_\varphi(T;1_4)\,\exp(2\pi i \tr(TZ)), \quad Z\in \fh_2,
\end{align*}
where $\cQ^+:=\cQ\cap \sV(\Q)^+$. This means $a_{\varphi}(T):=a_{\varphi}(T;1_4)$ $(T\in \cQ^+)$ is the Fourier coefficient of $f(Z)$ in the classical sense. Set
\begin{align*}
\delta:=\left[\begin{smallmatrix} 1 & 0 \\ 0 & -1 \end{smallmatrix}\right] \in {\bf GL}_2(\Q).
\end{align*}
Note that the map $T\mapsto -T\ss(\delta)$ preserves the set $\cQ^+$. If we identify $\cQ^+$ with the set of positive definite integral binary forms $aX^2+bXY+cY^2$, this operation corresponds to the sign change of $b$. 
\begin{lem}\label{AdeleFE-L2}
The set of adelic Fourier coefficients of $\varphi \in S_\varrho(\bK_0(N))$ has the following properties:  
\begin{align}
&a_{\varphi}(T ; \sm(h_\fin,\det h_\fin) g_\fin \kappa)=a_{\varphi}(T\ss(h);g_\fin), \quad h \in {\bf GL}_2(\Q), \, g_\fin\in \sG(\A_\fin),\,\kappa \in \bK_0(N).
\label{AdeleFE-L2-1}
\\
&a_{\varphi}(-T\ss\left(\delta\right);1)=\varrho\left(\delta\right)(a_\varphi(T;1)),\quad T\in \sV(\Q). 
\label{AdeleFE-L2-2}
\end{align}
\end{lem}
\begin{proof}
The relation in \eqref{AdeleFE-L1-1} follows from \eqref{AdeleFE-L2-1} by a simple change of variables. Let us show \eqref{AdeleFE-L2-2}. We have
\begin{align}
\sn\left(-X\ss\left(\delta\right)\right)=\kappa \sn(X)\kappa^{-1}, \quad X\in \sV(\A)
 \label{BesselConj-1}
\end{align}
with $\kappa:=\sm(\delta,1)\in \sG(\Q)$. Write $\kappa=\kappa_\fin \kappa_\infty\in \sG(\A_\fin)\,\sG(\R)$; then, $\kappa_\fin \in \bK_0(N)$ and $\kappa_\infty \in \bK_{\infty}$. We have  
{\allowdisplaybreaks
 \begin{align*}
\int_{\sN(\Q)\bsl \sN(\A)}\varphi(n)\,\psi_{-T\ss\left(\delta\right)}(n)^{-1}\,\d n
&=\int_{\sV(\Q)\bsl \sV(\A)}\varphi(\sn(-\ss \left(\delta \right)X))\,\psi(
\tr(TX))^{-1}\,\d X \\
&=\int_{\sV(\Q)\bsl \sV(\A)}\varphi(\kappa\sn(X)\kappa^{-1})\,\psi(\tr(TX))^{-1}\,\d X
\\
&=\varrho(\delta)\,\int_{\sN(\Q)\bsl \sN(\A)}\varphi(n)\psi_{T}(n)^{-1}\,\d n.
\end{align*}}
From this and \eqref{AdeleFE-L1-1}, 
$$
{\bf B}_{\varrho}^{-T\ss\left(\delta
\right)}(1_4)(a_{\varphi}(-T\ss\left(\delta\right);1))=\varrho(\delta)\,{\bf B}_{\varrho}^{T}(1_4)(a_{\varphi}(T;1)).
$$
By \eqref{AdeleFE-W}, ${\bf B}_{\varrho}^{-T\ss\left(\delta\right)}(1_4)
={\bf B}_{\varrho}^{T}(1_4)=e^{-2\pi \tr(T)}{\rm Id}_{V_\varrho}$.    \end{proof}

For a positive integer $\cM$, define 
\begin{align}
\cQ^{+}_{\cM}(D):=\{T=\left[\begin{smallmatrix} a & 2^{-1}b \\ 2^{-1}b & c\end{smallmatrix}\right]\in \cQ \mid a>0,\, {\gcd}(a,\cM)=1, \,-\det T=D/4\},
\label{cQplusD}
\end{align} 
${\bf \Gamma}_0(\cM\Z_p):={\bf SL}_2(\Z_p)\cap \left[\begin{smallmatrix} \Z_p & \Z_p \\ {\cM}\Z_p & \Z_p \end{smallmatrix}\right]$, 
${\bf \Gamma}_0(\cM\widehat{\Z}):=\prod_{p}{\bf \Gamma}_0(\cM\Z_p)$,  and ${\Gamma}_0(\cM):={\bf SL}_2(\Q)\cap {\bf \Gamma}_0(\cM\widehat \Z)$. Since $D$ is a fundamental, $\cQ^{+}_{\cM}(D)$ entirely consists of primitive forms, i.e., ${\rm gcd}(a,b,c)=1$; $\cQ^{+}_{\cM}(D) \subset \sV(\Q)$ is ${\bf \Gamma}_0(\cM)$-stable by the action \eqref{GLaction}; the ${\Gamma}_0(\cM)$-orbit of $T\in \cQ^{+}_{\cM}(D)$ is denoted by $[T]$. Note that $T_\theta\in \cQ_{\cM}^{+}(D)$ (see \eqref{Def-Ttheta}). Recall the notation in \S\ref{sec:RyClCh}. When all the prime divisors of $\cM$ are inert in $E/\Q$, we write $U(\cM):=U(\cM\fo_E)$ and ${\rm Cl}_{\cM}^{0}:={\rm Cl}_{\cM\fo_E}^{0}$ briefly. Let $\widehat{{\rm Cl}_{\cM}^{0}(E)}_{{\rm prim}}$ be the set of all characters $\Lambda$ on ${\rm Cl}_{\cM}^{0}(E)$ with conductor $\cM\fo_{E}$

\begin{lem} \label{BessePerL2} Let $\cM$ be a positive integer with prime factors being inert in $E/\Q$ and $\Lambda\in\widehat{{\rm Cl}_{\cM}^{0}(E)}_{{\rm prim}}$.
For each $u\in \A_E^\times$, let $I_{\theta}(u)$ decompose in ${\bf GL}_2(\A)$ as
\begin{align}
I_\theta (u)=\gamma_{u}h_u\kappa_u, \quad \gamma_u\in {\bf GL}_2(\Q),\,
h_{u}\in {\bf GL}_2(\R)^\circ,\,\kappa_u=(\kappa_{u,p})_{p<\infty}\in {\bf \Gamma}_0(\cM\widehat{\Z})
 \label{BesselPerL2-2}
\end{align}
and set $T_{\theta}(u):=T_\theta\,\ss(\gamma_u)$. Then ,
\begin{itemize}
\item[(i)] $T_{\theta}(u)\in \cQ_{\cM}^{+}(D)$. 
\item[(ii)] The ${\Gamma}_0(\cM)$-equivalence class of $T_{\theta}(u)$ does not depend on the decomposition \eqref{BesselPerL2-2}. If $u$, $u'$ belongs to the same $\A^\times E^\times E^\times_\infty U(\cM)$-coset, then $T_{\theta}(u)$ and $T_\theta(u')$ are ${\bf \Gamma}_0(\cM)$-equivalent. 
\item[(iii)] The map $[u]\mapsto [T_\theta(u)]$ from ${\rm Cl}_{\cM}^{0}(E)$ to $\cQ_{\cM}^{+}(D)/{\Gamma}_0(\cM)$ is a bijection. 
\end{itemize}
\end{lem}
\begin{proof} 
Since (i) and (ii) are proved directly, we only give some detail of (iii). 
\begin{comment}
(i) The equation $\det T_\theta(u)=-D/4$ is obvious. Since $T_{\theta}\ss(I_{\theta}(u))=T_\theta$, we get $(T_\theta(u))_{\fin}=T_\theta \ss(\kappa_{u}^{-1})$ and $(T_\theta(u))_{\infty}=T_\theta \ss(h_{u}^{-1})$. It can be easily checked that the group action of ${\bf GL}_2(\widehat \Z)$ on $\sV(\A_{\fin})$ preserves primitive elements in $\sV(2^{-1}\widehat \Z)$. Hence, $T_\theta(u)$ is in $\sV(2^{-1}\widehat \Z)\cap \sV(\Q)=\sV(2^{-1}\Z)$ and is primitive. Comparing the archimedean components, we get that the upper left entry of $T_\theta(u)$ is equal to 
$$\det(h_v)\cdot{}^t\left[\begin{smallmatrix} h_{u,22} \\  -h_{u,21} \end{smallmatrix}\right]T_\theta \left[\begin{smallmatrix} h_{u,22} \\  -h_{u,21} \end{smallmatrix}\right],$$
which is positive due to the positive definiteness of $T_\theta$ and $\det h_u>0$.

Suppose that $I_\theta(u)={\widetilde \gamma_u}{\widetilde h_u}{\widetilde \kappa_u}$ is another deconposition as in \eqref{BesselPerL2-2}. Then we have ${\widetilde \gamma_u}^{-1}\gamma_u={\widetilde h_u}h_u^{-1}{\widetilde \kappa_u}\kappa_u^{-1}$. Comparing both the finite part and the infinite part respectively, we deduce that ${\widetilde \gamma_u}^{-1}\gamma_u\in{\bf \Gamma}_0(\cM\widehat \Z)\cap{\bf GL}_2(\Q)^{+}={\bf \Gamma}_0(\cM)$. Therefore, the two images are ${\bf \Gamma}_0(\cM)$-equivalent. The second statement is immediate from the facts that $I_\theta(E_\infty^\times)\subset {\bf GL}_2(\R)^{\circ}$ and $I_\theta(\A_{\fin}^\times U(\cM))\subset {\bf \Gamma}_0(\cM\widehat \Z)$.
\end{comment}
For $u,u'\in\A_E^\times$, let $I_\theta(u)=\gamma_u h_u \kappa_u$ and $I_\theta(u')=\gamma_{u'} h_{u'} \kappa_{u'}$ be as in \eqref{BesselPerL2-2}. If $T_\theta(u)$ and $T_\theta(u')$ are ${\bf \Gamma}_0(\cM)$-equivalent, by Lemma \ref{acci.isom.}, there exist $\delta_0\in {\bm\Gamma}_0(\cM)$ and $r\in E^\times$ such tha $\gamma_{u'}=I_\theta(r)\gamma_u\delta_0$. From this and noting that $T_\theta=T_\theta\ss(I_\theta(u))=T_\theta\ss(I_\theta(u'))$ and $I_\theta^{-1}({\bf \Gamma}_0(\cM\widehat \Z))\subset U(\cM)$, we find that there exist $c_\fin\in U(\cM)$, $c_\infty\in E_\infty^\times$ such that $\kappa_{u'}=\delta_{0,\fin}^{-1}\kappa_u I_\theta(c_\fin)$, $h_{u'}=\delta_{0,\infty}^{-1} h_u I_\theta(c_\infty)$, which yields
$$I_\theta(u')=I_\theta(r)\gamma_u h_u \kappa_u I_\theta(c_\infty c_\fin)=I_\theta(urc_\fin c_\infty).$$
Since $I_\theta$ is injective, we conclude that $u'$ $u$ are $\A^\times E^\times E_\infty^\times U(\cM)$-equivalent, completing the proof of injectivity of the map. To show the surjectivity, take an element $T\in\cQ_{\cM }^{+}(D)$. By Witt's theorem and Lemma \ref{acci.isom.}, 
$T=T_\theta\ss(\gamma)$ for some $\gamma \in {\bf GL}_2(\Q)$. On the other hand, for each place $p$, we can prove that there exists an element $g_p\in{\bm \GL}_2(\Z_p)$ (resp. $g_p\in{\bm \Gamma}_0(\cM\Z_p)$) for $p\nmid\cM$ (resp. $p\mid \cM$) such that $T=T_\theta \ss(g_p)$ in ${\bf GL}_2(\Q_p)$. (Indeed, the case $p\mid \cM$ is due to a straightforward calculation. When $p\nmid \cM$, we use \cite[Lemma 2-4]{Sugano84}, noting that $T,T_\theta \in (\cQ_{\Z_p})_{\rm prim}$, to find $u_p\in E_p^\times$ and $g_p\in {\bf GL}_2(\Z_p)$ such that $\gamma=I_\theta(u_p) g_p$. Then, $g_p$ works.) Moreover, the condition $\det g_\infty>0$ follows from the positive definiteness of $T$ and $T_\theta$. Therefore, by Lemma \ref{acci.isom.}, there exists an element $u\in\A_E^\times$ such that
$$I_\theta(u)=\gamma h \kappa,\quad (h:=g_\infty^{-1}\in {\bf GL}_2(\R)^{\circ}, \kappa:=(g_p^{-1})_{p<\infty}\in {\bf \Gamma}_0(\cM\widehat \Z)).$$
This shows that $T_\theta(u)=T_\theta \ss(\gamma)=T$. 
\end{proof}

Since $\sm({\bf GL}_2(\widehat \Z),1)\subset \bK_0(N)$, the relation in \eqref{AdeleFE-L2-1} shows that the function $T\mapsto a_{\varphi}(T):=a_{\varphi}(T;1_4)$ on $\sV(\Z)^{+}$ is ${\bf SL}_2(\Z)$-invariant. Hence the following sum of vectors in $V_{\varrho}$ is well-defined: 
\begin{align}
{\bf R}(\varphi,E,\Lambda):=\sum_{[u] \in {\rm Cl}_{\cM}^{0}(E)} a_\varphi([T_{\theta}(u)])\,\Lambda(u)^{-1}. 
 \label{def-phiEomega}
\end{align}
Recall the integral in \eqref{BesselInt} and the function in \eqref{AdeleFE-W}. 
\begin{lem} \label{BessPerL3}
Let $\varphi \in S_{\varrho}(\bK_0(N))$ and $\Lambda\in\widehat{{\rm Cl}_{\cM}^{0}(E)}_{{\rm prim}}$.
\begin{itemize}
\item[(i)] For $g_\infty \in \sG(\R)^0$, 
$$
B^{T_\theta,\Lambda}(\varphi;g_\infty)=w_{D,\cM}^{-1}\,{\bf B}_{\varrho}^{T_\theta}(g_\infty)({\bf R}(\varphi,E,\Lambda))
$$
where $w_{D,N}:=\# (U(\cM)\cap\fo_E^\times)$.
\item[(ii)] Set $\sigma:=\left[\begin{smallmatrix} 1 & t \\ 0 & -1 \end{smallmatrix}\right]$. Then, 
\begin{align}
\varrho(r)(B^{T_\theta,\Lambda}(\varphi;b_\R^\theta))&=B^{T_\theta,\Lambda}(\varphi;b_\R^\theta), \quad r \in {\bf SO}(2),
 \label{BessPerL3-f1}
\\
\varrho(\delta)(B^{T_\theta,\Lambda}(\varphi;b_\R^\theta))&=B^{-T_\theta \ss(\sigma),\Lambda}(\varphi;b_\R^\theta).
 \label{BessPerL3-f2}
\end{align} 
 \end{itemize}
\end{lem}
\begin{proof} (i) The set $\A_{E}^\times/E^\times E_\R^\times$ is decomposed to a disjoint union of $\A_{\fin}^{\times}U(\cM)$-orbits $[u]\A_{\fin}^{\times}U(\cM)$ for different classes $[u]\in {\rm Cl}_{\cM}^{0}(E)$. Let $\mu$ denote the quotient measure on $\A_E^\times/E^\times E_\R^\times$. For $\tau\in \A_{E}^\times$, set $\sm_\theta(\tau):=\sm(I_{\theta}(\tau),\nr_{E/\Q}(\tau))$. Then by \eqref{BesselInt} and \eqref{AdeleFE-L1-1}, 
\begin{align*} B^{T_\theta,\Lambda}(\varphi ;g_\infty)
&={\bf B}_{\varrho}^{T_\theta}(g_\infty)\Bigl(\int_{\A_E^\times /E^\times E_\infty^\times} a_{\varphi}(T_\theta; \sm_\theta(u))\,\Lambda(u)^{-1}\,\d \mu(u)\Bigr)
\\
&={\bf B}_{\varrho}^{T_\theta}(g_\infty)\Bigl(\sum_{[u] \in {\rm Cl}_{\cM}^{0}(E)} \int_{[u] \A_{\fin}^\times U(\cM)}  a_{\varphi}(T_\theta;\sm_\theta(u))\,\Lambda(u)^{-1}\,\d\mu(u) \Bigr).
\end{align*}
Since $\sm_{\theta}(\widehat \cO_{E}^\times)\subset \bK_{0}(N)$, by \eqref{AdeleFE-L1-1}, the formula in the parenthesis becomes
\begin{align*}
\sum_{[u]\in {\rm Cl}_{\cM}^{0}(E)} \mu([u]\A_{\fin}^\times U(\cM))\,a_{\varphi}(T_\theta;
\sm_\theta(u))\,\Lambda(u)^{-1}.
\end{align*}
Let $I_\theta(u)$ decompose as in \eqref{BesselPerL2-2}. Then, by \eqref{AdeleFE-L2-1} and Lemma \ref{BessePerL2}, we have 
$$
a_{\varphi}(T_\theta;\sm_\theta(u))=a_{\varphi}(T_\theta \ss(\gamma_u);1)=a_\varphi(T_\theta(u)).
$$
It remains to compute $\mu([u]\A_{\fin}^\times U(\cM))$. Since any fiber of the natural surjection from $\A_{\fin}^\times U(\cM)$ onto $[u]\A_{\fin}^\times U(\cM)$ has a simply transitive action of the global unit group $\cO_E^\times$, we have $\mu([u]\A_{\fin}^\times U(\cM))=1/\# (U(\cM)\cap\fo_E^\times)=1/w_{D,N}$. 

(ii) We have $I_\theta(e^{i\alpha})=A_\theta^{-1} \left[\begin{smallmatrix} \cos \alpha & -\sin \alpha \\ \sin \alpha & \cos \alpha \end{smallmatrix}\right]A_\theta$, or equivalently 
$$\sm(I_\theta(e^{i\alpha}),1)=b_\R^\theta \diag(\left[\begin{smallmatrix} \cos \alpha & -\sin \alpha \\ \sin \alpha & \cos \alpha \end{smallmatrix}\right], \left[\begin{smallmatrix} \cos \alpha & -\sin \alpha \\ \sin \alpha & \cos \alpha \end{smallmatrix}\right])(b_\R^\theta)^{-1}$$ for $\alpha \in \R$. Due to this and $\Lambda_\infty=1$, formula \eqref{BessPerL3-f1} is proved by a change of variables $\tau\rightarrow \tau u^{-1}$ in the integral \eqref{BesselInt}. Since $
b_\R^{\theta}\sm(\delta,1)(b_\R^\theta)^{-1}=\sm(\sigma,1)$, formula \eqref{BessPerL3-f2} is proved in the same way as the proof of Lemma \ref{AdeleFE-L2}.
\end{proof}

Let $\varphi \in S_{\varrho}(\bK_0(N))-(0)$; then $l_1\equiv l_2\pmod{2}$ as noted above. As such, $\dim_\C(V_\varrho^{{\bf SO}(2)})=1$; indeed, $V_\varrho^{{\bf SO}(2)}=\C\,\sv_\varrho^{0}$ with $\sv_\varrho^0:=(X^2+Y^2)^{(l_1-l_2)/2}$. Set 
$$
\sv_\varrho^\theta:=\bigl(\tfrac{-2}{\sqrt{|D|}}\bigr)^{\frac{l_1+l_2}{2}}\,\varrho({}^t A_\theta)\sv_\varrho^{0}=\bigl(\tfrac{-2}{\sqrt{|D|}}\bigr)^{\frac{l_1-l_2}{2}}(X^2+tXY+N Y^2)^{\frac{l_1-l_2}{2}}.
$$
Note that ${\bf B}_\varrho^{T_\theta}(b_\R^\theta)=(\frac{-2}{\sqrt{|D|}})^{-(l_1+l_2)/2}e^{-2\pi \sqrt{|D|}}\,\varrho({}^t A_\theta)^{-1}$, so that ${\bf B}_\varrho^{T_\theta}(b_\R^\theta)(\sv_\varrho^\theta)=e^{-2\pi \sqrt{|D|}}\sv_\varrho^0$. Thus, as a corollary to Lemma \ref{BessPerL3}, the vector ${\bf R}(\varphi,E,\Lambda)\in V_\varrho$ is non-zero if and only if $B^{T_\theta,\Lambda}(\varphi;b_\R^\theta)\not=0$, in which case there exists a unique scalar $R(\varphi,E,\Lambda)$ such that ${\bf R}(\varphi,E,\Lambda)=R(\varphi,E,\Lambda)\,\sv_\varrho^\theta$, or equivalently 
\begin{align}
B^{T_\theta,\Lambda}(\varphi;b_\R^\theta)=w_{D}^{-1}e^{-2\pi \sqrt{|D|}}\, R(\varphi,E,\Lambda)\sv_\varrho^0.
 \label{BessPerL3-f3}
\end{align}
Define a $C^\infty$-function $B_\varrho^{T_\theta}:\sG(\R)^0 \longrightarrow V_\varrho$ as 
\begin{align}
B^{T_\theta}_{\varrho}(g_\infty):={\bf B}^{T_\theta}_\varrho(g_\infty)(\sv_\varrho^{\theta}), \quad g_\infty\in \sG(\R)^0. 
 \label{Def-VectArchBess}
\end{align}

\subsection{Sign condition}
The Galois group ${\rm Gal}(E/\Q)$ acts on ${\rm Cl}_{\cM}^{0}(E)$ naturally, hence on the orbit space ${\cQ}_{\cM}^{+}(D)/{\bf SL}_2(\Z)$ through the bijection in Lemma \ref{BessePerL2}(iii). The following lemma describes the conjugate action on ${\cQ}_{\cM}^{+}(D)/{\bf SL}_2(\Z)$ explicitly. Recall the element $\sigma \in {\bf GL}_2(\Z)$ defined in Lemma \ref{BessPerL3} (ii). 

\begin{lem} \label{ConjClass}
For $u\in {\rm Cl}_\cM^0(E)$, the element $T_{\theta}(\bar u)$ is ${\bf \Gamma}_0(\cM)$-equivalent to $-T_\theta(u)\,\ss\left(\delta\right)$. In other words, the conjugate action on $\cQ_{\cM}^+(D)/{\bf \Gamma}_0(\cM)$ is induced by the map $T\mapsto -T\,\ss(\delta)$. 
\end{lem}
\begin{proof} The defining formula of $I_{\theta}$ in \eqref{Embed-DefI} yields
\begin{align*}
I_\theta(a+\theta b)=\left[\begin{smallmatrix} 1 & \theta \\ 1& \bar \theta \end{smallmatrix}\right]^{-1}  
\left[\begin{smallmatrix} a+\theta b & 0 \\ 0 & a+\bar \theta b \end{smallmatrix}\right] \left[\begin{smallmatrix} 1 & \theta \\ 1 & \bar \theta \end{smallmatrix}\right], \quad a,b\in \Q. 
\end{align*}
Set $t:=\tr_{E/\Q}(\theta)$. Then, by $\bar \theta=t-\theta$, we have $a+b\bar\theta=a'+b'\theta$ with $a'=a+tb\in \Q$ and $b'=-b\in \Q$. Then , a computation reveals 
\begin{align}
I_{\theta}(a+b\bar \theta)&=\sigma I_\theta(a+b\theta)\sigma^{-1}.
\label{ConjClass-f1}
\end{align}
By the decomposition of $I_\theta(u)$ for $u\in \A_E^\times$ in \eqref{BesselPerL2-2}, we may take $\gamma_{\bar u}=\sigma \gamma_{u} \sigma^{-1}$. Thus, $
T_\theta(\bar u)=T_\theta\ss(\gamma_{\bar u})
=T_\theta \ss\left( \sigma \gamma_{u}\sigma^{-1}\right)$. A computation shows $T_\theta \ss\left(\sigma \right)=-T_\theta$. Hence, 
\begin{align*}
T_\theta(\bar u)=-T_\theta \ss\left(\gamma_{u}\sigma^{-1}\right)
=-T_\theta(u) \ss\left(\sigma^{-1}\right)
=-T_{\theta}(u)\ss(\delta)\ss\left(\left[\begin{smallmatrix} 1 & t \\ 0 & 1\end{smallmatrix}\right]\right).
\end{align*}
The last matrix is ${\bf SL}_2(\Z)$-equivalent to $-T_{\theta}(u)\ss(\delta)$. This completes the proof. \end{proof}

For $\Lambda \in \widehat {{\rm Cl}_{\cM}^{0}(E)}_{\rm prim}$, recall its conjugate $\Lambda^\dagger\in \widehat {{\rm Cl}_{\cM}^{0}(E)}_{\rm prim}$ (see \eqref{Def-ConjOmega}). 

\begin{lem} \label{BesselConj}
Let $\varphi\in S_\varrho(\bK_0(N))$ and $\Lambda\in \widehat {{\rm Cl}_{\cM}^{0}(E)}_{\rm prim}$ and suppose $B^{T_\theta,\Lambda}(\varphi;b_\R^\theta)\not=0$. Then, 
\begin{align*}
R(\varphi,E,\Lambda^\dagger)
=(-1)^{l_2}\,R(\varphi,E,\Lambda). 
\end{align*}
\end{lem}
\begin{proof} This is proved by Lemma \ref{ConjClass} and \eqref{BessPerL3-f3} and by $\varrho(\delta )\sv_\varrho^0=(-1)^{l_2}\sv_\varrho^0$. Note that $\delta^{-1} \sigma=\left[\begin{smallmatrix} 1 & t\\ 0 & 1 \end{smallmatrix} \right] \in {\bf SL}_2(\Z)$, so that $a_{\varphi}(-T_\theta \ss(\sigma),1)=a_{\varphi}(-T_\theta \ss(\delta),1)$. 
%{\allowdisplaybreaks
%\begin{align*}
%\sum_{[u]\in {\rm Cl}(E)} \Lambda'(u)^{-1}\,a_\varphi([T_\theta(u)])
%&=
%\sum_{[u]\in {\rm Cl}%(E)}\Lambda(\bar{u})^{-1}\,a_\varphi([T_\theta(u)])\\
%&=
%\sum_{[u]\in {\rm Cl}(E)}\Lambda(u)^{-1}\,a_\varphi([T_\theta(\bar {u})])
%\\
%&=
%\sum_{[u]\in {\rm Cl}(E)}\Lambda(u)^{-1}\,a_\varphi([-T_{\theta}%(u)\ss\left(\delta\right)])\\
%&=(-1)^{l}\, \sum_{[u] \in {\rm Cl}%(E)}\Lambda(u)^{-1}\,a_\varphi([T_\theta(u)]).
%\end{align*}}
\end{proof}

\subsection{Automorphic representations} \label{sec:AutRep}
Let $\Pi_{\rm cusp}(\sZ\bsl \sG)$ denote the set of all those irreducible cuspidal representation of $\sG(\A)$ with trivial central characters, i.e., those irreducible $(\fg,\bK_\infty)\times \sG(\A_\fin)$-submodules $(\pi,V_{\pi})$ of the space of cusp forms on $\sZ(\A)\sG(\Q)\bsl \sG(\A)$. We endow $V_{\pi}$ with the restriction of the $L^2$-inner product 
$$(\varphi\mid \varphi_1)_{L^2}:=\int_{\sZ(\A)\sG(\Q)\bsl \sG(\A)}\varphi(g)\overline{\varphi_1(g)}\,\d g,  \qquad \varphi,\,\varphi_1\in V_\pi.$$ 
For $\pi\in \Pi_{\rm cusp}(\sZ\bsl \sG)$, we fix its restricted tensor decomposition $\pi\cong \otimes_{p\leq \infty} \pi_p$ with $(\pi_p,V_{\pi_p})$ being an irreducible admissible unitarizable representation of $\sG(\Q_p)$ with trivial central character such that a $\bK_p$-invariant vector $\xi_{p}^0\in V_{\pi_p}-(0)$ is preassigned for almost all $p<\infty$. 
An element $\lambda \in (\Z^2)_{\rm dom}$ is the highest weight of the minimal $\bK_\infty$-type of a holomorphic discrete series representation (HDS for short) of ${\bf Sp}_2(\R)$ if and only if $l_2>2$, in which case the Harish-Chandra parameter of the HDS is $(l_1-1,l_2-2)$. For such $\lambda$ and $N\in \Z_{>0}$, let $\Pi_{\rm cusp}(\lambda,N)$ be a subset of $\pi\in \Pi_{\rm cusp}(\sZ\bsl \sG)$ having the following properties:
\begin{itemize}
\item[(i)] As a $(\fg,\bK_\infty)$-module $\pi_\infty \cong D_{l_1-1,l_2-2}\oplus D_{-l_2+2,-l_1+1}$, where $D_{m_1,m_2}$ is the discrete series representation of $\sG(\R)^0$ of Harish-Chandra parameter $(m_1,m_2)$ with central character $\sZ(\R)\cong \R^\times \ni z\longmapsto {\rm sgn}(z)^{m_1+m_2+1}$. 
\item[(ii)] $V_{\pi}^{\bK_0(N)}\not=(0)$. 
\item[(iii)] Let $(\varrho,V_\varrho)$ be the irreducible rational representation of highest weight $\lambda$ as before. The space $S_{\varrho}(\bK_0(N))$ is an orthogonal direct sum of spaces $V_{\pi}^{\bK_0(N)}[\varrho]:=\{\varphi \in S_{\varrho}(\bK_0(N))\mid \varphi_v\in V_{\pi}\,(\forall v\in V_\varrho)\}$ for $\pi \in \Pi_{\rm cusp}(\lambda,N)$, where $\varphi_v(g):=(\varphi(g)\mid v)_\varrho$ for $\varphi \in S_{\varrho}(\bK_0(N))$ and $v\in V_\varrho$. 
\end{itemize}
Although there may be many choices of $\Pi_{\rm cusp}(\lambda,N)$, we fix one of them once and for all. For $\pi \in \Pi_{\rm cusp}(\lambda,N)$, we fix a $\bK_\infty$-intertwining map $V_{\varrho}\oplus \overline{V_\varrho}\hookrightarrow D_{l_1-1,l_2-2}\oplus D_{-l_2+2,-l_1+1}\,(\cong \pi_\infty)$ once and for all, where $\overline{V_\varrho}:=\overline{\C}\otimes_{\C}V_\varrho$ is the complex conjugate of $V_\varrho$. 

\subsection{Basic assumptions}\label{sec:BA}
From now on, we fix a pair $(\lambda=(l_1,l_2),N,M,\cM)\in (\Z^2)_{\rm dom} \times \Z_{>0}\times \Z_{>0}\times \Z_{>0}$ satisfying the following conditions: 
\begin{itemize}
\item[(A-i)] $N$ and $\cM$ is square free.
\item[(A-ii)] $l_1\equiv l_2\pmod{2}$, and $l_2\in \Z_{\geq 3}$ so that $(l_1-1,l_2-2)$ is the Harish-Chandra parameter of a HDS. 
\item[(A-iii)] Any two of $N$, $M$, and $\cM$ are coprime.
\item[(A-iv)] $M$ is odd
\item[(A-v)] All prime divisors of $NM\cM$ are inert in $E/\Q$. 
\end{itemize}
We fix characters $\Lambda\in\widehat{{\rm Cl}_{\cM}^{0}(E)}_{{\rm prim}}$ and $\mu=\otimes_{p\leq \infty}\mu_p\in\widehat{\A^\times/\Q^\times\R_{>0}}$ with $M:={\rm cond}(\mu)$ as well.
As usual, the character $\mu$ induces a primitive Dirichlet character $\tilde \mu:(\Z/M\Z)^\times\rightarrow \C^1$ in such a way that 
 $\mu(u)=\tilde \mu(a)^{-1}$ for all $u\in \widehat \Z^\times$ and $a\in \Z$ with $u-a\in M\widehat \Z$
; thus $\mu_p(p)=\tilde \mu(p)$ for primes $p\nmid M$, and $$\prod_{p\mid M} W_{\Q_p}(\mu_p,\psi_p)=M^{-1/2}G(\tilde \mu),$$ 
with $W_{\Q_p}(\mu_p,\psi_p)$ being as in Definition \ref{Def:LclGaussSum} and $G(\tilde \mu):=\sum_{a\in (\Z/M\Z)^\times} \tilde \mu(a)e^{2\pi i a/M}$, the Gauss sum of $\tilde \mu$. 
Let $\Pi_{\rm cusp}^{(T_\theta,\Lambda)}(\lambda,N)$ denote the set of all those $\pi\in \Pi_{\rm cusp}(\lambda,N)$ that satisfies the condition 
\begin{itemize}
\item[(A-vi)] $\pi$ admits a global $(T_\theta,\Lambda)$-Bessel model, i.e., there exists $\varphi \in V_\pi$ such that $B^{T_\theta,\Lambda}(\varphi;g)\not=0$ for some $g\in \sG(\A)$.
\end{itemize}
 We remark that the conditions listed above (when $\mu={\bf 1}$ and $l_1=l_2$) are also imposed in the  most part of \cite{DPSS}. Let $\Pi_{\rm cusp}^{(T_\theta,\Lambda)}(\lambda,N)^{\rm SK}$ denote the set of $\pi \in \Pi_{\rm cusp}^{(T_\theta,\Lambda)}(\lambda,N)$ that is the Saito-Kurokawa lift (\cite{PS83}) of a cuspidal automorphic representation of ${\bf PGL}_2(\A)$, which is locally described by  \cite{Schmidt2005} (see also \cite{Schmidt2007}). From \cite[\S4]{Schmidt2005}, the set $\Pi_{\rm cusp}^{(T_\theta,\Lambda)}(\lambda,N)^{\rm SK}=\emp$ unless $l_1=l_2$, i.e., $V_{\varrho}$ is one dimensional.
 We note that all the representations $\pi\cong \otimes_{p}\pi_p$ in $\Pi_{\rm cusp}^{(T_\theta,\chi)}(l,N)\setminus\Pi_{\rm cusp}^{(T_\theta,\chi)}(l,N)^{\rm SK}$ are non-CAP
 ({\it cf}. \cite{Soudry}, \cite[\S3.5]{KWY}, \cite[Corollary 4.5]{PitaleSchmidt2009}); then, by \cite{Weissauer},  $\pi_p$ is tempered for all $p \nmid N$.
 Moreover, by invoking \cite{Arthur} (see also \cite{Schmidt2018}) , any $\pi\in \Pi_{\rm cusp}^{(T_\theta,\chi)}(l,N)\setminus\Pi_{\rm cusp}^{(T_\theta,\chi)}(l,N)^{\rm SK}$ is either a Yoshida lift, i.e., there exists a pair of irreducible cuspidal automorphic representations $(\sigma_1,\sigma_2)$ of ${\bf GL}_2(\A)$ such that $L(s,\pi)=L(s,\sigma_1)L(s,\sigma_2)$, or a ``general type'', that is, there exists an irreducible cuspidal automorphic representation $\Pi$ of ${\bf GL}_4(\A)$ such that $L(s,\Pi,\wedge^2)$ has a pole at $s=1$ and $L(s,\pi)=L(s,\Pi)$. Let $\Pi_{\rm cusp}^{(T_\theta,\chi)}(l,N)^{\rm Y}$ (resp. $\Pi_{\rm cusp}^{(T_\theta,\chi)}(l,N)^{\rm G}$) be the set of $\pi\in \Pi_{\rm cusp}^{(T_\theta,\chi)}(l,N)$ that is a Yoshida lift (resp. a general type). Then,  
\begin{align}
\Pi_{\rm cusp}^{(T_\theta,\chi)}(l,N)= \Pi_{\rm cusp}^{(T_\theta,\chi)}(l,N)^{\rm G}\cup \Pi_{\rm cusp}^{(T_\theta,\chi)}(l,N)^{\rm Y} 
\cup \Pi_{\rm cusp}^{(T_\theta,\chi)}(l,N)^{\rm SK}\quad \text{(disjoint union)}. \label{SKT}
\end{align}
On the other hand, the set $\Pi_{\rm cusp}^{(T_\theta,\Lambda)}(\lambda,N)$ can be separated into the subsets consisting of all ``newforms" (i.e., $N_\pi=N)$ and all ``oldforms" (i.e., $N_\pi\not=N$) denoted by $\Pi_{\rm cusp}^{(T_\theta,\Lambda)}(\lambda,N)^{\rm new}$ and $\Pi_{\rm cusp}^{(T_\theta,\Lambda)}(\lambda,N)^{\rm old}$ respectively. For $\bullet\in\{{\rm G},{\rm Y},{\rm SK}\}$ and $\ast\in\{{\rm new},{\rm old}\}$, we set
$$
\Pi_{\rm cusp}^{(T_\theta,\Lambda)}(\lambda,N)^{\bullet,\ast}= \Pi_{\rm cusp}^{(T_\theta,\Lambda)}(\lambda,N)^{\bullet}\cap \Pi_{\rm cusp}^{(T_\theta,\Lambda)}(\lambda,N)^{\ast}.
$$
Let $N_{\pi}$ be the product of primes $p<\infty$ such that $\pi_p$ is not spherical. By condition (ii) in \S\ref{sec:AutRep}, we have $N_\pi\mid N$. Due to (A-ii), the representations $\pi_p$ for $p\mid N_\pi$ have non-zero $\bK_0(p\Z_p)$-fixed vectors, which, in turn, implies that $\pi_p$ is Iwahori spherical so that their isomorphism classes are listed in \cite[Table 3]{Schmidt} (For extended tables, we refer to \cite[Appendix]{RobertsSchmidt}.)  Note that in \cite{Schmidt} a different symplectic form is used to define the symplectic group; up to the adjustment for this difference, the group $P_{1}$ in \cite{Schmidt} is our $\bK_0(p\Z_p)$. In \cite{RobertsSchmidt}, all irreducible admissible representations of $\sG(\Q_p)$ that admit local Bessel models are classified, and the result is conveniently summarized in \cite[Table 2]{PitaleSchmidt}. Here is a summary of what is available for our $\pi$ (as in \S\ref{sec:BA}): 
\begin{itemize}
\item $\pi_\infty$ as a representation of $\sG(\R)^0$ is a direct sum of a HDS and its complex conjugate; as such, $\pi$ is {\rm CAP} if and only if it is a Saito-Kurokawa lift from a cuspidal representation of ${\bf PGL}_2(\A)$; which happens only if $l_1=l_2$.  
\item  Suppose $p\nmid N_\pi$. Then, the local representation $\pi_p$ is of type ${\rm I}$ and tempered if $\pi$ is non-{\rm CAP}, and is of type ${\rm II}b$ when $\pi$ is {\rm CAP}. 
\item Suppose $p\mid N_\pi$. Then, the local representation $\pi_p$ is either of type ${\rm III}a$, in which case $\dim_{\C}V_{\pi_p}^{\bK_0(p\Z_p)}=2$, or of type ${\rm VI}b$, in which case $\dim_{\C}V_{\pi_p}^{\bK_0(p\Z_p)}=1$; when $\pi$ is {\rm CAP}, $\pi_p$ has to be of type ${\rm VI}b$.  
\end{itemize}

\subsection{Bessel models}\label{sec:BesselMod}
 Let $p<\infty$; for any irreducible admissible representation $(\pi_p,V_{\pi_p})$ of $\sG(\Q_p)$, let $(V_{\pi_p}^*)^{T_\theta,\Lambda}$ denote the space of all $\C$-linear forms $\ell:V_{\pi_p}\longrightarrow \C$ that satisfies 
$$
\ell(\pi_p(\sm(I_{\theta}(\tau),\nr_{E/\Q}(\tau))n)\xi)=\Lambda(\tau)\psi_{T_\theta}(n)\,\ell(\xi), \quad \xi\in V_{\pi_p},\,\tau\in E_p^\times,\,n\in \sN(\Q_p).$$
It is known that $\dim_{\C}(V_\pi^*)^{T_\theta,\Lambda}\leq 1$ (\cite{Novodvorski}, \cite{TaklooBinghashPrasad}). We say that $\pi_p$ has a local $(T_\theta,\Lambda_p)$-Bessel model if $(V_{\pi}^*)^{T_\theta,\Lambda}\not=(0)$; when this is the case, the space of functions of the form $g\mapsto \ell(\pi_p(g)\xi)$ with $\xi\in V_{\pi_p}$ is independent of $\ell\in (V_{\pi_p}^*)^{T_\theta,\Lambda_p}-(0)$; this space is denoted by $\cB(T_\theta,\Lambda_p)[\pi_p]$ and is called the local $(T_\theta,\Lambda_p)$-Bessel model of $\pi_p$. 
When $\pi_p$ is spherical, it is known that $\pi_p$ has a local $(T_\theta,\Lambda_p)$-Bessel model, and the space $\cB(T_\theta,\Lambda_p)[\pi_p]$ contains a unique $\bK_p$-invariant function $B_{\pi_p}$ such that $B_{\pi_p}^{0}(1_4)=1$ (\cite[Theorem 2-I]{Sugano84}, \cite{BFF}). Let $\xi_{p}^0$ be the non-zero $\bK_p$-fixed vector in $V_{\pi_p}$; then there exists a unique element $\ell_{\pi_p}^0\in (V_{\pi_p}^*)^{T_\theta,\Lambda_p}-(0)$ such that 
$\ell_{\pi_p}^0(\pi_p(g) \xi_{\pi_p}^0)=B_{\pi_p}^0(g)$ for all $g\in \sG(\Q_p)$. The pair $(\ell_{\pi_p}^0,\xi_{\pi_p}^0)$ is referred to as the unramified $(T_\theta,\Lambda_p)$-Bessel datum for $\pi_p$.

Let $\pi\cong \otimes_{p\leq \infty}\pi_p \in \Pi_{\rm cusp}^{(T_\theta,\Lambda)}(\lambda,N)$ (see \S\ref{sec:BA}).

%Let $N_{\pi}$ be the product of $p\mid N$ such that $\pi_p$ is not spherical. 
\begin{defn}
A system $\{(\ell_{p}, \xi_{p})\}_{p<\infty}$ with $\ell_{p}\in (V_{\pi_p}^*)^{T_\theta,\Lambda_p}-(0)$, $\xi_{p}\in V_{\pi_p}^{\bK_0(N_\pi \Z_p)}$ is called a $(T_\theta,\Lambda)$-Bessel data for $\pi$ if $(\ell_{p},\xi_p)=(\ell_{\pi_p}^0,\xi_{\pi_p}^0)$ for $p\nmid N_\pi$ and $\ell_p(\xi_p)=1$ for all $p<\infty$.
\end{defn}
By \cite[Theorem 2 8.2 and 9.3]{PitaleSchmidt}, a $(T_\theta,\Lambda)$-Bessel data exists for our $\pi$. Once a $(T_\theta,\Lambda)$-Bessel data $\{(\ell_p, \xi_{p})\}_{p<\infty}$ is fixed, one can define $\varphi_{\pi,\varrho}^0\in V_{\pi}^{\bK_0(N_\pi)}[\varrho]$ to be the $V_\varrho$-valued cusp form such that for any $v\in V_\varrho$, the function $(\varphi_{\pi,\varrho}^0(g)\mid v)_{\varrho}$ in $V_\pi\cong\otimes_{p\leq \infty}V_{\pi_p}$ corresponds to the pure tensor $\bar v \otimes (\otimes_{p<\infty} \xi_{p})$, where $\bar v=1\otimes v \in \overline{V_\varrho}\hookrightarrow V_{\pi_\infty}$. A particular choice of $\{(\ell_p,\xi_p)\}$ will be made in \S\ref{sec:Local period} so that $\varphi_{\pi,\varrho}^0$ corresponds to a newform on the arithmetic quotient $\Gamma_0(N)\bsl \fh_2$ in the sense of \cite[\S3.2]{DPSS}.

\begin{lem} \label{BesselMod-L1}
 Let $\{(\ell_p,\xi_p)\}_{p<\infty}$ be a $(T_\theta,\Lambda)$-Bessel data for $\pi\in \Pi_{\rm cusp}(\lambda,N)$. Let $\phi=\otimes_p \phi_p\in\mathcal{S}(\A_E^2)$ be a decomposable element. Then, for any $\varphi\in V_{\pi}^{\bK_0(N)}[\varrho]$ such that, for any $v\in V_\varrho$, $\varphi_v$ corresponds to the pure tensor $v \otimes(\otimes_{p<\infty} v_p) \in \bigotimes_{p\leq \infty}V_{\pi_p}$, we have
\begin{align}
B^{T_\theta,\Lambda}(\varphi;g)=B^{T_\theta,\Lambda}(\varphi_{\pi,\varrho}^0;g_\infty)\prod_{p<\infty} \ell_p(\pi_p(g_p) v_p), \quad g=(g_p)_{p}\in \sG(\A). \label{BesselMod-L1-f0}
\end{align} 
For $\Re(s)>1$, $v\in V_\varrho$ and $b_\fin=(b_p)_{p<\infty} \in \sG(\A_\fin)$, \begin{align}
\langle E(\phi,s,\Lambda,\mu),R(b_\fin b_\R^\theta)\overline{\varphi_v} \rangle=\tfrac{\sqrt{|D|}}{2} Z_v^{(\infty)}(\phi_\infty,\varphi^0_{\pi};s,\mu_\infty,\Lambda)\times \prod_{p\leq \infty} Z_p(\phi_p,\overline{B_{v_p}};s,\mu_p;b_p),
\label{Bessel-Mod-L1-f1}
\end{align}
where 
\begin{align*}
Z_v^{(\infty)}(\phi_\infty,\varphi^0_\pi;s,\mu_\infty,\Lambda)&:=\int_{\bK_\infty^\#} \int_{\R^\times} f_{\phi_\infty}^{(s,1,\mu_\infty)}(k_\infty^\#) \mu_\infty(a)|a|_\R^{s-1}\\
&\qquad \times (v\mid B^{T_\theta, \Lambda}(\varphi^0_{\pi,\varrho} ;\sm(a 1_2,a) \iota_\theta(k^\#) b_\R^\theta))_\varrho\,\d^\times a\,\d k^\#_\infty, \\
Z_{p}(\phi_p, \overline{B_{v_p}};s,\mu_p;b_p)&:=\int_{\bK_p^\#} \int_{\Q_p^\times} f_{\phi_p}^{(s,\Lambda_p,\mu_p)}(k_p^{\#})\mu_p(a)|a|_p^{s-1}\,\overline{B_{v_p}}(\sm(a_p 1_2,a_p)\iota_\theta(k_p^\#)b_p)\,\d^\times a_p\,\d k^\#_p
\end{align*}
with $
B_{v_p}(g_p):=\ell_{p}(\pi_p(g_p)v_p)$ for $g_p\in \sG(\Q_p)$.
\end{lem}
\begin{proof} Fix $g_\infty \in \sG(\R)^0$ and $v\in V_{\varrho}$ and regard $\varphi\mapsto B^{T_\theta,\Lambda}(\varphi_v;g_\infty)$ as a linear functional on $\otimes_{p<\infty} V_{\pi_p}$ by the natural inclusion $\bigotimes_{p<\infty}V_{\pi_p} \hookrightarrow v \otimes(\bigotimes_{p<\infty}V_{\pi_p})\hookrightarrow V_{\pi}$. Then, by the local multiplicity one theorem for Bessel functional on $\sG(\Q_p)$ recalled above, there exists $C_v(g_\infty)\in \C$ such that
\begin{align}
(v\mid B^{T_\theta,\Lambda}(\varphi;g_\infty))_{\varrho}=C_v(g_\infty)\,
\prod_{p<\infty} \overline{\ell_{p}(v_p)} 
 \label{Bessel-Mod-L1-f110}
\end{align}
for $\varphi$ corresponding to $\otimes_{p<\infty} v_p$. To determine $C_v(g_\infty)$, set $v_p=\xi_{p}$ for all $p<\infty$; then by $\ell_p(\xi_p)=1$, we get $(v\mid B^{T_\theta,\Lambda}(\overline{\varphi_{\pi,\varrho}^0};g_\infty))_\varrho=C_v(g_\infty)$. Now we apply \eqref{RSint-L1-f00} with $T_\theta,\Lambda$ being replaced with $-T_\theta,\Lambda^{-1}$ and with $\varphi$ being replaced with $R(b_\fin)\varphi$. \end{proof}

\subsection{Gamma factor and global Bessel period} \label{sec:Gamma}
Recall the point $b_\R^\theta$ in \eqref{Def-b} and the vector $v_\varrho^\theta\in V_\varrho$ in \eqref{Def-VectArchBess}. Since $\nu(b_{\R}^\theta)>0$, the point $\sm(a1_2,a)\iota_\theta(k^\#_\infty)b_{\R}^\theta$ with $a\in \R^\times$ and $k_\infty^\#\in \bK_\infty^\#$ belongs to $\sG(\R)^0$ if and only if ${\rm sgn}(a){\rm sgn}(k_\infty^\#)=+1$, where ${\rm sgn}(k_\infty^\#):={\rm sgn}(\nu(\iota_\theta(k_\infty^\#))$. Thus, by \eqref{BesselInt} and \eqref{AdeleFE-L1-1}, we have $B^{T_\theta, \Lambda}(\varphi_\pi^0;\sm(a 1_2,a) \iota_\theta(k_\infty^\#) b_\R^\theta)=0$ unless ${\rm sgn}(a){\rm sgn}(k_\infty^\#)=+1$. By Lemma \ref{BessPerL3} and \eqref{Def-VectArchBess}, 
$$
B^{T_\theta, \Lambda}(\varphi_\pi^0;g_\infty)=w_D^{-1}R(\varphi_\pi^0,E,\Lambda)\,B^{T_\theta}_\varrho(g_\infty), \quad g_\infty\in \sG(\R)^0.
$$
Substituting this, we have that $Z^{(\infty)}_v(\phi_\infty, \varphi^0_\pi ;s,\Lambda)$ equals 
\begin{align}&2w_D^{-1}
\overline{R(\varphi^0_\pi,E,\Lambda)} \int_{(\bK^\#_\infty)^0} \int_{0}^{\infty} |a|_\R^{s-1}\,f_{\phi_\infty}^{(s,1,\mu_\infty)}(k_\infty^\#) (v\mid B^{T_\theta}_{\varrho}(\sm(a 1_2,a)\iota_\theta(k_\infty^\#) b_\R^\theta))_{\varrho}\,\d^\times a\,\d k^\#_\infty. 
 \label{CompLZ-f0}
\end{align}
Now we specify $\phi_\infty$. Recall that the highest weight of $\varrho$ is $\lambda=(l_1,l_2)$ and $l_1-l_2\in 2\Z_{>0}$. Set $d:=l_1-l_2$ and define 
$$\sf_{\varrho}(u):=(d+1)
\frac{(\varrho(C \bar u C^{-1})\sv_\varrho^0\mid \sv_\varrho^0)_{\varrho}}{(\sv_\varrho^0\mid \sv_\varrho^0)_\varrho}, \qquad u \in (\bK^\#_\infty)^0$$
with $C:=\frac{1}{\sqrt{2}}\left[\begin{smallmatrix} 1 & 1 \\ i & -i \end{smallmatrix}\right]$. If $u=\left[\begin{smallmatrix} a& -\bar b \\ b & \bar a \end{smallmatrix}\right] \in (\bK_\infty^\#)^0={\bf SU}(2)$ with $a=a'+ia'',\,b=b'+ib''$ and $A,B\in {\bf Mat}_2(\R)$ are defined as in \eqref{Embed-f1}, so that $(b_\R^\theta)^{-1}\iota_\theta(u)b_\R^\theta=k_\infty(A+iB)$ (Lemma \ref{MCPtArc-L1}), then a computation reveals $C\bar uC^{-1}=A-iB$. Thus, the automorphism $u\mapsto C\bar u C^{-1}$ of ${\bf SU}(2)$ brings the subgroup $\sB^\#(\R)\cap (\bK^\#_\infty)^0=\{\left[\begin{smallmatrix} a & 0 \\ 0 & a^{-1}\end{smallmatrix} \right]\mid a\in \C^1\}$ to the subgroup ${\bf SO}(2)$, by which $\sv_\varrho^0$ is fixed. Hence, $\sf_\varrho$ is left $\sB^\#(\R)\cap (\bK^\#_\infty)^0$-invariant. Thus, by $\sG^\#(\R)=\sB^\#(\R)(\bK_\infty^\#)^0$, there exists a unique element $\sf^{(s)}_\varrho$ of ${\mathcal V}^\#(s,1,\mu_\infty)$ such that $\sf_\varrho^{(s)}|\bK_\infty^\#=\sf_\varrho$. Define $\phi_\infty\in {\mathcal S}(\C^2)$ by 
\begin{align}
\phi_\infty(x,y):=\sum_{j=0}^{d/2}\tfrac{(x \bar x)^{j}}{(j!)^2}\tfrac{(-y\bar y)^{d/2-j}}{((d/2-j)!)^2}\,\tfrac{2}{\pi}\exp\left(-\tfrac{2\pi}{\sqrt{|D|}}(x\bar x+y \bar y)\right), \quad (x,y)\in \C^2. 
 \label{Def-phiInfd}
\end{align}
By \eqref{RSint-f0}, 
\begin{align}
f_{\phi_\infty}^{(s,1,\mu_\infty)}=|D|^{\frac{s+1}{2}}(-1)^{\frac{l_1-l_2}{2}}{\Gamma_{\C}\left(s+\tfrac{l_1-l_2}{2}+1\right)}\,\sf_\varrho^{(s)}.
 \label{Ksphericalsection}
\end{align}
For $u\in (\bK_\infty^\#)^0$ and $A,B$ as above, by 
By \eqref{AdeleFE-W} and \eqref{Def-VectArchBess} and by Lemma \ref{MCPtArc-L1}, we have
\begin{align*}
&(v \mid B^{T_\theta}_{\varrho}(\sm(a 1_2,a) \iota_{\theta}(u) b_\R^\theta)_\varrho
=(-1)^{\frac{l_1+l_2}{2}} a^{\frac{l_1+l_2}{2}}\,\exp\left(-2{\sqrt{|D|}\pi}\,a\right)\,(v \mid \varrho(C\bar u C^{-1})^{-1} \sv^{0}_\varrho)_{\varrho}.  
\end{align*}
By using the orthogonal relation of matrix coefficients on ${\bf SU}(2)$, the integral in \eqref{CompLZ-f0} is computed as
\begin{align*}
&|D|^{\frac{s}{2}}(-1)^{l_1}{\Gamma_{\C}\left(s+\tfrac{l_1-l_2}{2}+1\right)} (v\mid v_\varrho^0)_{\varrho}\int_{0}^{\infty}a^{s+\frac{l_1+l_2}{2}-1}\,\exp\left(-2\sqrt{|D|}\pi\,a\right)\,\d^\times a\\
&=|D|^{\frac{1}{2}(1-\frac{l_1+l_2}{2})}(-1)^{l_1}(v\mid v_\varrho^0)_{\varrho}{\Gamma_{\C}\left(s+\tfrac{l_1-l_2}{2}+1\right)}\Gamma_\C\left(s+\tfrac{l_1+l_2}{2}-1\right)
\end{align*}
for $\Re(s)+\tfrac{l_1+l_2}{2}>1$. Recall $L(s,\pi_\infty)=\Gamma_{\C}(s+\tfrac{l_1-l_2}{2}+\tfrac{1}{2})\Gamma_{\C}(s+\tfrac{l_1+l_2}{2}-\tfrac{3}{2})$. Thus,
\begin{align}
Z^{(\infty)}_v(\phi_\infty, \varphi^0_\pi;s,\mu_\infty,\Lambda)&=w_D^{-1} \overline{R(\varphi^0_\pi,E,\Lambda)} (v\mid v_\varrho^0)_{\varrho}(-1)^{l_1}|D|^{\frac{1}{2}(1-\frac{l_1+l_2}{2})}{L\left(s+\tfrac{1}{2},\pi_\infty\right)}.
\label{LocalZetaArchEXPFML}
\end{align}
The formula in \cite[7.23 Lemma]{WallachPSPM} yields 
\begin{align}
M(s)\sf_\varrho^{(s)}=\tfrac{\pi}{s-d/2}\prod_{j=1}^{d/2}\tfrac{s-d/2+j-1}{s+d/2-j}\,\sf_\varrho^{(-s)}. 
\label{Ms-flatd}
\end{align}
By this combined with \eqref{JaqS-f1} and \eqref{Ksphericalsection}, we easily deduce
\begin{align}
f_{\widehat{\phi_\infty}}^{(-s,1,\mu_\infty)}=\tfrac{{|D|}}{4}(-1)^{\frac{l_1-l_2}{2}}f_{\phi_\infty}^{(-s,1,\mu_\infty)}.
\label{FTfphi=fphi}
\end{align}

\subsection{The spinor $L$-function and its functional equation}\label{sec:SpinorFEQ}
Let $\pi\cong \otimes_{v}\pi_v$ be a cuspidal automorphic representation of $\sG(\A)$ with the trivial central character: then, $\bar \pi_p\cong \pi_p^\vee \cong \pi_p$ for all $p<\infty$ by \cite[Proposition 2.3]{TB}. The twist $\pi_\mu$ of $\pi$ by an idele class character $\mu:\A^\times /\Q^\times \longrightarrow \C^1$ is defined on the space $V_\pi$ of $\pi$ as $\pi_\mu(g):=\pi(g)\cdot(\mu\circ \nu)(g)$ for $g\in \sG(\A)$, so that the central character of $\pi_\mu$ is $\mu^2$. 
 Let $\pi \in \Pi_{\rm cusp}^{(T_\theta,\Lambda)}(\lambda,N)$ and $\mu$ be as in \S\ref{sec:BA}, so that $\lambda=(l_1,l_2)\in \Z^2$, $l_1\equiv l_2\pmod{2}$, $l_1\geq l_2>2$ and the ramification loci of $\pi$ and $\mu$ are disjoint. We define the spinor $L$-function of $\pi$ twisted by $\mu$ as the Euler product
$$L(s,\pi,\mu):=\prod_{p<\infty}L(s,(\pi_\mu)_p). \quad \Re(s)>5/2$$
 with $L(s,(\pi_\mu)_p)$ the local $L$-factor listed in \cite[Table A.8]{RobertsSchmidt}. Define 
 \begin{align*}
  \widehat L(s,\pi,\mu):=\Gamma_\C\left(s+\tfrac{l_1-l_2}{2}+\tfrac{1}{2}\right)\Gamma_\C\left(s+\tfrac{l_1+l_2}{2}-\tfrac{3}{2}\right)\,L(s,\pi,\mu).    
 \end{align*}
  Using Lemmas \ref{RSBasicL0}, \ref{BesselMod-L1} and Lemma \ref{BesselConj}, combining \eqref{LocalZetaArchEXPFML}, \eqref{FTfphi=fphi} and the computations of the local zeta integrals for cases 1, 4, 5, and 6 in table \eqref{Table3}, we obtain a meromorphic continuation of $L(s,\pi,\mu)$ to $\C$ as well as the functional equation  
\begin{align*}
\widehat L(s,\pi,\mu)=\varepsilon(s,\pi,\mu)\,\widehat L(1-s,\pi, \bar \mu) \quad\text{with}\quad  \varepsilon(s,\pi,\mu):=(-1)^{l_2}\tilde\mu(N_\pi^2)\left(\tfrac{G(\tilde \mu)}{\sqrt{M}}\right)^4(M^4N_\pi^2)^{1/2-s}
\end{align*}
as expected. Moreover $\widehat L(s,\pi,\mu)$ is holomorphic except for possible simple poles at $s=3/2,-1/2$, which does not occur when $\mu$ is non-trivial. By \eqref{Ms-flatd}, $M(s)\sf_\varrho^{(s)}$ has a simple zero at $s=1$ if $l_1>l_2$, which in turn implies the holomorphy at $s=1$ of the global intertwining operator applied to the section $f_{\phi}^{(s)}$ as well as the Eisenstein series $E(\phi,s,\Lambda,\mu)$ for $\phi_\infty$ as above. Hence, by \eqref{Bessel-Mod-L1-f1} with an appropriate $\phi$, $\widehat L(s,\pi,\mu)$ is holomorphic at $s=3/2$ when $l_1>l_2$. If $\mu$ is real valued, then $G(\tilde \mu)/\sqrt{M}\in \{1,i\}$ and $\varepsilon(1/2,\pi,\mu)=(-1)^{l_2}$, so that $L(1/2,\pi,\mu)=0$ unless $l_2$ is even.

\section{Spectral average of Rankin-Selberg integrals} \label{sec: SpecAvRSI}
The space $S_{\varrho}(\bK_0(N))$ is endowed with the Hermitian inner product associated to the norm $\int_{\sZ(\A)\sG(\Q)\bsl \sG(\A)}(\varphi(g)\mid \varphi(g))_\varrho\,\d g$ $(\varphi \in S_\varrho(\bK_0(N))$. Let $\cH(\sG(\Q_p)\sslash \bK_p)$ be the Hecke algebra for $(\sG(\Q_p),\bK_p)$ for $p<\infty$. For any finite set $S$ of primes $p$ that is prime to $N$, define $\cH_S:=\otimes_{p\in S} \cH(\sG(\Q_p)\sslash \bK_p)$. The $\C$-algebra $\cH_S$ acts on the finite dimensional Hilbert space $S_{\varrho}(\bK_0(N))$ normally by 
$$
[R(f_S)\varphi](g)=\int_{\sG(\Q_S)} \varphi(gx_S)f_S(x_S)\,\d x_S, \quad g\in \sG(\A), \quad f_S\in \cH_{S},\,\varphi \in S_l(\bK_0(N)),$$
where $\sG(\Q_S):=\prod_{p\in S}\sG(\Q_p)$. 
Let $\mu$ and $M$ be as in \S\ref{sec:BA}. We define the Schwartz-Bruhat function $\phi\in\mathcal{S}(\A_E^2)$ associated 
 with $\mu$ as
$$\phi=\prod_{p\le\infty}\phi_p,\quad\phi_p(x,y)=\begin{cases}\cchi_{\cO_{E_p}}(x)\cchi_{\cO_{E_p}}(y)\quad&(p<\infty,p\nmid M\cM)\\\cchi_{p^e\cO_{E_p}}(x)\cchi_{1+p^e\cO_{E_p}}(y)&(p<\infty, p^e\parallel M\cM, e\ge1),\end{cases}$$
 and $\phi_\infty$ being as in \eqref{Def-phiInfd}. In this section, we investigate the averages 
\begin{align}
{\mathbb I}^{(s)}(l,N,f_S)&:=\frac{1}{[\bK_\fin:\bK_0(N)]}
\sum_{\varphi\in {\mathcal B}(\lambda,N)} \langle E(\phi,s,\Lambda,\mu),R(b) R(f_S)\overline{\varphi_{\sv_\varrho^0}}\rangle\,{B^{T_\theta,\Lambda}(\varphi_{\sv_{\varrho}^0};b_\R^\theta)}
, \label{Average}
\end{align}
where $\sv_\varrho^0\in V_\varrho^{{\bf SO}(2)}$ is the vector from \S\ref{sec:FC} and $b=(b_p)_{p\leq \infty} \in \sG(\A)$ is define by \begin{align}
b_p=\begin{cases}1_4\quad&(p<\infty, p\nmid NM),\\
\eta_p&(p<\infty, p\mid N),\\
b_p^M:=\left[\begin{smallmatrix} p^e1_2 & T_\theta^{\dagger} \\ 0 & 1_2 \end{smallmatrix}\right]&(p<\infty, p^e\parallel M, e\ge1),\\
b_p^{\cM}:=\left[\begin{smallmatrix} p^2 &  &  &  \\  & p &  &  \\  &  & 1 &  \\  &  &  & p \end{smallmatrix}\right]&(p\mid \cM),\\
b_\R^{\theta}&(p=\infty)\end{cases}\label{def of b}
\end{align}
with $\eta_p:=\left[\begin{smallmatrix} 0 & 0 & 0 & -1 \\ 0 & 0 & 1 & 0 \\ 0 & p & 0 & 0 \\ -p & 0 & 0 & 0 \end{smallmatrix}\right] \in \sG(\Q_p)$ being the Atkin-Lehner element, $\bK_\fin:=\bK_0(1)$, and ${\mathcal B}(\lambda,N)=\bigcup_{\pi \in \Pi_{\rm cusp}(\lambda,N)} {\mathcal B}_\pi(\lambda,N)$ is an orthonormal basis of $S_\varrho(\bK_0(N))$ with ${\mathcal B}_\pi(\lambda,N)$ being an orthonormal basis of $V_\pi^{\bK_0(N)}[\varrho]$. The sum is independent of the choice of an orthonormal basis and can be written as ${\mathbb I}^{(s)}(\Pi_{\rm cusp}^{(T_\theta,\Lambda)}(\lambda,N),f_S)$, where for any subset $X\subset \Pi_{\rm cusp}^{(T_\theta,\Lambda)}(\lambda,N)$ we define
\begin{align}
{\mathbb I}^{(s)}(X,f_S)=\frac{1}{[\bK_\fin:\bK_0(N)]}\sum_{\pi \in X} \widehat {f_S}(\pi_{S})
\sum_{\varphi \in {\mathcal B}_\pi(\lambda,N)} \langle E(\phi,s,\Lambda,\mu),R(b)\overline{\varphi_{\sv_\varrho^0}} \rangle\,{B^{T_\theta,\Lambda}(\varphi_{\sv_\varrho^0};b_\R^\theta)}, 
 \label{CompAv-f1}
\end{align}
where $\widehat{f_S}(\pi_S)$ is the spherical Fourier transforms of $f_S$ at $\pi_S:=\otimes_{p\in S}\pi_p$, which is defined as the eigenvalue of the operator $\pi_S(f_S)$ on the $\bK_S=\prod_{p\in S}\bK_p$-fixed vectors $\pi_S^{\bK_S}\cong \C$. For $\bullet\in\{{\rm T},{\rm G}, {\rm Y}, {\rm SK}\}$ and $\ast\in\{{\rm new},{\rm old}\}$, we define ${\mathbb I}^{(s)}(\lambda,N,f_S)^{\bullet}$, ${\mathbb I}^{(s)}(\lambda,N,f_S)^{\ast}$, and ${\mathbb I}^{(s)}(\lambda,N,f_S)^{\bullet,\ast}$ to be ${\mathbb I}^{(s)}(X,f_S)$ with $X=\Pi_{\rm cusp}^{(T_\theta,\Lambda)}(\lambda,N)^{\bullet}$, $\Pi_{\rm cusp}^{(T_\theta,\Lambda)}(\lambda,N)^{\ast}$ and $\Pi_{\rm cusp}^{(T_\theta,\Lambda)}(\lambda,N)^{\bullet,\ast}$, respectively. Then, due to \eqref{SKT}, the average \eqref{Average} has the expression:

\begin{align}
{\mathbb I}^{(s)}(\lambda,N,f_S)&=
{\mathbb I}^{(s)}(\lambda,N,f_S)^{{\rm G},{\rm new}}
+{\mathbb I}^{(s)}(\lambda,N,f_S)^{{\rm Y},{\rm new}}
\label{AverageSum1}
\\
&+{\mathbb I}^{(s)}(\lambda,N,f_S)^{{\rm SK},{\rm new}}+{\mathbb I}^{(s)}(\lambda,N,f_S)^{{\rm old}}.
\notag 
\end{align}
%{\mathbb I}^{(s)}(\lambda,N,f_S)&=
%{\mathbb I}^{(s)}(\lambda,N,f_S)^{{\rm T}}+
%{\mathbb I}^{(s)}(\lambda,N,f_S)^{{\rm SK}.
%{\rm new}}+{\mathbb I}^{(s)}(\lambda,N,f_S)^{{\rm SK},{\rm old}}.\label{averageSum2}

\subsection{A construction of orthonormal basis}
Let $\pi\cong \otimes_{p\leq \infty}\pi_p $ be an element of the set $\Pi_{\rm csup}^{(T_\theta,\Lambda)}(\lambda,N)$ (see \S\ref{sec:BA}). We fix a $(T_\theta,\Lambda)$-Bessel data $\{(\ell_p,\xi_p)\}_{p<\infty}$ of $\pi$ (see \S\ref{sec:BesselMod}) once and for all. Set $\varphi_{\pi}^0(g):=(\varphi_{\pi,\varrho}^0(g)\mid \sv_{\varrho}^0)_\varrho$, which is an element of $V_{\pi}^{\bK_0(N)}$. Since $\pi$ is a unitary representation, all of its factors $\pi_p$ are unitarizable. For each $p<\infty$, we can uniquely fix a $\sG(\Q_p)$-invariant inner product $(\cdot\mid \cdot)_p$ on $V_{\pi_p}$ by demanding $(\xi_{p}\mid \xi_{p})_{p}=1$; for $p=\infty$, we fix a $\sG(\R)$-invarinat inner product so that its pull back to $V_\varrho\hookrightarrow \pi_\infty$ ({\it cf}. \S\ref{sec:AutRep}(i)) coincides with the inner product $(\cdot\mid \cdot)_{\varrho}$ of $V_\varrho$. Let $\varphi_\pi^0$ be the global new form attached to $\{\xi_p\}_{p<\infty}$ (see \S\ref{sec:BesselMod}). Then,  
\begin{align}
\frac{(\varphi\mid \varphi)_{L^2}}{(\varphi_{\pi}^0\mid \varphi_\pi^0)_{L^2}}=(v|v)_{\varrho}\prod_{p<\infty} (v_p|v_p)_{p}
\label{InProd}
\end{align}
for any $\varphi \in V_{\pi}$ that corresponds to $v \otimes (\otimes_{p<\infty} v_p)$ with $v\in V_\varrho$. As below, for $p<\infty$, we fix an orthonormal basis ${\mathcal B}(\pi_p,\bK_0(N\Z_p))$ of $V_{\pi_p}^{\bK_0(N \Z_p)}$ in such a way that
$$
{\mathcal B}(\pi_p,\bK_0(N\Z_p))=\{\xi_{\pi_p}^0\} \quad \text{if $p\nmid N$}.
$$
Then, a pure tensor of the form  
\begin{align}
(\varphi_\pi^0\mid \varphi_\pi^0)_{L^2}^{-1/2}\times 
\otimes_{p<\infty} v_p, \quad v_p\in {\mathcal B}(\pi_p,\bK_0(N\Z_p)),
\label{PureTens}
\end{align}
yield an element $\varphi \in V_{\pi}^{\bK_0(N)}[l]$ such that for $v\in V_\varrho$ the element $\varphi_v\in V_{\pi}$ corresponds to $v$ tensored with \eqref{PureTens}. The set of functions $\varphi$ obtained in this way from \eqref{PureTens} will be denoted by ${\mathcal B}_\pi(\lambda,N)$. If $\pi \in \Pi_{\rm cusp}(\lambda,N)$ does not satisfy condition (A-iv) in \S\ref{sec:BA}, then we fix arbitrary orthonormal basis ${\mathcal B}_{\pi}(\lambda,N)$ of the space $V_{\pi}^{\bK_0(N)}[\lambda]$. Let $\mathcal B(\lambda,N)$ be the union of the sets ${\mathcal B}_\pi(\lambda,N)$ for $\pi \in \Pi_{\rm cusp}(\lambda,N)$; then, by (ii) in \S\ref{sec:AutRep}, the set ${\mathcal B}(\lambda,N)$ is an orthonormal basis of $S_\varrho(\bK_0(N))$. 

\subsection{Computation of the average} \label{sec:CompAv}
Let $\varphi$ correspond to a pure tensor as in \eqref{PureTens}. Then, by \eqref{Bessel-Mod-L1-f1} and by the computations recalled in \S\ref{sec:Gamma} and \S\ref{sec:Ur1}, 
\begin{align}
\langle E(\phi,s,\Lambda,\mu),R(b)\overline{\varphi_{\sv_\varrho^0}} \rangle&=
(\varphi_\pi^0\mid \varphi_\pi^0)_{L^2}^{-1/2}\,
w_D^{-1}\,\overline{R(\varphi^0_\pi,E,\Lambda)}\,\widehat L\left(s+\tfrac{1}{2},\pi,\mu\right) \label{CompAv-f2}\\
&\times (\sv_\varrho^0\mid \sv_\varrho^0)_\varrho (-1)^{l_2} 2^{-2}
|D|^{\frac{1}{2} (4-\frac{l_1+l_2}{2})} \,\prod_{p\mid N} Z_p^{*}(\phi_p,\overline{B_{v_p}};s,\mu_p;b_p),\notag
 \end{align}
where 
$$
Z_p^{*}(\phi_p,\overline{B_{v_p}};s,\mu_p;b_p):={L\left(s+\tfrac{1}{2},\pi_p,\mu_p\right)}^{-1}\,Z_p(\phi_p,\overline{B_{v_p}};s,\mu_p;b_p)
$$
is the normalized local zeta-integral. Moreover, by \eqref{Bessel-Mod-L1-f110} and Lemma \ref{BessPerL3},  
\begin{align}
B^{T_\theta,\Lambda}(\varphi_{\sv_\varrho^0};b_\R^\theta)&=\langle \varphi_\pi^0\mid \varphi_\pi^0\rangle_{L^2}^{-1/2}\,w_{D}^{-1}R(\varphi_\pi^0,E,\Lambda)\,(B^{T_\theta}_{\varrho}(b_\R^\theta) \mid \sv_\varrho^0)_\varrho \,\prod_{p\mid N}\ell_p(v_p).
 \label{CompAv-f3}
\end{align}
Note that $B_{\varrho}^{T_\theta}(b_\R^\theta)=(-1)^{\frac{l_1+l_2}{2}} \exp(-2\pi \sqrt{|D|})\,\sv_\varrho^{0}$ by \eqref{AdeleFE-W} and \eqref{Def-b}. From \eqref{CompAv-f1}, \eqref{CompAv-f2} and \eqref{CompAv-f3}, we get

\begin{prop}\label{Main-Prp1} Let $f_{S}\in \cH_S$. Then, ${\mathbb I}^{(s)}(\lambda,N,f_S)$ equals to
\begin{align} 
&\frac{2^{-2}|D|^{\frac{1}{2}(3-\frac{l_1+l_2}{2})}
e^{-2\pi \sqrt{|D|}}}{w_D^2[\bK_\fin:\bK_0(N)]} \times M^{s-6}\zeta_M(1)\zeta_M(4)\tilde \mu(2D)G(\tilde \mu)
\notag
\\
&\quad \times (\sv_\varrho^0\mid \sv_\varrho^0)_\varrho^{2}\sum_{\pi \in \Pi_{\rm cusp}^{(T_\theta,\Lambda)}(\lambda,N)} \widehat{f_S}(\pi_S)\frac{|R(\varphi^0_\pi,E,\Lambda)|^2}{( \varphi_\pi^0|\varphi_\pi^0)_{L^2}} \widehat L\left(s+\tfrac{1}{2},\pi,\mu\right) 
\label{SpectralAveage-f}
\\
&\quad \times \prod_{p<\infty} {\mathbb I}_{\pi_p, \bK_0(N\Z_p)}^{(s)}(\xi_{p}, \phi_p,\Lambda_p,\mu_p;b_p),
\notag
\end{align}
where $\xi_p$ for $p<\infty$ is from the fixed $(T_\theta,\Lambda)$-Bessel data $\{(\ell_p,\xi_p)\}_{p<\infty}$ of $\pi$, and 
\begin{align}
{\mathbb I}^{(s)}_{\pi_p,\bK_0(N\Z_p)}(\xi_p, \phi_p,\Lambda_p,\mu_p;b_p):=\sum_{v \in {\mathcal B}(\pi_p,\bK_0(N\Z_p)} Z^{*}_p(\phi_p,\overline{B_{v_p}};s,\mu_p,b_p)\,{\ell_p(v
)}.
\label{SpectralAveage-LclPr}
\end{align}
We also use the notation $\zeta_M(s)$ to denote $\prod_{p\mid M}(1-p^{-s})^{-1}$. 
\end{prop}
Note that if $\pi_p$ with $p\nmid N$ is unramified and $(\ell_{\pi_p}^0,\xi_{\pi_p}^0)$ is the unramified Bessel datum of $\pi_p$, then \eqref{SpectralAveage-LclPr} is $1$. For other cases, we compute \eqref{SpectralAveage-LclPr} in \S\ref{sec:Local period} completely. 
Substituting them, we obtain 

\begin{thm} \label{Main-T1}
Let $(\lambda=(l_1,l_2),N)$, $\mu$, $\tilde \mu$ and $M$ be as in \S\ref{sec:BA}. Let $S$ be a finite set of prime numbers relatively prime to $DMN$. Then, for any $f_S=\otimes_{p\in S}f_p \in \cH_S$, the value ${\mathbb I}^{(s)}(\lambda,N,f_S)$ equals
\begin{align}
&\frac{2^{-2}|D|^{\frac{1}{2}(3-\frac{l_1+l_2}{2})}e^{-2\pi \sqrt{|D|}}}{w_D^2[\bK_\fin:\bK_0(N)]}
 \cM^{-4} M^{s-6}\zeta_M(1)\zeta_{M\cM}(4)\tilde \mu(2D)G(\tilde \mu)
N^{s-1}\tilde \mu^{-1}(N)\prod_{p\mid N}(1+p^{-2})^{-1}
\label{SpectralAverage} \\
&\times (\sv_\varrho^0\mid \sv_\varrho^0)_\varrho^{2}\sum_{\pi \in \Pi_{\rm cusp}^{(T_\theta,\Lambda)}(\lambda,N)} \widehat{f_S}(\pi_S)\frac{|R(\varphi^0_\pi,E,\Lambda)|^2}{( \varphi_\pi^0|\varphi_\pi^0)_{L^2}} \widehat L\left(s+\tfrac{1}{2},\pi,\mu\right)\,\ft(\pi,\mu)\notag  
\end{align}
with $\ft(\pi,\mu)=\prod_{p\mid N} \ft(\pi_p,\mu_p)$, 
\begin{align}
\ft(\pi_p,\mu_p):=
\begin{cases}
1 \quad (\text{$p\mid N_\pi$ and $\pi_p$ is of type VIb}),  \\
2 \quad (\text{$p\mid N_\pi$ and $\pi_p$ is of type IIIa}), \\
2(p-1)p^{-5}L(1,\pi_p,{\rm Std})\,\{1+\mu_p^2(p)-\tfrac{\mu_p(p)}{p+1}\tr(p^{-1}T_p+\eta_p|\pi_p^{\bK_0(p\Z_p)})\} \quad (\text{$p\mid \frac{N}{N_\pi}$}).
\end{cases}\label{localfactor}
\end{align}
\end{thm}
To refine \eqref{SpectralAverage} by considering ${\mathbb I}^{(s)}(\lambda,N,f_S)^{\bullet,{\rm new}}$ ($\bullet \in \{{\rm G},{\rm Y},{\rm SK}\}$), we use an explicit formula of the quantity $\frac{|R(\varphi^0_\pi,E,\Lambda)|^2}{( \varphi_\pi^0|\varphi_\pi^0)_{L^2}}$ first proved by Dickson, Pitale, Saha, and Schmidt (\cite[Theorems 1.13 and 3.13]{DPSS}) when $\Lambda$ is unramified and $l_1=l_2$ under the refined GGP conjecture for Bessel periods posed by Liu(\cite{Liu}). Owing to a theorem by Furusawa and Morimoto \cite[Theorem 8.1]{FurusawaMorimoto3}, the formula in \cite[Theorems 1.13 and 3.13]{DPSS} is extended to vector-valued forms unconditionally so that it can be applied to \eqref{SpectralAverage}; by Theorem \ref{localGPPeriod}, we can further extend the formula to any ramified $\Lambda$ with square-free conductor.  

\begin{thm} \label{Main-T2}
Let the notations and the assumptions be as in Theorem \ref{Main-T1}. 
\begin{itemize}
\item[(i)]
${\mathbb I}^{(s)}(\lambda,N,f_S)^{{\rm G},{\rm new}}$ equals
\begin{align*} 
&\frac{2^{\# S(N)+2l_1-6}|D|^{\frac{l_1+l_2+2}{4}}e^{-2\pi \sqrt{|D|}}}{[\bK_\fin:\bK_0(N)]}\\
&\times \Mcal^{-8}M^{s-6}\zeta_M(1)\zeta_{\Mcal M }(4)\tilde \mu(2D)G(\tilde \mu)
N^{s-1}\tilde \mu^{-1}(N)\prod_{p\in S(N)}(1+p^{-1}) \\
&\times \sum_{\pi \in \Pi_{\rm cusp}^{(T_\theta,\Lambda)}(\lambda,N)^{\rm G},{\rm new}}\widehat{f_S}(\pi_S)\frac{\widehat L(\frac{1}{2},\pi\times{\mathcal AI}(\Lambda^{-1}))}{\widehat L(1,\pi;{\rm Ad})}\widehat L\left(s+\tfrac{1}{2},\pi,\mu\right),
\end{align*}
where $S(N)$ is the set of all the prime divisors of $N$, and ${\mathcal AI}(\Lambda^{-1})$ is the automorphic induction to $\GL_2(\A)$ from the character $\Lambda^{-1}$ of $\A_E^\times/E^\times$.

\item[(ii)]
If $N$ has even number of prime divisors, then
${\mathbb I}^{(s)}(\lambda,N,f_S)^{{\rm Y},{\rm new}}=0$. If $N$ has odd number of prime divisors, then
${\mathbb I}^{(s)}(\lambda,N,f_S)^{{\rm Y},{\rm new}}$ equals
{\allowdisplaybreaks\begin{align*} 
&\frac{2^{\# S(N)+2l_1-7}|D|^{\frac{l_1+l_2+2}{4}}e^{-2\pi \sqrt{|D|}}}{[\bK_\fin:\bK_0(N)]}\\
&\times \Mcal^{-8}M^{s-6}\zeta_M(1)\zeta_{\Mcal M}(4)\tilde \mu(2D)G(\tilde \mu)
N^{s-1}\tilde \mu^{-1}(N)\prod_{p\in S(N)}(1+p^{-1}) \\
&\times \sum_{\substack{\pi_1\in\Pi_{{\bf PGL}_2,{\rm cusp}}(l_1+l_2-2,N)^{\rm new}\\\pi_2\in\Pi_{{\bf PGL}_2,{\rm cusp}}(l_1-l_2+2,N)^{\rm new}}}
\widehat{f_S}({\rm Y}(\pi_1,\pi_2)_S)\frac{\widehat L(\frac{1}{2},\pi_1\times{\mathcal AI}(\Lambda^{-1}))\widehat L(\frac{1}{2},\pi_2\times{\mathcal AI}(\Lambda^{-1}))}{\widehat L(1,\pi_1;{\rm Ad})\widehat L(1,\pi_2;{\rm Ad})\widehat L(1,\pi_1\times\pi_2)}\\
&\hspace{80mm}\times\widehat L\left(s+\tfrac{1}{2},\pi_1\times\mu\right)\widehat L\left(s+\tfrac{1}{2},\pi_2\times\mu\right),
\end{align*}}where $\Pi_{{\bf PGL}_2,{\rm cusp}}(k,N)^{\rm new}$ is the set of all irreducible cuspidal representations of ${\bf PGL}_2(\A)$ associated to holomorphic newforms of weight $k$ and level $N$ and ${\rm Y}(\pi_1,\pi_2)\in \Pi_{\rm cusp}(\lambda,N)^{\rm new}$ denotes the Yoshida lift of $\pi_1$ and $\pi_2$.
\item[(iii)] If $\Lambda\ne\cchi$ or $l_1>l_2$, then ${\mathbb I}^{(s)}(\lambda,N,f_S)^{{\rm SK},{\rm new}}=0$. If $\Lambda=\cchi$ and $l_1=l_2\,(=:l)$, then ${\mathbb I}^{(s)}((l,l),N,f_S)^{{\rm SK},{\rm new}}$ equals
{\allowdisplaybreaks\begin{align*}
&\frac{3\cdot 2^{2l-3}|D|^{(l+2)/2}e^{-2\pi \sqrt{|D|}}}{[\bK_\fin:\bK_0(N)]}\cdot\frac{s}{4\pi}\\
&\times M^{s-6}\zeta_M(1)\zeta_M(4)\tilde \mu(2D)G(\tilde \mu)
N^{s-1}\tilde \mu^{-1}(N)\prod_{p\in S(N)}(1+p^{-1}) \cdot \widehat L(s+1,\mu)\widehat L(s,\mu)\\
&\times \sum_{\pi_0 \in \Pi_{{\bf PGL}_2,{\rm cusp}}^{(T_\theta,1)}(2l-2,N)^{\rm new}} \widehat{f_S}({\rm SK}(\pi_0)_S)\frac{\widehat L(\frac{1}{2},\pi_0 \times \chi_{D})\widehat L(1,\chi_D)^2}{\widehat L(\frac{3}{2},\pi_0)\widehat L(1,\pi_0;{\rm Ad})}\widehat L\left(s+\tfrac{1}{2},\pi_0\times\mu\right), 
\end{align*}}where $\Pi_{{\bf PGL}_2,{\rm cusp}}^{(T_\theta,1)}(2l-2,N)^{\rm new}$ is the set of all $\pi_0\in\Pi_{{\bf PGL}_2,{\rm cusp}}(2l-2,N)^{\rm new}$ such that the Saito-Kurokawa lift ${\rm SK}(\pi_0)$ of $\pi_0$ has the $(T_{\theta},\cchi)$-Bessel model, and $\chi_D$ is the Kronecker character of modulo $D$.
\end{itemize}
\end{thm}
\begin{proof}
The equation (i) and (ii) is a direct corollary to Theorem \ref{Main-T1} and \cite[Theorem 8.1]{FurusawaMorimoto}(for the scalar case we refer to \cite[Theorems 1.13 and 3.14]{DPSS}), extended to the ramified $\Lambda$ by Theorem \ref{localGPPeriod};
note that the polynomial $Q_{S,\varrho}$ (with $S=T_\theta$) in \cite[(8.2.16)]{FurusawaMorimoto} equals $(-1)^{\frac{l_1-l_2}{2}}(\frac{2}{\sqrt{|D|}})^{\frac{l_1+l_2}{2}}\sv_\varrho^\theta$, our $R(\varphi,E,\Lambda)$ equals the quantity $w_D(\frac{\sqrt{|D|}}{2})^{-\frac{l_1+l_2}{2}} \, \mathcal B_\Lambda(\varphi;E)/(Q_{S,\varrho},Q_{S,\varrho})_{l_1-l_2}$ defined by \cite[(8.2.17)]{FurusawaMorimoto}, and $(Q_{S,\varrho},Q_{S,\varrho})_{l_1-l_2}=(\frac{|D|}{4})^{-\frac{l_1+l_2}{2}}(\sv_{\varrho}^0,\sv_{\varrho}^0)_{l_1-l_2}=(\frac{|D|}{4})^{-\frac{l_1+l_2}{2}}(\sv_\varrho^0\mid \sv_\varrho^0)_\varrho$. We also note the formula
\begin{align*}
\widehat L(s,\pi,\mu)&=\widehat L(s,\pi_1\times\mu)\widehat L(s,\pi_2\times\mu),
\end{align*}
where $\pi=Y(\pi_1,\pi_2)$. This is obtained immediately by comparing the local factors. It is noted in \cite[p.296]{DPSS} that each local representation $\pi_p$  of $\pi\cong \otimes_{p}\pi_p\in\Pi_{\rm cusp}^{(T_\theta,\Lambda)}(\lambda,N)^{\rm Y,new}$ for $p\mid N$ is of type VIb.
Finally, we check the statement (iii). As quoted in \cite[Theorem 3.11]{DPSS}, the $(T_\theta,\Lambda)$-Bessel periods of the Saito-Kurokawa lifts are zero unless $\Lambda$ is trivial due to Qiu. By comparing the local $L$-functions (\cite[ Theorem 5.2(ii)]{Schmidt2007}, {\it cf}. \cite[Theorem 3.1]{PS83}), 
\begin{align}
\widehat L(s,\pi,\mu)=\frac{1}{4\pi}\left(s-1/2\right)\widehat L(s,\pi_0\times\mu) \widehat L(s+1/2,\mu) \widehat L(s-1/2,\mu)\label{L function for SK lift},
\end{align}
where $\pi={\rm SK}(\pi_0)$. By \cite{Schmidt2007} and the definition, the set $\Pi_{{\rm cusp}}^{(T_\theta,1)}((l,l),N)^{{\rm SK}}$ corresponds bijectively to the set $\Pi_{{\bf PGL}_2,{\rm cusp}}^{(T_\theta,1)}(2l-2,N)$ via the Saito-Kurokawa lifting if $l\ge 3$. Note that the condition $l\ge 3$ is ensured by assumption (A-ii) in \S \ref{sec:BA}.
\end{proof}

By \eqref{AverageSum1}, we have that $\II^{(s)}(\lambda,N,f_S)-\II^{(s)}(\lambda,N,f_S)^{\rm G,new}$ is the sum of $\II^{(s)}(\lambda,N,f_S)^{\rm Y,new}$, $\II^{(s)}(\lambda,N,f_S)^{\rm SK,new}$ and $\II^{(s)}(\lambda,N,f_S)^{\rm old}$. Our main focus is their values at $s=0$. For those, we have the following upper bounds in the level aspect. For $f_S\in \cH_S$, let $\|f_S\|_1:=\int_{\sG(\Q_S)}|f_S(g_S)|\,\d g_S$ denote its $L^1$-norm. 

\begin{thm}\label{AsymptoticFormula}
Let $l$, $\mu$, $\widetilde \mu$, $M$, S be as in Theorem \ref{Main-T1}.
\begin{itemize}
\item[(i)]
Let $\lambda $ be fixed as before and $N$ an inert prime in $E/\Q$. Then, there exist constants $C_{D,\lambda,\mu},C_{D,\lambda,\mu}'>0$ depending only on $D$, $l$, $\mu$, and $f_S$ such that
\begin{align}
|{\mathbb I}^{(0)}(\lambda,N,f_S)^{{\rm old}}|&<C_{D,\lambda,\mu}\|f_S||_1 N^{-8}\label{Upperbound,old},\\
|{\mathbb I}^{(0)}(\lambda,N,f_S)^{{\rm Y,new}}|&<C_{D,\lambda,\mu}'\|f_S\|_1 N^{-\frac{3}{2}}.\label{Upperbound,Ynew}
\end{align}
\item[(ii)]
Let $N$ be $1$ or an inert prime in $E/\Q$.
When $\mu\ne{\bf 1}$, ${\mathbb I}^{(0)}(l,N,f_S)^{{\rm SK,new}}=0.$
When $\mu={\bf 1}$, there exists a constant $C>0$ such that
\begin{align}
|{\mathbb I}^{(0)}((l,l),N,f_S)^{{\rm SK,new}}|<C|D|^{\frac{l}{2}}l^{\frac{1}{2}}\|f_S\|_1 N^{-\frac{3}{2}}.\label{Upperbound,SKnew}
\end{align}
\end{itemize}
\end{thm}
\begin{proof} We have the inequality $|\widehat f_S(\pi_S)|\leq \|f_S\|_1$ for all irreducible unitary representation $\pi_S$ of $\sG(\Q_S)$. Because $N$ is a prime, one can check that the summation range of ${\mathbb I}^{(0)}(\lambda,N,f_S)^{{\rm old}}$ is $\Pi_{\rm cusp}^{(T_\theta,\chi)}(l,N)^{{\rm old}}=\Pi_{\rm cusp}^{(T_\theta,\chi)}(\lambda,1)$, which is independent of $N$. From the formula of $\ft(\pi_p,\mu_p)$ in (\ref{localfactor}) for $p=N$, combined with the temperedness of $\pi_N$ due to \cite{Weissauer} and the matrix of $T_{1,0}$ in \cite[Table 3]{PitaleSchmidt2}, we get $\ft(\pi_N,\mu_N)=O(N^{-4})$. 
This and the equation $[\bK_{\bf f}:\bK_0(N)]=N^3(1+N^{-2})(1+N^{-1})$ for prime $N$ yield the bound \eqref{Upperbound,old}. 

Next we treat the average for Yoshida lifts. For $\pi_1\in\Pi_{{\bf PGL}_2,{\rm cusp}}(l_1+l_2-2,N)^{\rm new}$ and $\pi_2\in\Pi_{{\bf PGL}_2,{\rm cusp}}(l_1-l_2+2,N)^{\rm new}$, we need the lower bound 
$$L(1,\pi_1\times \pi_2) \gg_{l_1,l_2} \exp(-C \sqrt{\log N}) 
$$
uniform in $N$. This is a special case of Lemma \ref{RSLbound}, because $\pi_1$ and $\pi_2$ are everywhere tempered by Deligne's estimate and by the fact that the local $p$-components of $\pi_i$ for $p\mid N$ are the (twisted) Steinberg representation, which is tempered. Note that $\pi_1\not\cong \pi_2$ due to the weight condition. Now, We deduce the inequality \eqref{Upperbound,Ynew} bounding the sum from above by a product of the average considered in \cite[Theorem 1.1]{FeigonWhitehouse}; to do this, we invoke the subconvexity bound
\begin{align}
L(\tfrac{1}{2},\pi)=O(C(\pi)^{\frac{1}{4}-\delta}),\quad(\exists \, \delta>0)
\label{subconvexity}
\end{align}
for automorphic cuspidal representations $\pi$ of $\GL_2(\A)$ (\cite{MichelVenkatesh}) and the non-negativity of $L(\frac{1}{2},\pi_i\times\mathcal{AI}(\Lambda^{-1})=L(\frac{1}{2}, {\rm BC}_{E/\Q}(\pi_i)\otimes \Lambda^{-1})$ for $i=1,2$ due to \cite{JacquetChen}. 

Let us prove \eqref{Upperbound,SKnew}; by Theorem \ref{Main-T2} (ii), we may assume $\lambda=(l,l)$. Suppose $\mu\not={\bf 1}$; then, by Theorem \ref{Main-T2} (ii), we have that ${\mathbb I}^{(s)}((l,l),N,f_S)^{{\rm SK,new}}\times s^{-1}$ is entire, hence $\mathbb I^{(0)}((l,l),N,f_S)^{\rm SK, new}=0$
 because $\widehat{L}(s,\mu)$, as well as $\widehat L(s,\pi_0\times \chi_D)$, is entire. In the rest of the proof, we assume $\mu={\bf 1}$. Then, $s\widehat L(s,\mu)\widehat L(s+1,\mu)$ has a simple pole at $s=0$. By \cite[Theorem 3.1 and Table 2]{Schmidt2005}, ${\rm SK}(\pi_0)$ with $\pi_0\in \Pi_{{\rm cusp},{\bf PGL_2}}(2l-2,N)$ has the global $(T_\theta,\Lambda)$-Bessel model only if the sign of the functional equation of $\pi_0$ is $-1$ so that $\widehat L(s+1/2,\pi_0)$ has a zero at $s=0$. Hence, theorem \ref{Main-T2} (ii) gives us the following majorant of ${\mathbb I}^{(0)}((l,l),N,f_S)^{{\rm SK,new}}$.
\begin{align*}2^{3}\pi^{\frac{5}{2}}(l-1)^{-2}\frac{\Gamma(l)\|f_S\|_1}{\Gamma(l-\frac{1}{2})}|D|^{\frac{l}{2}}[\bK_\fin:\bK_{0}(N)]^{-1}\sum_{\pi_0\in\Pi_{{\bf PGL}_2,{\rm cusp}}^{(T_\theta,{\bf 1})}(2l-2,N)}\frac{L(\frac{1}{2},\pi_0\times\chi_D)\,|L'(\frac{1}{2},\pi_0)|}{L(1,\pi_0;{\rm Ad})}.
\end{align*}
 To estimate this, we invoke a subconvexity bound
$L({1}/{2},\pi_0\times \chi_{D})=O((l^2N)^{\frac{1}{4}-\delta})$ for some $\delta>0$ from \eqref{subconvexity} and a lower bound $L(1,\pi_0;{\rm Ad})>c_0\exp(-c_1\sqrt{\log(1+l^2N)})$ for some constant $c_0,c_1>0$ 
which is known by \cite[Theorem 0.1]{HoffsteinLockhart} (see also the remark after \cite[Corollary 7]{Li}). The lower bound implies $L(1,\pi_0;{\rm Ad})^{-1}=O((l^2N)^{2\delta})$. Moreover, by a common argument, a subconvexity bound for the central derivative $L'(1/2,\pi_0)=O((l^2N)^{\frac{1}{4}-\delta})$ is derived from the bound $L(1/2,\pi_0|\cdot|_\A^{it})=O((l^2N(1+|t|)^{1/4-\delta})$ that follows from \eqref{subconvexity}. 
By the uniform bound $\#\Pi_{{\bf PGL}_2,\rm cusp}(2l-2,N)=O(lN)$ and the asymptotic $\frac{\Gamma(l)}{\Gamma(l-\frac{1}{2})}\sim l^\frac{1}{2}\ (l\to\infty)$, the bound \eqref{Upperbound,SKnew} follows.
\end{proof}

\begin{lem}\label{RSLbound} Let $\pi_1$ and $\pi_2$ be irreducible cuspidal automorphic representations of ${\bf PGL}_{n}$ such that $C(\pi_1),C(\pi_2)\leq Q$ with $Q>2$. We assume that $\pi_1\not\cong \pi_2$ and both of them are self-dual and tempered everywhere. Then, 
$$
L(1,\pi_1\times \pi_2)\gg \exp(-C\sqrt{\log Q}) 
$$
with an absolute constant $C>0$. 
\end{lem}
\begin{proof} We recall the argument indicated in \cite{Sarnak} (attributed originally to \cite{Moreno}), which eliminates a possibility of Siegel zeros of the $L$-function $L(s,\pi_1\times \pi_2)$. Fix $t\in\R$. Then, $L(s,\Pi\times\widetilde{\Pi})$ with $\Pi$ being the isobaric sum
$$\Pi:=\pi_1 \boxplus (\pi_1\times|\cdot |_\A^{it}) \boxplus (\pi_1\times|\cdot |_\A^{-it}) \boxplus \pi_2 \boxplus (\pi_2\times|\cdot |_\A^{it}) \boxplus (\pi_2\times|\cdot |_\A^{-it})$$
is a Dirichlet series with non-negative coefficients (\cite[Lemma a]{HoffsteinRamakrishnan}). Moreover, by (\cite[Proposition 9.4]{PSShalika}) and due to $\pi_1$ and $\pi_2$ being self-dual, $L(s,\Pi\times\widetilde{\Pi})$ is expressed as
$$
L(s,\Pi\times\widetilde{\Pi})=\prod_{1\le i,j\le 2}\left[
{\substack{L(s,\pi_i\times\pi_j)^3 L(s-it,\pi_i\times\pi_j)^2 L(s+it,\pi_i\times\pi_j)^2 \\ \times L(s-2it,\pi_i\times\pi_j) L(s+2it,\pi_i\times\pi_j)}}\right]^2,
$$
which shows that $L(s,\Pi\times\widetilde{\Pi})$ has a pole of order $6$ at $s=1$ (due to $\pi_1\ncong\pi_2$), 
and has a zero of order $8$ at $s=\sigma$ if $L(\sigma+it,\pi_i\times\pi_j)=0$ $(i,j=1,2)$. 
Hence, by \cite[Lemma, p.178]{GHL}, one can show that $L(s,\pi_1\times\pi_2)$ has no zeros on the interval $(1-\frac{C_0}{\log M)}, 1)$ for some constant $C_0>0$ with $M=(1+|t|)^{24}C(\pi_2\times\pi_2)^{18} C(\pi_1\times\pi_1)^9 C(\pi_2\times \pi_2)^2$, where $C(\pi_i\times \pi_j)$ is the analytic conductor of $L(s,\pi_i\times \pi_j)$. By \cite[lemma b]{HoffsteinRamakrishnan}, we have $C(\pi_i\times \pi_j)\ll Q^{2n}$. Thus, for an absolute constant $C_0'>0$, $L(s,\pi_1\times\pi_2)$ is zero-free on the region $1-\frac{C_0'}{\log Q(1+|\Im(s)|)}<\Re(s)<1$. Now, we apply the argument of \cite[Corollary 7]{Li}) to the $L$-function $L(s,\pi_1\times \pi_2)$. Since $\pi_1,\pi_2$ are assumed to be tempered everywhere, the optimal bound of the lambda function $\Lambda_{\mathcal A}(n)$ for $\mathcal A=\pi_1\times \pi_2$ is available so that the automorphy of $\pi_1\times \pi_2$ is not necessary, which simplifies the proof to get the lower bound of  $L(1,\pi_1\times\pi_2)$.     
\end{proof}

As a consequence of Theorem \ref{AsymptoticFormula}, we get  
\begin{align}
\II^{(0)}(\lambda,N, f_S)^{\rm {G,new}}=\II^{(0)}(\lambda,N,f_S)+O_{\Lambda,\lambda,\mu}(\|f_S\|_1 N^{-\frac{3}{2}}).    
\label{AsymtoticFormulaCor}
\end{align}

\section{Computation of local zeta-integrals for $p$-adic fields} \label{sec:CompLZ}

In this section, let $F$ be a non-archimedean local field of characteristic 0, $\cO$ the integer ring of $F$, $\fp$ the maximal ideal of $\cO$, $\varpi$ a generator of $\fp$ and $q=\#(\cO/\fp)$. Let $|\cdot|$ denote the normalized absolute value of $F$, i.e., $|\varpi|=q^{-1}$. 
Fix a non-trivial additive character $\psi:F\longrightarrow \C^1$ with ${\rm cond}(\psi):=\min\{n\in\Z\ ;\ \psi|_{\varpi^n\cO}=1\}=0$.

We compute the local zeta-integral a la Piatetski-Shapiro (\cite{PS97}) for several representations, taking particular test functions; as a result, we determine the local $L$-factors and the local $\varepsilon$-factors in \cite{PS97} to confirm that they coincides with the expected ones listed in \cite[Tables A8,A9]{RobertsSchmidt}. As explained in \S\ref{sec:BA}, for a particular global application in mind, we only deal with representations of types I, IIb, IIIa and VIb (but allowing the central characters to be non-trivial when we are concerned with newvectors.)

\subsection{Local zeta integral for Bessel models}
We first review some generalities on local zeta integrals and then recall results from  \cite{PitaleSchmidt} on explicit formulas of Bessel functions for Iwahori spherical representations of $\sG$, which are possible  local components of $\pi\cong\otimes_p\pi\in\Pi_{\rm cusp}^{(T_\theta,\Lambda)}(\varrho,N)$. 

Let $K=\sG(\cO)$ be a standard maximal compact subgroup of $\sG(F)$ and $K_0(\fp^e):=\{\left[\begin{smallmatrix} A & B \\ C & D \end{smallmatrix}\right]\in K\ |\ C\equiv0\ ({\rm mod }\ \fp^e)\}$ be the congruence subgroup of level $\fp^e$.

For a symmetric matrix $T=[\begin{smallmatrix} \sa & \sb/2 \\ \sb/2 & \sc \end{smallmatrix}]\in {\rm Sym}_2(F)$ such that $\sd:=\sb^2-4\sa\sc\ne0$, set $\xi=[\begin{smallmatrix} \sb/2 & \sc \\ -\sa & -\sb/2 \end{smallmatrix}]$. Let $L:=F\oplus F\xi$ be the two-dimensional $F$-algebra, $\cO_L$ be its integer ring, and $\fp_L:=\fp\fo_L$. Define the additive character $\psi_L$ on $L$ by $\psi_L:=\psi\circ\tr_{L/F}$. The maps
{\allowdisplaybreaks\begin{align}
&L\owns x+y\xi\longmapsto x+\tfrac{\sqrt{\sd}}{2}y \in F(\sqrt{\sd})\quad(\sd\notin(F^\times)^2),\notag\\
&L\owns x+y\xi\longmapsto \left(x+\tfrac{\sqrt{\sd}}{2}y , x-\tfrac{\sqrt{\sd}}{2}y\right)\in F\oplus F\quad(\sd\in(F^\times)^2)\notag
\end{align}}are isomorphisms of $F$-algebras. We assume that 
\begin{itemize}
\item[(i)] $\sa\in\cO^\times$ and $\sb, \sc\in\cO$.
\item[(ii)] If $\sd\notin(F^\times)^2$, then $\sqrt{\sd}\cO_L$ is the different ideal of $L/F$.
\item[(iii)] If $\sd\in(F^\times)^2$, then $\sd\in\cO^\times$.
\end{itemize}
We can check that
$$L^\times=\left\{\tau\in\GL_2(F)\ \middle|\ {}^t\tau T\tau=\det \tau\cdot T\right\}.$$
The map $L\owns \tau\mapsto\tau^{\dagger}\in L$ is the non-trivial $F$-automorphism on $L$. We set $\theta_0=\sb/2-\xi\in\cO_L$, then $\{\sa,\theta_0\}$ is an $\cO$-basis of $\cO_L$.
Let $\langle\cdot,\cdot\rangle$ denote the symplectic form on $L^2$ over $F$ defined by
$$\langle x,y\rangle:=\tr_{L/F}(-(2\xi)^{-1}(x_1y_2-x_2y_1)),\quad x=\left[\begin{smallmatrix} x_1 \\ x_2 \end{smallmatrix}\right], y=\left[\begin{smallmatrix} y_1 \\ y_2 \end{smallmatrix}\right]\in L^2.$$
We set
{\allowdisplaybreaks\begin{align*}
G^\#(F)&:=\left\{g\in \GL_2(L)\ \middle|\ \det g\in F^\times\right\},\\
B^\#(F)&:=\left\{\left[\begin{smallmatrix} \tau & \beta \\ 0 & a \tau^\dagger \end{smallmatrix}\right]\ \middle|\ \tau\in L^\times, \beta\in L, a\in F^\times\right\},\\
K^\#&:=G^\#(F)\cap\GL_2(\cO_L).\\
K_0^\#(\fp^e)&:=G^\#(F)\cap\left[\begin{smallmatrix} \cO_L & \cO_L \\ \fp_L^e & \cO_L \end{smallmatrix}\right].
\end{align*}}Since $\langle g^\# x,g^\# y\rangle=\det g^\#\langle x,y\rangle$ for any $x,y\in L^2$ and $g^\#\in\sG(F)$, we get a natural embedding $G^\#(F)\owns g\mapsto\iota(g^\#)\in\sG(F)$. More precisely, $\iota(g)$ is the representation matrix of with respect to an $F$-basis 
$\{[\begin{smallmatrix} \sa \\ 0 \end{smallmatrix}],[\begin{smallmatrix} \theta_0 \\ 0 \end{smallmatrix}],[\begin{smallmatrix} 0 \\ -\sa^{-1} \theta_0^\dagger \end{smallmatrix}],[\begin{smallmatrix} 0 \\ 1 \end{smallmatrix}]\}$. A computation yields 
\begin{gather}
\iota\left( \left[\begin{smallmatrix} \tau & 0 \\ 0 & a \tau^\dagger \end{smallmatrix}\right] \right)=\sm(\tau,a\det\tau),\notag\\
\iota\left( \left[\begin{smallmatrix} 1 & \beta \\ 0 & 1 \end{smallmatrix}\right] \right)=\left[\begin{smallmatrix} 1_2 & X_\beta \\ 0 & 1_2 \end{smallmatrix}\right],\quad \beta=\beta_2\sa+\beta_3\theta_0\in L,\ X_\beta:=\left[\begin{smallmatrix} -\sa^{-1}\sb \beta_2-\sa^{-1}\sc\beta_3& \beta_2 \\ \beta_2 & \beta_3 \end{smallmatrix}\right],\label{iota}\\
\iota\left( \left[\begin{smallmatrix} 0 & 1 \\ 1 & 0 \end{smallmatrix}\right] \right)=\left[\begin{smallmatrix} 0 & 0 & -\sa^{-2}\sb & \sa^{-1} \\ 0 & 0 & \sa^{-1} & 0 \\ 0 & \sa & 0 & 0 \\ \sa & \sb & 0 & 0 \end{smallmatrix}\right].\notag
\end{gather}
We note that $K^{\#}=\iota^{-1}(K)$ and $K_0^{\#}(\fp^e)=\iota^{-1}(K_0(\fp^e))$.
We call 
$$R:=\left\{\iota\left( \left[\begin{smallmatrix} \tau & 0 \\ 0 & \tau^\dagger \end{smallmatrix}\right] \right)n\ \middle| \tau\in L^\times, n\in \sN(F)\right\}
$$
 the Bessel subgroup of $\sG(F)$ (with respect to $T$). For a character $\Lambda:L^\times \to\C^1$, we can check that the map
$$R\owns \iota\left(\left[\begin{smallmatrix} \tau & 0 \\ 0 & \tau^\dagger \end{smallmatrix}\right]\right)\sn(X)\mapsto \Lambda(\tau)\psi(\tr{(TX)})\in\C^1$$
defines a character on $R$. We denote this character by $\Lambda\otimes\psi_T$. 

Let $(\pi,V_{\pi})$ be an irreducible admissible representation of $\sG(F)$. Let $(V_{\pi}^*)^{T,\Lambda}$ denote the space of all $\C$-linear forms $\ell:V_{\pi}\longrightarrow \C$ that satisfies 
$$
\ell(\pi(r)\xi)=\Lambda\otimes\psi_T(r)\,\ell(\xi), \quad \xi\in V_{\pi},\, r\in R.
$$
Then, it holds that $\dim_{\C}(V_\pi^*)^{T,\Lambda}\leq 1$ from (\cite{Novodvorski},\cite{{TaklooBinghashPrasad}}). We say that $\pi$ has a local $(T,\Lambda)$-Bessel model if $(V_{\pi}^*)^{T,\Lambda}\not=(0)$. In this case, we can define the local $(T,\Lambda)$-Bessel model of $\pi$ as in \S \ref{sec:BesselMod}. Moreover, when $\pi$ is spherical, we define  
 the unramified $(T,\Lambda)$- Bessel datum $(\ell_\pi^0,\xi_\pi^0)\in (V_\pi^*)^{T,\Lambda}\times V_\pi^{K}$ for $\pi$ as in \S \ref{sec:BesselMod}, so that 
$$
 B_\pi^0(g)=\ell_{\pi}^0(\pi(g) \xi_{\pi}^0), \quad g\in \sG(\Q_p).
$$
is the unramified Bessel function such that $B_\pi^0(1_4)=1$. Recall that all irreducible admissible representations of $\sG(F)$ that admit local Bessel models are classified in \cite{RobertsSchmidt}, and the result is conveniently summarized in \cite[Table 2]{PitaleSchmidt}. 
For our cases, as we mentioned in \S \ref{sec:BesselMod}, we may assume that $\pi$ is of type I ,IIb, IIIa, or VIb. We refer some of its properties in (\cite[Table 2]{PitaleSchmidt});

\begin{align}
\label{Table1}
\begin{array}{|c|c|c|c|c|}
\hline
 {{\rm type}} & {\pi} & {\dim V_\pi^{K}} & {\dim V_\pi^{K_0(\fp)}} & {{\rm Cent.\ Char.}}\\
\hline
 {{\rm I}} & {\chi \times \chi' \rtimes \sigma} & {1} & {4} & {\chi\chi'\sigma^2} \\
\hline
 {{\rm IIb}} & {\chi1_{\GL_2} \rtimes \sigma} & {1} & {3} & {\chi^2\sigma} \\
\hline
 {{\rm IIIa}} & {\chi \rtimes \sigma_{{\rm St}_{\GL_2}} } & {0} & {2} & {\chi\sigma^2} \\
\hline
 {{\rm VIb}} & {\tau(T, |\cdot|^{-\frac{1}{2}}\sigma)} & {0} & {1} & {\sigma^2} \\
\hline
\end{array}
\end{align}
Here, $\chi$, $\chi'$, and $\sigma$ are unramified quasicharacters on $F^\times$ such that representations are irreducible and admit $(T,\Lambda)$-Bessel models. For $a\in F^\times$. For later use we set $\alpha=\chi(\varpi)$, $\beta=\chi'(\varpi)$, and $\gamma=\sigma(\varpi)$. According to \cite[Table A2]{RobertsSchmidt}, for type IIIa and type VIb, the corresponding representations are unitarizable if and only if the inducing quasi-characters are unitary, whereas, for type I, the unitarity of the inducing quasi-characters is equivalent to the corresponding representations being unitarizable and tempered. In what follows, having a global application in mind, we suppose the unitarity of all the inducing characters, i.e., $|\alpha|=|\beta|=|\gamma|=1$.

Similarly as (\ref{RSint-f0}), for a Schwartz-Bruhat function $\phi\in\mathcal{S}(L^2)$, a character $\mu$ on $F^\times$, and $s\in\C$ with $\Re(s)>1$, we define a function on $\sG^{\#}(G)$ by
\begin{align}
f_\phi^{(s,\Lambda,\mu)}(g^\#)=\mu(\det g^\#)|\det g^\#|^{s+1}\int_{L^\times}\phi([\begin{smallmatrix} 0 & 1 \end{smallmatrix}][\begin{smallmatrix} \tau & 0 \\ 0 & \tau^\dagger \end{smallmatrix}]g^\#) \Lambda\mu_L(\tau) |\tau \tau^\dagger|^{s+1} \d^\times\tau,\notag
\end{align}
where $\mu_L(\tau):=\mu(\tau \tau^\dagger), \tau\in L^\times$ and $\d^\times\tau$ is the normalized Haar measure on $L^\times$ such that $\vol(\cO_L)=1$.
For an irreducible smooth admissible representation $\pi$ of $\sZ(F)\bsl\sG(F)$, admitting $(T,\Lambda)$-Bessel model and $B\in \cB(T,\Lambda)[\pi]$, define the zeta integral by
\begin{align}
Z(\phi,B,s,\mu;g)=\int_{F^{\times}}\int_{K^{\#}}B(\sm(a1_2,a)\iota(k^{\#})g)\mu(a)|a|^{s-1}f_\phi^{(s,\Lambda,\mu)}
(k^\#)\d k^{\#}\d^\times a,\quad s\in\C, \Re(s)>1.\label{zeta int}
\end{align}
It is shown that $Z(\phi,B,\mu;g)$ is a rational function in $q^{-s}$ (\cite[Lemma 5.3.1]{SchmidtTran}). Then, the local $L$-function $L(s,\pi,\mu)$ is defined to be the unique function such that $L(s,\pi,\mu)^{-1}$ is a polynomial in $q^{-s}$ with constant term $1$ and such that $\tfrac{Z(\phi,B,s,\mu,1)}{L(s,\pi,\mu)}\in \C[q^s, q^{-s}]$ for all $\phi$ and $B$ (\cite{PS97}, see also \cite[Proposition 5.3.3]{SchmidtTran}). 
By \cite[Proposition 3.2]{PS97}, there exists an entire function $\varepsilon(\pi,s,\mu,\psi)$ such that the local functional equation
\begin{align}
\varepsilon(\pi,s+1/2,\mu,\psi)\frac{Z(\phi,B,s,\mu;g)}{L(s+1/2,\pi,\mu)}=\frac{Z(\widehat \phi,B^\dagger,-s,\mu^{-1};g)}{L(-s+1/2,\hat \pi,\mu^{-1})}\label{localFE}
\end{align}
holds. Here, we set $ B^\dagger\in\cB(T,\Lambda^\dagger)[\pi]$ as $B^\dagger(g):=B\left(\left[\begin{smallmatrix} 1 & \sa^{-1}\sb & 0 & 0 \\ 0 & -1 & 0 & 0 \\ 0 & 0 & -1 & 0 \\ 0 & 0 & -\sa^{-1}\sb & 1 \end{smallmatrix}\right]g\right)$, $g\in\sG(F)$, and $\Lambda^\dagger$ denote the Galois conjugate of $\Lambda$ defined by $\Lambda^\dagger(\tau):=\Lambda(\tau^\dagger)$ for $\tau\in L^\times$, and $\widehat \phi$ is the Fourier transform of $\phi \in {\mathcal S}(L^2)$ defined as 
$$
\widehat \phi(x,y)=\int_{L^2} \phi(u,v) \psi_L\left(\langle  \left[\begin{smallmatrix} x \\ y \end{smallmatrix}\right] , 
\left[\begin{smallmatrix} u \\ v  \end{smallmatrix}\right] \rangle \right)\,\d u\d v
$$ 
with $\d u$ and $\d v$ being the Haar measures on $L$ such that $\vol(\cO_L)=1$. 

\begin{comment}
The $L$-functions $L(s,\pi,\mu)$ are explicitly determined for several cases as shown in \cite[Table 5]{SchmidtTran}, which, for $\mu=1$, coincide with the local $L$-functions listed in \cite[Table A.8]{RobertsSchmidt}.

\begin{align}
\label{Table2}
\begin{array}{|c|c|c|}
\hline
 {{\rm type}} & {\pi} & {L(s,\pi)} \\
\hline
 {{\rm I}} & {\chi \times \chi' \rtimes \sigma} & {(1-\alpha\beta\gamma q^{-s})^{-1}(1-\alpha\gamma q^{-s})^{-1}(1-\beta\gamma q^{-s})^{-1}(1-\gamma q^{-s})^{-1}} \\
\hline
 {{\rm IIb}} & {\chi1_{\GL_2} \rtimes \sigma} & {(1-\alpha^2\gamma q^{-s})^{-1}(1-\gamma q^{-s})^{-1}(1-\alpha\gamma q^{-s-\frac{1}{2}})^{-1}(1-\alpha\gamma q^{-s+\frac{1}{2}})^{-1}} \\
\hline
 {{\rm IIIa}} & {\chi \rtimes \sigma_{{\rm St}_{\GL_2}} } & {(1-\alpha\gamma q^{-s-\frac{1}{2}})^{-1}(1-\gamma q^{-s-\frac{1}{2}})^{-1}}  \\
\hline
 {{\rm VIb}} & {\tau(T, |\cdot|^{-\frac{1}{2}}\sigma)} & {(1-\gamma q^{-s-\frac{1}{2}})^{-2}} \\
\hline
\end{array} 
\end{align}
\end{comment}
We shall get explicit evaluations of the zeta integrals $Z(\phi,B,s,\mu;g)$ for particular choices of $\phi$, $B$ and $g$ as shown in the next table, which allow us to determine $\varepsilon(s,\pi,\mu,\psi)$: 
\begin{align}
\label{Table3}
\begin{array}{|c|c|c|c|c|c|c|c|c|}
\hline
  & {{\rm type}} & {\phi} & {B} & {\Lambda} & {\mu} & L/F & {g}\\
\hline
 {\rm case\,1} & {\text{I or IIb}} & {\cchi_{\cO_{L}\oplus\cO_{L}}} & {K{\rm -fixed}} & {{\rm unramified}} & {{\rm unramified}} & - & {1_4}\\
\hline
 {\rm case\,2} & {\text{I or IIb}} & {\cchi_{\fp_{L}^e\oplus 1+\fp_{L}^e}} & {K{\rm -fixed}} & {{\rm unramified}} & {{\rm ramified}} & {\rm inert} & {b }\\
\hline
 {\rm case\,3} & {\text{I or IIb}} & {\widehat{\cchi_{\fp_{L}^e\oplus 1+\fp_{L}^e}}} & {K{\rm -fixed}} & {{\rm unramified}} & {{\rm ramified}} & {\rm inert} & {b}\\
 \hline
 {\rm case\,4} & {\text{I or IIb}} & {\cchi_{\fp_{L}^{f}\oplus 1+\fp_{L}^{f}}} & {K{\rm -fixed}} & {{\rm ramified}} & {{\rm unramified}} & {\rm inert} & {b'}\\
 \hline
 {\rm case\,5} & {\text{I or IIb}} & {\widehat{\cchi_{\fp_{L}^{f}\oplus 1+\fp_{L}^{f}}}} & {K{\rm -fixed}} & {{\rm ramified}} & {{\rm unramified}} & {\rm inert} & {b'}\\
\hline
 {\rm case\,6} & {\text{I or IIb}} & {\cchi_{\cO_{L}\oplus\cO_{L}}} & {K_0(\fp){\rm -fixed}} & {{\rm unramified}} & {{\rm unramified}} & {\rm inert} & {\eta}\\
\hline
 {\rm case\,7} & {\text{IIIa}} & {\cchi_{\cO_{L}\oplus\cO_{L}}} & {K_0(\fp){\rm -fixed}} & {{\rm unramified}} & {{\rm unramified}} & {\rm inert} & {\eta}\\
\hline
 {\rm case\,8} & {\text{VIb}} & {\cchi_{\cO_{L}\oplus\cO_{L}}} & {K_0(\fp){\rm -fixed}} & {{\rm unramified}} & {{\rm unramified}} & {\rm inert} & {\eta}\\
\hline
\end{array}
\end{align}
Here, the integer $e$ (resp. $f$) is the conductor of $\mu$ (resp. $\Lambda$) defined to be the minimum non-negative integer $n$ satisfying $\mu|_{(1+\fp^n)\cap\fo^\times}=1$ (resp. $\Lambda|_{(1+\fp_L^n)\cap\fo_L^\times}=1$), $b,b'\in\sG(F)$ are given as
\begin{align}
b:=\left[\begin{smallmatrix} \varpi^e 1_2 & T^\dagger \\ 0 & 1_2 \end{smallmatrix}\right],\quad b':=\left[\begin{smallmatrix} \varpi^{2f} &  &  &  \\  & \varpi^f &  &  \\  &  & 1 &  \\  &  &  & \varpi^f \end{smallmatrix}\right],
\label{Def-bF}
\end{align}
and $L/F$ is said to be inert if it is an unramified field extension. For $l,m\in\Z$, set
$$h(\ell,m):=\left[\begin{smallmatrix} \varpi^{\ell+2m} &  &  &  \\  & \varpi^{\ell+m} &  &  \\  &  & 1 &  \\  &  &  & \varpi^{m} \end{smallmatrix}\right]\in\sG(F).$$
As in \cite{PitaleSchmidt}, for smooth representation $(\pi,V_\pi)$ of $\sG(F)$, we define an operator $T_{\ell,m}\in\End(V_\pi)$ by
$$
T_{\ell,m}v:=\vol(K_0(\fp))\int_{K_0(\fp)h(\ell,m)K_0(\fp)}\pi(g)v\,\d g,\quad v\in V_\pi
$$
with $\d g$ the Haar measure on $\sG(F)$ such that $\vol(K)=1$, and the Atkin-Lehner involution by
$$\eta v:=\pi(\eta)v,\quad \eta:=\left[\begin{smallmatrix}  &  &  & -1 \\  &  & 1&  \\  & \varpi &  &  \\ -\varpi &  &  &  \end{smallmatrix}\right],\quad v\in V_\pi.$$

If $\pi$ is irreducible and has the $(T,\Lambda)$-Bessel models, then for any $B\in\cB(T,\Lambda)[\pi]^{K_0(\fp)}$, by \cite[Lemma 4.4 (i)]{PitaleSchmidt}, we have the support conditions:
\begin{gather}
B(h(\ell,m))=B(h(\ell,m)s_2)=0\quad {\rm if\ }\ell<0\label{Support of B1},\\
B(h(\ell,m)s_1s_2)=B(h(\ell,m)s_2s_1s_2)=0\quad {\rm if\ }\ell<-1.
\label{Support of B2}
\end{gather}

\subsection{Unramified computation (case 1)}\label{sec:Ur1}
When $\pi$ is spherical, $\pi$ is of type I or IIb in our cases. Recall that there exists a unique element $B_\pi^{0}\in\cB(T,\Lambda)[\pi]^{K}$ such that $B_\pi^{0}(1_4)=1$. Then, the following formula for an unramified $\mu$ in its greatest generality is due to Sugano (\cite[Theorem 2.1]{Sugano84}): 
\begin{align*}
Z(\phi,B_\pi^{0},s,\mu;1_4)=L(s+1/2,\pi,\mu), \quad \Re(s)\gg 0.
\end{align*}
In \cite[Proposition 5.9]{Lemma}, a proof of this formula is given when $L=F\oplus F$ and $\pi$ is of type I based on the explicit formula of $B_\pi^0$ due to \cite{BFF}.

\subsection{Unramified computation (case 2 and case 3)}\label{sec:Ur2}
Now we assume that $e>0$ and that $T\in\vV(F)$ satisfies the following conditions;
\begin{itemize}
\item $L/F$ is an unramified field extension.
\item $T\in\GL_2(\cO),\quad -\frac{\sd}{2}=\tr(TT^\dagger)\in\cO^\times.$
\end{itemize}
Then $\varpi$ remains a prime element in $\fp_L$. Let $b$ be as in \eqref{Def-bF}, and set $\phi=\cchi_{\fp_L^e\oplus 1+\fp_L^e}\in\mathcal{S}(L^2)$. 
Our goal of this subsection is to calculate the zeta integrals $Z(\phi,B_\pi^{0},s,\mu;b)$ and $Z(\widehat{\phi},B_\pi^{0},s,\mu;b)$ explicitly. We remark that $Z(\phi,B_\pi^{0},s,\mu;1_4)$ must be zero because the $a$-integral in \eqref{zeta int} vanishes.

\begin{defn}[Root number] \label{Def:LclGaussSum}
Let $\psi$ and $\mu$ be as above, we define the Gauss sum $W_F(\mu,\psi)$ by
$$
W_F(\mu,\psi):=q^{-\frac{e}{2}}\mu(\varpi)^{-e}\sum_{\alpha\in(\cO/\fp^e)^\times}\psi(\varpi^{-e}\alpha)\mu(\alpha).
$$
\end{defn}
Note that this sum has the absolute value $1$ and is independent of the choice of $\varpi$. The following lemma is more or less well-known (\cite[(2.18)]{Tate1}).

\begin{lem}\label{Gauss sum}
Let $\psi$ and $\mu$ be as above. For $n\in\Z$,
$$
\int_{\cO^\times}\psi(\varpi^n a)\mu(a)\d^\times a
=\begin{cases}
q^{-\frac{e}{2}+1}(q-1)^{-1} \mu(\varpi)^e W_F(\mu,\psi)\quad&(n=-e),\\
0 &(n\ne-e).
\end{cases}
$$
\end{lem}

\begin{comment}
\begin{proof} This is more or less well-known. 
By noting that $\vol(1+\fp^e)=q^{1-e}(q-1)^{-1}$, the case of $n=-e$ is immediate.

When $n>-e$, we choose $u\in\cO$ such that $\mu(1+\varpi^{e-1}u)\ne1$. Then, by the change of variable $a\mapsto a(1+\varpi^{e-1}u)$,
\begin{align*}
\int_{\cO^\times}\psi(\varpi^n a)\mu(a)\d^\times a
&=\int_{\cO^\times}\psi(\varpi^n a(1+\varpi^{e-1}u))\mu(a(1+\varpi^{e-1}u))\d^\times a\\
&=\mu(1+\varpi^{e-1}u)\int_{\cO^\times}\psi(\varpi^n a)\psi(\varpi^{e+n-1}au)\mu(a)\d^\times a\\
&=\mu(1+\varpi^{e-1}u)\int_{\cO^\times}\psi(\varpi^n a)\mu(a)\d^\times a.
\end{align*}
Hence $\int_{\cO^\times}\psi(\varpi^n a)\mu(a)\,\d^\times a=0$.

On the other hand, by considering the change of variables $a\mapsto a(1+\varpi^{e}u)$ for $u\in\cO$ and integrating with respect to $u$, we have
{\allowdisplaybreaks\begin{align*}
\int_{\cO^\times}\psi(\varpi^n a)\mu(a)\d^\times a
&=\int_{\cO} \int_{\cO^\times}\psi(\varpi^n a(1+\varpi^{e}u))\mu(a(1+\varpi^{e}u))\d^\times a \d u\\
&=\int_{\cO} \int_{\cO^\times}\psi(\varpi^n a)\psi(\varpi^{e+n}au)\mu(a)\d^\times a \d u\\
&=\int_{\cO} \psi(\varpi^{e+n}u) du\int_{\cO^\times} \psi(\varpi^n a)\mu(a) \d^\times a.
\end{align*}} The $u$-integral vanishes if $n<-e$.
\end{proof}
\end{comment}

\begin{lem}\label{Gauss sum split} Let $L/F$ be an unramified quadratic extension. Then, for any character $\mu$ of $F$ with $e:={\rm cond}(\mu)>0$, 
$$W_L(\mu_L,\psi_L)=(-1)^{e} W_F(\mu,\psi)^2. $$
\end{lem}
\begin{proof} For a virtual representation $V$ of the Weil group $\mathfrak W_F$ of $F$, let $\varepsilon(s,V,\psi)$ denote the local epsilon factor a la Deligne-Langlands (\cite{Tate2}). By (\cite[(3.4.8)]{Tate2}), we have 
\begin{align}\varepsilon(s,{\rm Ind}_{{\mathfrak W}_L}^{{\mathfrak W}_F}\rho,\psi)=\varepsilon(s,\rho,\psi_L)
\label{Gauss sum split-f0}
\end{align}
for any $\rho$ of degree $0$. Any character $\mu$ of $F^\times$ is viewed as a one-dimensional representation of ${\mathfrak W}_F$ by $F^\times \cong {\mathfrak W}_{F}^{\rm ab}$. If $\mu_L:=\mu \circ \nr_{L/F}$, then ${\rm Ind}_{{\mathfrak W}_L}^{{\mathfrak W}_F}\mu_L\cong \mu \oplus \mu \eta_{L/F}$. If $\mu$ is unramified, then $\varepsilon(s,\mu,\psi)=1$; since $L/F$ is unramified, so is the character $\eta_{L/F}$. By these observations, the equality \eqref{Gauss sum split-f0} holds true when $\rho={\bf 1}$. Thus, \eqref{Gauss sum split-f0} is true when $\rho=\mu$ with $e>0$. Since $\varepsilon(s,\mu\eta_{L/F}^{j}, \psi)=q^{-e(s-1/2)}W_F(\bar \mu\eta_{L/F}^j,\psi)$ for $j=0,1$ and $\varepsilon(s,\mu_L, \psi_L)=q^{-2e(s-1/2)}W_L(\bar \mu_L,\psi_L)$, we obtain 
$W_{L}(\mu_L,\psi_L)=W_F(\mu,\psi)W_F(\mu\eta_{L/F}, \psi)$. By $e>0$ and $\eta_{L/F}(\varpi)=-1$, we have $W_F(\mu\eta_{L/F},\psi)=\eta_{L/F}(\varpi)^{e} W_F(\mu,\psi)=(-1)^{e}W_F(\mu,\psi)$. 
 \end{proof}   

\begin{lem}\label{lem of f}
For $k^\#=\left[\begin{smallmatrix} a & b \\ c & d \end{smallmatrix}\right]\in K^\#$, we have
\begin{align*}
f_\phi^{(s,\Lambda,\mu)}(k^\#)
&=\begin{cases}
q^{-2e+2}(q^2-1)^{-1}\mu_L(d)^{-1}\mu(\det k^{\#})\quad &(c\in\fp_L^e),
\\
0 & ({\rm otherwise}),
\end{cases}\\
f_{\widehat{\phi}}^{(s,\Lambda,\mu)}(k^\#)
&=\begin{cases}
q^{e(2s-3)+2}(q^{2}-1)^{-1}\Lambda(\varpi)^{-e}\mu_L(c)^{-1}W_L(\mu_L,\psi_L)\quad &(c\in\cO_L^\times),\\
0 & (c\notin\cO_L^\times).
\end{cases}\\
\end{align*}

\end{lem}
\begin{proof}
These equations immediately follow by Lemma \ref{Gauss sum}.
\end{proof}

\begin{lem}\label{sum of lambda}
For $u\in\cO^\times$,
$$
\sum_{\substack{\eta\in\cO_L/\fp_L^e\\ u+\nr_{L/F}(\eta)\in \cO^\times}}\mu(u+\nr_{L/F}(\eta))=(-1)^eq^e\mu(u).
$$
\end{lem}

\begin{proof}
Because $L/F$ is an unramified extension, the norm map $\nr_{L/F}:\cO_L^\times\rightarrow \cO^\times$ is surjective. Thus, there exists $v\in\cO_L^\times$ such that $u=\nr_{L/F}(v)$. By replacing $\eta$ by $v\eta$, without loss of generality, we may assume that $u=1$.

For an integer $i$ with $0\le i\le e+1$, we set 
$$
\overline{\cO_{L,i}}=
\begin{cases}
\left\{\eta\in\left(\cO_L/\fp_L^e\right)^\times\ \middle|\ 1+\nr_{L/F}(\eta)\not\equiv 0\ (\text{mod }\fp)\right\}\quad&(i=0),\\
\fp_L^{i}/\fp_L^{e}&(0<i\le e),\\
\varnothing&(i=e+1)
\end{cases}
$$
and
$$
\overline{\cU_{i}}=
\begin{cases}
\cO^\times/(1+\fp^e)\quad &(i=0),\\
(1+\fp^i)/(1+\fp^e)\quad &(0<i\le e),\\
\varnothing &(i=e+1).
\end{cases}
$$
When $0\le i <e/2$, we can check that
$$
\overline{\cO_{L,i}} \setminus \overline{\cO_{L,i+1}} \owns \eta\ (\text{mod }\fp_L^e) \mapsto 1+N_{L/F}(\eta)\ (\text{mod }\fp^e) \in \overline{\cU_{2i}} \setminus \overline{\cU_{2i+1}}
$$
is a surjective map having the fiber of cardinality $q^{e-1}(q+1)$ at each point.

When $e/2\le i\le e$, it is immediate that 
$$
\mu(1+\nr_{L/F}(\eta))=1,\quad \eta\in\overline{\cO_{L,i}}.
$$
Hence we have
\begin{align*}
\sum_{\substack{\eta\in\cO_L/\fp_L^e\\ 1+\nr_{L/F}(\eta)\in \cO^\times}}\mu(1+\nr_{L/F}(\eta))&=\sum_{i=0}^e\sum_{\eta\in \overline{\cO_{L,i}} \setminus \overline{\cO_{L,i+1}}}\mu(1+\nr_{L/F}(\eta))\\
&=q^{e-1}(q+1)\sum_{0\le i<e/2}\sum_{\varepsilon\in\overline{\cU_{2i}} \setminus \overline{\cU_{2i+1}}}\mu(\varepsilon)+\sum_{e/2\le i\le e}\#(\overline{\cO_{L,i}} \setminus \overline{\cO_{L,i+1}}).
\end{align*}
Then the lemma follows from
$$
\sum_{\varepsilon\in\overline{\cU_{2i}} \setminus \overline{\cU_{2i+1}}}\mu(\varepsilon)=
\begin{cases}
0&(0\le i<\frac{e-1}{2}),\\
-1&(i=\frac{e-1}{2})\\
\end{cases}
$$
and
$$
\sum_{e/2\le i\le e}\#(\overline{\cO_{L,i}} \setminus \overline{\cO_{L,i+1}})=q^{2e-2\lceil e/2 \rceil}.
$$
where $\lceil \cdot \rceil:\R\to\Z$ is the ceiling function.
\end{proof}

\begin{lem}\label{B(b)}
Let $B\in\cB(T,\Lambda)[\pi]^{K}$ and $\eta\in\cO$ with $\frac{\sa^6\sd}{4}+\nr_{L/F}(\eta)\in\varpi^{j}\cO^\times$  $(0\le j\le e)$.
\begin{itemize}
\item[(i)]
If we set $Y_\eta=-\sa^2 T^{\dagger}+X_\eta$, then we have
$$
\det Y_{\eta}=-\tfrac{\sa^4\sd}{4}-\sa^{-2}\nr_{L/F}(\eta).
$$
\item[(ii)]
There exist $\tau\in\cO_L^\times$ and $A\in\GL_2(\cO)$ such that
$$
Y_\eta=\tau \left[\begin{smallmatrix} \varpi^j & 0 \\ 0 & 1 \end{smallmatrix}\right]A.
$$
\item[(iii)]
For $n\in\Z$ and $a\in\cO^\times$,
\begin{align*}
&B\left(\sm(\varpi^n a1_2,\varpi^n a)\iota(\left[\begin{smallmatrix} 0 & -1 \\ 1 & \eta^\dagger \end{smallmatrix}\right])b\right)\\
&=\psi \left(-\varpi^n a\cdot \tfrac{\sa^4\sd}{2}\left(\tfrac{\sa^6\sd}{4}+\nr_{L/F}(\eta)\right)^{-1}\right)\cdot B\left(h(e+n-2j,j)\right).
\end{align*}
\end{itemize}
\end{lem}

\begin{proof}
(i) By noting that $\det T=-\sd/4$, $\det X_\eta=-\sa^{-2} \nr_{L/F}(\eta)$, and $\tr(TX_\eta)=\tr(X_\eta T^{\dagger})=0$, we obtain
$$
\det Y_\eta=\tfrac{1}{2}\tr(Y_\eta Y_\eta^{\dagger})=\tfrac{1}{2}\tr(\sa^4 T T^{\dagger}+X_\eta^{\dagger}X_\eta)=\sa^4\det T+\det (X_\eta)=-\tfrac{\sa^4\sd}{4}-\sa^{-2}\nr_{L/F}(\eta).
$$
(ii) By (i) and the assumption that $\frac{\sa^6\sd}{4}+\nr_{L/F}(\eta)\in\varpi^{j}\cO^\times$, we have $Y_\eta\in{\bf GL}_2(F)$. The disjoint union $
\GL_2(F)=\sqcup_{m\ge 0}L^\times \left[\begin{smallmatrix} \varpi^m & 0 \\ 0 & 1 \end{smallmatrix}\right] \GL_2(\cO)$ (\cite[Lemma 2--4]{Sugano84}) implies that there exist $m\in\Z_{\ge0}$, $\ell\in\Z$, $\tau\in\cO_L^\times$, $A\in\GL_2(\cO)$, such that
$$
Y_\eta=(\varpi^\ell \tau)\left[\begin{smallmatrix} \varpi^m & 0 \\ 0 & 1 \end{smallmatrix}\right]A=\tau\left[\begin{smallmatrix} \varpi^{m+\ell} & 0 \\ 0 & \varpi^{\ell} \end{smallmatrix}\right]A.
$$
We now prove that $m=j$ and $\ell=0$. Because $\tau\in\GL_2(\cO)$, the Smith normal form of $Y_\eta$ is $\left[\begin{smallmatrix} \varpi^{m+\ell} & 0 \\ 0 & \varpi^{\ell} \end{smallmatrix}\right]$. Hence by the theory of the Smith normal form, we only have to show that the elements of $Y_\eta$ are coprime. This follows from the identity $\tr(Y_\eta T)=\frac{\sa^2\sd}{2}$ and the assumption $\frac{\sd}{2}\in\cO^\times$.\\
(iii)
By using (\ref{iota}), we can check that
$$
\iota(\left[\begin{smallmatrix} 1 & 0 \\ -\eta^\dagger & 1 \end{smallmatrix}\right])=\left[\begin{smallmatrix} 1_2 & 0 \\ X_{\eta}^\dagger & 1_2 \end{smallmatrix}\right],\quad \iota(\left[\begin{smallmatrix} 0 & -1 \\ 1 & 0 \end{smallmatrix}\right]) b \iota(\left[\begin{smallmatrix} 0 & 1 \\ -1 & 0 \end{smallmatrix}\right])=\left[\begin{smallmatrix} 1_2 & 0 \\ -\sa^2 T & \varpi^e 1_2 \end{smallmatrix}\right].
$$
Hence we have
$$
\sm(\varpi^n a1_2,\varpi^n a)\iota(\left[\begin{smallmatrix} 0 & -1 \\ 1 & \eta^\dagger \end{smallmatrix}\right])b
=\left[\begin{smallmatrix} \varpi^{n}a1_2 & 0 \\ Y_\eta^{\dagger} & \varpi^{e} 1_2 \end{smallmatrix}\right]\iota(\left[\begin{smallmatrix} 0 & 1 \\ -1 & 0 \end{smallmatrix}\right]).
$$
By (ii), we may choose elements $\tau\in\cO_L^\times$ and $A\in\GL_2(\cO)$ satisfying $Y_\eta=\tau \left[\begin{smallmatrix} \varpi^j & 0 \\ 0 & 1 \end{smallmatrix}\right]A$. Then, a computation shows that
$$
\left[\begin{smallmatrix} \varpi^{n}a1_2 & 0 \\ Y_\eta^{\dagger} & \varpi^{e} 1_2 \end{smallmatrix}\right]=\sn(\varpi^n a(\det Y_\eta)^{-1}Y_\eta)\left[\begin{smallmatrix} \tau & 0 \\ 0 & {}^t\tau^\dagger \end{smallmatrix}\right]h(e+n-2j,j)k,
$$
where $k=\left[\begin{smallmatrix} A & 0 \\ 0 & {}^tA^\dagger \end{smallmatrix}\right]\left[\begin{smallmatrix} a\varpi^j(\det Y_\eta)^{-1} 1_2 & 0 \\ 0 & 1_2 \end{smallmatrix}\right]\left[\begin{smallmatrix} 0 & -1_2 \\ 1_2 & \varpi^e(\det Y_\eta)^{-1}Y_\eta \end{smallmatrix}\right]\in K$. Since $\Lambda$ is unramified and $B$ is $K$-invariant, we have the desired equality.
\begin{comment}
\begin{align*}
&B\left(\sm(\varpi^n a1_2,\varpi^n a)\iota(\left[\begin{smallmatrix} 0 & -1 \\ 1 & \eta^\dagger \end{smallmatrix}\right])b\right)
\\&=\psi(\varpi^n a(\det Y_\eta)^{-1}\cdot \tfrac{\sa^2\sd}{2})\cdot B\left(h(e+n-2j,j)\right)\\
&=\psi \left(-\varpi^n a\cdot \tfrac{\sa^4\sd}{2}\left(\tfrac{\sa^6\sd}{4}+\nr_{L/F}(\eta)\right)^{-1}\right)\cdot B\left(h(e+n-2j,j)\right).
\qedhere
\end{align*}
\end{comment}
\end{proof}

\begin{prop}\label{zeta int ramified character}
Suppose that $\pi$ is of type I or IIb. Let $\mu:F^\times\to\C^1$ and $\Lambda:L^\times\to\C^1$ be characters such that $\text{cond}(\mu)=e>0$ and $\Lambda$ is unramified. 
Suppose that $T=\left[\begin{smallmatrix} \sa & \sb /2\\ \sb/2 & \sc \end{smallmatrix}\right]\in\vV(F)\cap\GL_2(\cO)$ satisfies that $L/F$ is an unramified field extension and $2d\in\cO^\times$. When $\phi=1_{\fp_L^e\oplus 1+\fp_L^e}$ and $\Re(s)>1$, we have
\begin{gather*}
Z(\phi,B_\pi^{0},s,\mu;b)=q^{e(s-\frac{11}{2})+5}(q^4-1)^{-1}(q-1)^{-1}\mu(-2^{-1}\sd)^{-1}W_F(\mu,\psi),\\
Z(\widehat{\phi},B_\pi^{0},s,\mu;b)=(-1)^eq^{e(3s-\frac{11}{2})+5}(q^4-1)^{-1}(q-1)^{-1}\Lambda(\varpi)^{-e}\mu(-2^{-1}\sa^2)W_L(\mu_L,\psi_L)W_F(\mu,\psi).
\end{gather*}
\end{prop}

\begin{proof}
Let $K^\#(\fp^e)=\{k^\#\in K^\#\ |\ k^\#\equiv 1_2\ ({\rm mod}\ \fp_L^e)\}$ be the principal congruence subgroup of level $\fp_L^e$. If $\phi\in\mathcal{S}(L^2)$ is right $K^\#(\fp^e)$-invariant, the function
$$
K^\#\owns k^\#\mapsto B(\sm(a1_2,1)\iota(k^\#)b)f_\phi^{(s,\Lambda,\mu)}(k^\#)\in\C
$$
is left $B^\#(\cO)$-invariant and right $K^\#(\fp^e)$-invariant because $b^{-1}\iota(K^\#(\fp^e))b\subset K$.
Since $K^\#(\fp^e)$ is a normal subgroup of $K^\#$ and $B^\#(\cO)K^\#(\fp^e)=K_0^\#(\fp^e)$, we have
\begin{align}
&Z(\phi,B_\pi^{0},s,\mu;b)\notag\\
&=\sum_{[\gamma]\in B^\#(\cO)\bsl K^\#/K^\#(\fp^e)}\vol(B^\#(\cO)\gamma K^\#(\fp^e))f_\phi^{(s,\Lambda,\mu)}(\gamma)\int_{F^\times}B_\pi^0(\sm(a1_2,a)\iota(\gamma)b)\mu(a)|a|^{s-1}\d^\times a\notag\\
&=[K^\#:K_0^\#(\fp^e)]^{-1}\sum_{[\gamma]\in K_0^\#(\fp^e)\bsl K^\#}f_\phi^{(s,\Lambda,\mu)}(\gamma)\int_{F^\times}B_\pi^0(\sm(a1_2,a)\iota(\gamma)b)\mu(a)|a|^{s-1}\d^\times a.\label{zeta int2}
\end{align}
A complete set of representative for $K^\#(\fp^e)\bsl K^\#$ is given by
\begin{align}
\left[\begin{smallmatrix} 1 & 0 \\ \xi & 1 \end{smallmatrix}\right]\ (\xi\in\fp_L/\fp_L^{e}),\quad \left[\begin{smallmatrix} 0 & -1 \\ 1 & \eta^\dagger \end{smallmatrix}\right]\ (\eta\in\cO_L/\fp_L^{e}).\label{representative}
\end{align}
Now we prove the first equation. By Lemma \ref{lem of f}, when $\gamma$ runs through the above elements, then we have $f_\phi^{(s,\Lambda,\mu)}(1_2)=q^{-2e+2}(q^2-1)^{-1}$ and $f_\phi^{(s,\Lambda,\mu)}(\gamma)=0$ if $\gamma\ne1$. By substituting this in (\ref{zeta int2}), noting that $[K^\#:K_0^{\#}(\fp^{e})]=q^{2e-2}(q^2+1)$, and Lemma \ref{Gauss sum}, we have
\begin{align*}
&Z(\phi,B_\pi^{0},s,\mu;b)\\
&=q^{-4e+4}(q^4-1)^{-1}\int_{F^\times}B_\pi^0\left(\left[\begin{smallmatrix} a\varpi^e 1_2 & aT^\dagger \\ 0 & 1_2 \end{smallmatrix}\right]\right)\mu(a)|a|^{s-1}\d^\times a\\
&=q^{-4e+4}(q^4-1)^{-1}\int_{F^\times}B_\pi^0\left(\left[\begin{smallmatrix} a\varpi^e 1_2 & 0 \\ 0 & 1_2 \end{smallmatrix}\right]\right)\psi(a\cdot\tr(TT^\dagger))\mu(a)|a|^{s-1}\d^\times a\\
&=q^{-4e+4}(q^4-1)^{-1}\mu(-2^{-1}\sd)^{-1}\sum_{n\in\Z}B_\pi^0\left(\left[\begin{smallmatrix} \varpi^{e+n} 1_2 & 0 \\ 0 & 1_2 \end{smallmatrix}\right]\right)\mu(\varpi)^n q^{-n(s-1)}\int_{\cO^\times}\psi(\varpi^{n} a)\mu(a)\d^\times a\\
&=q^{e(s-\frac{11}{2})+5}(q^4-1)^{-1}(q-1)^{-1}\mu(-2^{-1}\sd)^{-1}W_F(\mu,\psi).
\end{align*}
Next, we prove the second equation. When we replace $\phi$ by $\widehat{\phi}$ in (\ref{zeta int2}) and consider the representative (\ref{representative}), we can ignore the contribution of $\xi$-terms in the summation from Lemma \ref{lem of f}. Hence, by noting the $K$-invariance of $B$, we have
\begin{align}
Z(\widehat{\phi},B_\pi^{0},s,\mu;b)&=q^{e(2s-5)+4}(q^4-1)^{-1}\Lambda(\varpi)^{-e}W_L(\mu_L,\psi_L)\notag\\
&\times \sum_{\eta\in\cO_L/\fp_L^e}\int_{F^\times}B_\pi^0(\sm(a1_2,a)\iota(\left[\begin{smallmatrix} 0 & -1 \\ 1 & \eta^\dagger \end{smallmatrix}\right])b)\mu(a)|a|^{s-1}\d^\times a.\label{zeta int phi hat}
\end{align}

By replacing $\eta$ by $\eta+\varpi^e u$ for some $u\in\cO_L$ if we need, we may assume that $\sa^6\sd+\nr_{L/F}(\eta)\in \varpi^j\cO_L^\times$ with $0\le j\le e$. Then, Lemmas \ref{Gauss sum} and \ref{B(b)} (iii) imply
{\allowdisplaybreaks\begin{align*}
&\int_{F^\times}B_\pi^0(\sm(a1_2,a)\iota(\left[\begin{smallmatrix} 0 & -1 \\ 1 & \eta^\dagger \end{smallmatrix}\right])b)\mu(a)|a|^{s-1}\d^\times a\\
&=\sum_{n\in\Z}B_\pi^{0}(h(e+n-2j,j))\mu(\varpi)^{n}q^{-n(s-1)}\int_{\cO^\times}\psi \left(-\varpi^{n} a\cdot \tfrac{\sa^4\sd}{2}\left(\tfrac{\sa^6\sd}{4}+\nr_{L/F}(\eta)\right)^{-1}\right)\mu(a) \d^\times a\\
&=\mu\left(-2\sa^{-4}\sd^{-1}\varpi^{-j}\left(\tfrac{\sa^6\sd}{4}+\nr_{L/F}(\eta)\right)\right)\sum_{n\in\Z}B_\pi^{0}(h(e+n-2j,j))\mu(\varpi)^{n}q^{-n(s-1)}\int_{\cO^\times}\psi(\varpi^{n-j}a)\mu(a) \d^\times a\\
&=q^{-\frac{e}{2}+1+(e-j)(s-1)}(q-1)^{-1}\mu(-2\sa^{-4}\sd^{-1})\mu\left(\tfrac{\sa^6\sd}{4}+\nr_{L/F}(\eta)\right)W_F(\mu,\psi)\cdot B_\pi^{0}(h(-j,j)).
\end{align*}}The last equation is equal to $0$ unless $j=0$ from (\ref{Support of B1}). By substituting this into (\ref{zeta int phi hat}) and Lemma \ref{sum of lambda}, we have that $Z(\phi,B_\pi^{0},s,\mu;b)$ equals
{\allowdisplaybreaks
\begin{align*}&q^{e(3s-\frac{13}{2})+5}(q^4-1)^{-1}(q-1)^{-1}\Lambda(\varpi)^{-e}\mu(-2\sa^{-4}\sd^{-1})\\
&\times W_L(\mu_L,\psi_L)W_F(\mu,\psi) \sum_{\substack{\eta\in\cO_L/\fp_L^e\\ \sa^6\sd+\nr_{L/F}(\eta)\in\cO^\times}}\mu \left(\tfrac{\sa^6\sd}{4}+\nr_{L/F}(\eta)\right)\\
&=(-1)^eq^{e(3s-\frac{11}{2})+5}(q^4-1)^{-1}(q-1)^{-1}\Lambda(\varpi)^{-e}\mu(-2^{-1}\sa^{2})W_L(\mu_L,\psi_L)W_F(\mu,\psi).
\qedhere
\end{align*}}
\end{proof}

\begin{cor}
Let $\pi$, $T$, $\mu$, and $\Lambda$ be as in Proposition \ref{zeta int ramified character}, then
\begin{align*}
\varepsilon(\pi,s,\mu,\psi)&=(-1)^e q^{4e(\frac{1}{2}-s)}\Lambda(\varpi)^{-e}\mu(-\sa^{-2}\sd)\overline{W_L(\mu_L,\psi_L)W_F(\mu,\psi)^2} \\
&=q^{4e(\frac{1}{2}-s)}\Lambda(\varpi)^{-e}\mu(-\sa^{-2}\sd)\overline{W_F(\mu,\psi)^4}.
\end{align*}
\end{cor}
\begin{proof}
   The first equality follows from Proposition \ref{zeta int ramified character} and \eqref{localFE}; the second equality is due to Lemma \ref{Gauss sum split}.  
\end{proof}

\subsection{Unramified computation (case 4 and case 5)}\label{sec:Ur3}
Suppose that $f>0$ and $L/F$ is inert. The section $f_{\phi}^{(s,\Lambda,\mu)}$ restricted to $K^\#$ is explicitly determined as
\begin{align*}
f_\phi^{(s,\Lambda,\mu)}(k^\#)
=\begin{cases}
q^{-2f+2}(q^2-1)^{-1}\Lambda(d)\quad &(c\in\fp_L^e),
\\
0 & ({\rm otherwise}),
\end{cases}
\quad k^\#=\left[\begin{smallmatrix} a & b \\ c & d \end{smallmatrix}\right]\in K^\#
\end{align*}
by a direct calculation.
\begin{prop}
Suppose that $\pi$ is of type I or IIb. Let $\mu$, $\Lambda$, $T$ be as above. When $\phi=\cchi_{\fp_L^{f}\oplus 1+\fp_L^{f}}$ and $\Re(s)>1$, we have
\begin{align*}
Z(\phi,B_{\pi}^{0},s,\mu;b')=Z(\widehat{\phi},B_{\pi}^{0},s,\mu;b')=q^{-4f+4}(q^{4}-1)^{-1}L(s+1/2,\pi,\mu).
\end{align*}
\end{prop}
\begin{proof}
It suffices to calculate the integral $Z(\phi,B_{\pi}^{0},s,\mu;b')$ due to \eqref{localFE} and the fact that $\varepsilon(\pi,s,\mu,\psi)=1$ (\cite[Theorem 4.4]{PS97}).
For any $k^\#=\left[\begin{smallmatrix} a & b \\ c & d \end{smallmatrix}\right]\in K_0^{\#}(\fp^{f})$, it can be checked that $k^\#=\left[\begin{smallmatrix} \tau &  \\  & \lambda \tau^\dagger \end{smallmatrix}\right]\left[\begin{smallmatrix} 1 & \beta \\  & 1 \end{smallmatrix}\right]\left[\begin{smallmatrix} 1 &  \\ cd^{-1} & 1 \end{smallmatrix}\right]$, where we set $(\tau,\lambda,\beta)\in \fo_{L}^\times\times\fo^\times\times\fo_L$ as
\begin{align*}
\tau=(\det k^\#)^{-1}d,\quad \lambda=(\det k^\#)(dd^\dagger),\quad \beta=-b/d.
\end{align*}
This and a short calculation yield the proposition due to \cite[Theorem 2.1]{Sugano84}.
\end{proof}

\subsection{Computation for old forms of type ${\rm I}$ and ${\rm IIb}$ (case 6)}\label{sec:Iold}
Now we assume that $\pi$ has the trivial character. Then, $\pi$ has the $(T,\Lambda)$-Bessel model if and only if $\Lambda=1$. Recall that the dimension of $V_\pi^{K_0(\fp)}$ is equal to $4$ or $3$ according as $\pi$ is of type I or type IIb.

\begin{prop}\label{local zeta int I, IIb}
Let $\pi$ be a smooth admissible irreducible representation of type I or type IIb with trivial central character. Suppose that $\mu : F^\times\to \C^1$ is an unramified character. When $\phi=\cchi_{\cO_L\oplus\cO_L}$ and $\Re(s)>1$, we have
\begin{align*}
&Z(\phi,B,s,\mu;\eta)=\frac{1}{q^2+1}L(s+1/2,\pi,\mu)\\
&\times \left[\eta B+q^{-1}T_{1,0} B+\left\{\mu(\varpi)^{-1}q^{s+1}+\mu(\varpi)q^{-s+1}-\tr((q^{-1}T_{1,0}+\eta)|_{V_\pi^{K_0(\fp)}})\right\}B\right](1_4)
\end{align*}
for $B\in\cB(T,1)[\pi]^{K_0(\fp)}$.
\end{prop}

\begin{proof} We will give a proof of the case that $\pi$ is of type I. In the case of type IIb, the proof is similar. Let $\{B_j^*\}_{1\le j\le4}\subset \cB(T,1)[\pi]^{K_0(\fp)}$ be an eigenbasis of $T_{1,0}$ with eigenvalues $\{\lambda_j\}_{1\le i\le 4}$. Then, we suffice to check the desired equation only for $B=B_j^*$ $(1\le j\le 4)$. Let $(\eta_{ij})_{1\le i,j\le 4}$ be the representation matrix of $\eta|_{K_0(\fp)}$ with respect to $\{B_j^*\}_{1\le i\le 4}$. Then, the first relation in (\ref{relations of B}) shows that
$$
B_j^*(h(l+1,0))=\lambda_j q^{-3} B_j^*(h(l,0)),\quad l\in\Z_{\ge 0}.
$$

It is easy to check the equality $
f_\phi^{(s,\Lambda,\mu)}(k^\#)=L(s+1,\Lambda\mu_L)$ for all $k^\#\in K^\#$. By considering the left $K_0^\#(\fp)$-coset decomposition of $K^\#$;
\begin{align}
K^\#=K_0^\#(\fp)\sqcup\left(\sqcup_{\xi\in\cO_L/\fp_L}\left[\begin{smallmatrix} 1 & \xi \\ 0 & 1 \end{smallmatrix}\right]\left[\begin{smallmatrix} 0 & 1 \\ 1 & 0 \end{smallmatrix}\right]K_0^\#(\fp^e)\right),\label{left coset decomp}
\end{align}
the zeta integral $Z(\phi,B,s,\mu;\eta)$ becomes $[K^\#:K_0^\#(\fp)]^{-1}L(s+1,\Lambda\mu_L)$ times
{\allowdisplaybreaks\begin{align}
&\int_{F^\times}\left\{\eta B(\sm(a1_2,a))+\sum_{\xi\in\cO_L/\fp_L}\eta B\left(\sm(a1_2,a)\iota\left(\left[\begin{smallmatrix} 1 & \xi \\ 0 & 1 \end{smallmatrix}\right]\left[\begin{smallmatrix} 0 & 1 \\ 1 & 0 \end{smallmatrix}\right]\right)\right)\right\}\mu^{-1}(a)|a|^{s-1}\d^\times a\notag\\
&=\sum_{l\in\Z}\left\{\eta B(h(l,0))+q^2 B(h(l+1,0))\right\}\mu (\varpi)^{l}q^{-l(s-1)}\notag\\
&=\sum_{l=0}^\infty \eta B(h(l,0))\mu(\varpi)^lq^{-l(s-1)}+q^{s+1}\mu(\varpi)^{-1}\sum_{l=0}^\infty  B(h(l,0))\mu(\varpi)^lq^{-l(s-1)}\label{zeta int for general B}
\end{align}}

because $B(h(l,0))=0$ if $l<0$  by (\ref{Support of B1}). By substituting $B=B_j^*$ in (\ref{zeta int for general B}) and $\Lambda=1$, we have
\begin{align*}
Z(\phi,B_j^*,s,\mu;\eta)=\frac{L(s+1,\Lambda\mu_L)}{q^2+1}\left\{\sum_{i=1}^4 \frac{\eta_{ij}B_i^*(1_4)}{1-\lambda_j\mu(\varpi)q^{-(s+2)}}+\frac{\mu(\varpi)^{-1}q^{s+1} B_j^*(1_4)}{1-\lambda_i\mu(\varpi) q^{-(s+2)}}\right\}.\notag\\
\end{align*}
Thus, we may put
\begin{align*}
Z(\phi,B_j^*,s,\mu;\eta)&=\frac{1}{q^2+1}\prod_{j=1}^4(1-\lambda_jq^{-2}X)^{-1}\frac{A_0+A_1X+A_2X^2+A_3X^3+A_4X^4}{X(1-q^{-2}X^2)}.
\end{align*}
with $X=\mu(\varpi)q^{-s}$ for some $A_i\in\C\,(1\le i\le 4)$. \\
By noting that $\{\lambda_1,\lambda_2,\lambda_3,\lambda_4\}=\{\alpha\beta\gamma q^{3/2}, \beta\gamma q^{3/2}, \alpha\gamma q^{3/2}, \gamma q^{3/2}\}$ from \cite[Table A.8]{RobertsSchmidt}, we have
\begin{align}
(q^2+1)\cdot \frac{Z(\phi,B_j^*,s,\mu;1_4)}{L(s+1/2,\pi,\mu)}=\frac{A_0+A_1X+A_2X^2+A_3X^3+A_4X^4}{X(1-q^{-2}X^2)}.
\label{zeta int old form}
\end{align}
A direct computation shows that
{\allowdisplaybreaks\begin{gather}
A_0=q B_j^*(1_4),\label{A_0}\\
A_1=\eta B_j^*(1_4)+q^{-1}\tr(T_{1,0}|_{V_\pi^{K_0(\fp)}})\cdot B_j^*(1_4)-q^{-1}\cdot T_{1,0} B_j^*(1_4).\label{A_1}
\end{gather}}The functional equation (\ref{localFE}) and the relation $\varepsilon(\pi,s,\mu,\psi)=1$ (\cite[Theorem 4.4]{PS97}) show that the right-hand side of (\ref{zeta int old form}) is invariant under the transformation $X\to X^{-1}$. By comparing the coefficients, we have
\begin{gather}
A_2=q^{-2}(q^2-1)A_0,\quad A_3=-q^{-2}A_1,\quad A_4=-q^{-2}A_0.
\label{A_2-4}
\end{gather} 
Substituting (\ref{A_0}), (\ref{A_1}), and (\ref{A_2-4}) into (\ref{zeta int old form}), the right-hand side of \eqref{zeta int old form} equals
{\allowdisplaybreaks\begin{align*}
\left[\eta B_j^*+q^{-1}T_{1,0} B_j^*+\left\{\mu(\varpi)^{-1}q^{s+1}+\mu(\varpi)q^{-s+1}-\tr(q^{-1}T_{1,0}|_{V_\pi^{K_0(\fp)}})\right\}B_j^*\right](1_4).
\end{align*}}
\end{proof}

\subsection{Computation for newforms of type ${\rm III}a$ (case 7)}\label{sec:IIIa}
By \cite[\S 9.1]{PitaleSchmidt}, there exists a basis $\{B_1,B_2\}$ of $\cB(T,\Lambda)[\pi]^{K_0(\fp)}$ such that
\begin{gather}
T_{1,0}B_1=\alpha\gamma q B_1,\quad T_{0,1}B_2=\gamma q B_2,\notag\\
T_{0,1}B_1=\alpha\gamma^2 (\alpha q+1) q B_1,\quad T_{0,1}B_2=\alpha\gamma^2 (\alpha^{-1}q+1) q B_2,\label{basis1}\\
\eta B_1=\alpha\gamma B_2,\quad \eta B_2=\gamma B_1.\notag
\end{gather}

We now consider the values of $B_1$ and $B_2$ at the identity element. We remark that \cite[Lemma 9.2]{PitaleSchmidt} ensures its non-vanishing and $\Lambda(\varpi)=\alpha\gamma^2$ by the condition of central character of $\pi$.

\begin{lem}  \label{IIIa B(1)}
Let $T\in\vV(F)$ such that $L/F$ is an unramified field extension. Suppose that $\mu:F^\times\to\C^1$ and $\Lambda:L^\times\to\C^1$  are unramified characters. Then we have
$$
B_2(1_4)=\alpha^{-1} B_1(1_4)\not=0. 
$$
\end{lem}

\begin{proof}
For $B\in \cB(T,\Lambda)[\pi]^{K_0(\fp)}$ and any non-negative integer $l$, by Lemmas 5.1 and 5.3 in \cite{PitaleSchmidt}, we get the following relations;
{\allowdisplaybreaks\begin{gather}
T_{1,0}B(h(l,0))=q^3 B(h(l+1,0)),\notag\\
T_{1,0}B(h(l,0)s_2)=q^2(q-1)B(h(l+1,0))+q^2B(h(l-1,1)s_1s_2),\notag\\
T_{1,0}B(h(l,0)s_2s_1s_2)=q^2(q-1)B(h(l+1,0))+\mu(\varpi)B(h(l-1,0)s_2s_1s_2)\notag\\
+(q^2-1)B(h(l-1,1)s_1s_2),\label{relations of B}\\
T_{0,1}B(h(l,0))=q^3(q+1)B(h(l,1)),\notag\\
T_{0,1}B(h(l,0)s_2)=q^4B(h(l,1)s_1s_2)+q\mu(\varpi)B(h(l-2,1)s_1s_2)\notag\\
+q^2(q-1)B(h(l,1))-q\mu(\varpi)B(1_4).\notag
\end{gather}}Set $B=B_1$. By putting $l=0$ in the above relations, applying the relation (\ref{basis1}), and using the equation that $B_2(h(l,0))=\gamma B_1(h(l-1,0)s_2s_1s_2)$ for $l\ge0$, we have
\begin{gather}
\alpha\gamma qB_1(s_2)=\alpha\gamma(q-1)B_1(1_4)+q^2B_1(h(-1,1)s_1s_2),\label{Eq1}\\
\mu(\varpi)\gamma^{-1}(q^{-1}-1)B_2(1_4)=\alpha\gamma(q-1)B_1(1_4)+(q^2-1)B_1(h(-1,1)s_1s_2),\label{Eq2}\\
\mu(\varpi)(\alpha q+1)B_1(s_2)=q^4B_1(h(0,1)s_1s_2)+\frac{\mu(\varpi)(\alpha q^2-\alpha q-q^2-1)}{q+1}B_1(1_4).\label{Eq3}
\end{gather}
Similarly, by putting $l=1$ in the third relation in (\ref{relations of B}), we have
\begin{gather}
\mu(\varpi)q^{-3}(1-q)B_2(1_4)=\mu(\varpi)\alpha q^{-2}(q-1)B_1(1_4)+(q^2-1)B_1(h(0,1)s_1s_2).\label{Eq4}
\end{gather}
From (\ref{Eq1}), (\ref{Eq2}), (\ref{Eq3}), and (\ref{Eq4}), we obtain $B_2(1_4)=\alpha^{-1}B_1(1_4)$.
 The non-vanishing of $B_1$ at $1_4$ is a consequence of \cite[Theorem 9.3]{PitaleSchmidt}. 
\end{proof}

\begin{prop}\label{zeta int IIIa}
Suppose that $\pi$ is of type IIIa. Let $T\in\vV(F)$ be such that $L/F$ is an unramified field extension. Suppose that $\mu:F^\times\to\C^1$ and $\Lambda:L^\times\to\C^1$  are unramified characters. When $\phi=1_{\cO_L\oplus\cO_L}$ and $\Re(s)>1$, we have
\begin{gather}
Z(\phi,B,s,\mu;\eta)=\frac{\Lambda(\varpi)^{-1}\mu(\varpi)^{-1}q^{s+1}}{q^2+1}L(s+1/2,\pi,\mu)\cdot B(1_4),\quad B\in\cB(s,\Lambda)[\pi]^{K_0(\fp)}.\notag
\end{gather}
\end{prop}

\begin{proof}
By substituting $B=B_i$ in (\ref{zeta int for general B}) for $i=1,2$, the statement is immediate for $B=B_i$ by (\ref{basis1}), the first relation in (\ref{relations of B}), and Lemma \ref{IIIa B(1)}.
Since $B_1$ and $B_2$ span $\cB(T,\Lambda)[\pi]^{K_0(\fp)}$, we complete the proof.
\end{proof}

\begin{cor}
Let $\pi$, $T$, $\mu$, and $\Lambda$ be as in Proposition \ref{zeta int IIIa}, then
$$\varepsilon(\pi,s,\mu,\psi)=\mu(\varpi)^{2}q^{2(\frac{1}{2}-s)}.$$
\end{cor}

\subsection{Computation for newforms of type ${\rm VI}b$ (case 8)} \label{sec:IVb}
Suppose that $\pi$ is of type VIb. By \cite[Table 3]{PitaleSchmidt}, any element $B\in\cB(T,\Lambda)[\pi]^{K_0(\fp)}$ satisfies
\begin{gather}
T_{1,0}B=\gamma q B,\quad \eta B=\gamma B.\label{basis2}
\end{gather}
By \cite[Theorem 9.3]{PitaleSchmidt}, the value $B(1_4)\not=0$ if and only if $B=0$.

\begin{prop}\label{zeta int VIb}
Suppose that $\pi$ is of type VIb. Let $T\in\vV(F)$ be such that $L/F$ is an unramified field extension. Suppose that $\mu:F^\times\to\C^1$ and $\Lambda:L^\times\to\C^1$  are unramified characters. When $\phi=1_{\cO_L\oplus\cO_L}$ and $\Re(s)>1$, we have
\begin{gather}
Z(\phi,B,s,\mu;\eta)=\frac{\Lambda(\varpi)^{-1}\mu(\varpi)^{-1}q^{s+1}}{q^2+1}L(s+1/2,\pi,\mu)\cdot B(1_4),\quad B\in\cB(T,\Lambda)[\pi]^{K_0(\fp)}.\notag
\end{gather}
\end{prop}
\begin{proof}
The proof is similar to that of Proposition $\ref{zeta int IIIa}$ by using relations in (\ref{relations of B}). 
\end{proof}

\begin{cor}
Let $\pi$, $T$, $\mu$, and $\Lambda$ be as in Proposition \ref{zeta int VIb}, then
$$\varepsilon(\pi,s,\mu,\psi)=\mu(\varpi)^{2}q^{2(\frac{1}{2}-s)}.$$
\end{cor}

\subsection{The values of Bessel models at $1_4$ for old forms}\label{sec:B(1) for old forms}
For quasi-characters $\chi$, $\chi'$, and $\sigma$ on $F^\times$ which are trivial on $\cO^\times$, the representation $(\chi \times \chi' \rtimes \sigma,V)$ can be realized as the right regular representation of the space of all smooth functions $f$ on $\sG(F)$ satisfying
$$
f\left(\left[\begin{smallmatrix} a & \ast & \ast & \ast \\ 0 & b & \ast & \ast \\ 0 & 0 & ca^{-1} & 0 \\ 0 & 0 & \ast & cb^{-1} \end{smallmatrix}\right]g\right)=|a^2b||c|^{-\frac{3}{2}}\chi(a)\chi'(b)\sigma(c)f(g),\quad a,b,c\in F^\times, g\in\sG(F).
$$
The space $V$ has a unique $K$-invariant element $f_K\in V$ such that $f_K(1_4)=1$. 
Let $I=K\cap\left[\begin{smallmatrix} \cO & \fp & \cO & \cO \\ \cO & \cO & \cO & \cO \\ \fp & \fp & \cO & \cO \\ \fp & \fp & \fp & \cO 
\end{smallmatrix}\right]$ be the Iwahori subgroup of $K$. For $w\in W$, let $f_w\in V^I$ be the unique element such that $f_w|_K=\cchi_{IwI}$.

By \cite[Table 3]{PitaleSchmidt2}, the space $V^{I}$ is 8-dimensional space spanned by
$$
f_{1_4}, \quad f_{s_1}, \quad f_{s_2},\quad f_{s_2s_1}, \quad f_{s_1s_2s_1}, \quad f_{s_1s_2}, \quad f_{s_1s_2s_1s_2}, \quad f_{s_2s_1s_2}. 
$$
A basis of $4$-dimensional space $V^{K_0(\fp)}$ is given by 
$$
f_1^{{\rm I}}:=f_{1_4}+f_{s_1}, \quad f_2^{{\rm I}}:=f_{s_2}+f_{s_2s_1}, \quad f_3^{{\rm I}}:=f_{s_1s_2s_1}+f_{s_1s_2},\quad f_4^{{\rm I}}:=f_{s_1s_2s_1s_2}+f_{s_2s_1s_2}.
$$

Recall that $\pi$ is isomorphic to $\chi \times \chi' \rtimes \sigma$ or $\chi1_{\GL_2}\rtimes\sigma$ for some characters $\chi$, $\chi'$, and $\sigma$ on $F^\times$ according as $\pi$ is of type I or IIb. The representation $(\chi1_{\GL_2}\rtimes\sigma,V)$ can be realized as a subrepresentation of $|\cdot|^{1/2}\chi\times |\cdot|^{-1/2}\chi \rtimes \sigma$, and a basis of $3$-dimensional space $V^{K_0(\fp)}$ is given by 
$$
f_1^{{\rm II}b}:=f_{1_4}+f_{s_1},\quad f_2^{{\rm IIb}}:=f_{s_2}+f_{s_2s_1}+f_{s_1s_2s_1}+f_{s_1s_2},\quad f_3^{{\rm IIb}}:=f_{s_1s_2s_1s_2}+f_{s_2s_1s_2}.
$$ 
The representation matrices of operators $T_{1,0}$ and $\eta$ on $V_\pi^{K_0(\fp)}$ are given in \cite[Table 3]{PitaleSchmidt2}. 
\begin{comment}
as follows;
\begin{align}
\label{Table4}
\begin{array}{|c|c|c|}
\hline
 {} & {{\rm type\ I}} & {{\rm type\ IIb}} \\
\hline
 {T_{1,0}} & {\left[\begin{smallmatrix} \alpha\beta\gamma q^{3/2} & 0 & 0 & 0 \\ \alpha\beta\gamma(q-1)q^{1/2} & \beta\gamma q^{3/2} & 0 & 0 \\ \alpha\beta\gamma(q-1)q^{1/2} & \beta\gamma(q-1)q^{1/2} & \alpha\gamma q^{3/2} & 0 \\ \alpha\beta\gamma(q-1)q^{1/2} & \beta\gamma(q-1)q^{1/2} & \alpha\gamma(q-1)q^{1/2} & \gamma q^{3/2} \end{smallmatrix}\right]} & {\left[\begin{smallmatrix} \alpha^2\gamma q^{3/2} & 0 & 0 \\ \alpha^2\gamma(q-1)q^{1/2} & \alpha\gamma q^{2} & 0 \\ \alpha^2\gamma(q-1)q^{1/2} & \alpha\gamma(q^2-1) & \gamma q^{3/2} \end{smallmatrix}\right]} \\
\hline
 {\eta} & {\left[\begin{smallmatrix} 0 & 0 & 0 & \gamma q^{3/2} \\ 0 & 0 & \alpha\gamma q^{1/2} & 0 \\ 0 & \beta\gamma q^{-1/2} & 0 & 0 \\ \alpha\beta\gamma q^{-3/2} & 0 & 0 & 0 \end{smallmatrix}\right]} & {\left[\begin{smallmatrix} 0 & 0 & \gamma q^{3/2} \\ 0 & \alpha\gamma & 0 \\ \alpha^2\gamma q^{-3/2} & 0 & 0 \end{smallmatrix}\right]} \\
\hline
\end{array}
\end{align}
\end{comment}
Now we assume that $\pi$ has the trivial character. Let $(\ell_\pi^0,f_K)$ be the Bessel data for $\pi$ defined above. For $\bullet\in\{{\rm I}, {\rm IIb}\}$, define a basis $\{B_i^{\bullet}\}_{1\le i\le \dim V_\pi^{K_0(\fp)}}$ of $V_\pi^{K_0(\fp)}$ by
$$
B_i^{\bullet}(g):=\ell_\pi^{0}(\pi(g)f_i^{\bullet}),\quad g\in\sG(F).
$$
We need the values of these functions at $g=1_4$, which are given as follows. 
\begin{prop}\label{values of B(1_4)}
Let $\pi$ be a smooth admissible irreducible $K$-spherical representation of type I or IIb, and $\{B_i^{\bullet}\}$ be as above for $\bullet\in\{{\rm I}, {\rm IIb}\}$. Then
\begin{itemize}
\item[(i)] If $\pi$ is of type I,
{\allowdisplaybreaks
\begin{gather*}
B_1^{{\rm I}}(1_4)=\frac{\alpha\beta}{(q-\alpha)(q-\beta)},\quad B_2^{{\rm I}}(1_4)=\frac{-q \beta}{(q-\alpha)(q-\beta)},\\
B_3^{{\rm I}}(1_4)=\frac{-q \alpha}{(q-\alpha)(q-\beta)},\quad B_4^{{\rm I}}(1_4)=\frac{q^2}{(q-\alpha)(q-\beta)}.
\end{gather*}}
\item[(ii)] If $\pi$ is of type IIb,
{\allowdisplaybreaks\begin{gather*}
B_1^{{\rm IIb}}(1_4)=\frac{\alpha^2}{(q^{1/2}-\alpha)(q^{3/2}-\alpha)},\quad B_2^{{\rm IIb}}(1_4)=\frac{-q^{1/2}(1+q)\alpha}{(q^{1/2}-\alpha)(q^{3/2}-\alpha)},\\
B_3^{{\rm IIb}}(1_4)=\frac{q^2}{(q^{1/2}-\alpha)(q^{3/2}-\alpha)}.
\end{gather*}}
\end{itemize}
\end{prop}

\begin{proof} Set $d:=\dim V_\pi^{K_0(\fp)}$. By the construction of the elements $\{f_i^{\bullet}\}_{1\le i\le d}$, we have $B_\pi^{0}=\sum_{i=1}^{d} B_i^{\bullet}$. For any $\ell\in\Z_{\ge0}$, we let $T_{1,0}^l$ act on the both sides of this equation and then evaluate at $g=1_4$ using the first relation in (\ref{relations of B}); thus, 
$$
B_\pi^{0}(h(\ell,0))=q^{-3\ell}\sum_{i=1}^{d} (T_{1,0}^\ell B_i^{\bullet})(1_4).
$$
The values of the left hand side is given in \cite[Corollary 1.8]{BFF}; by \cite[Table 3]{PitaleSchmidt2}, we have a system of linear equations among $B_i^\bullet(1_4)$, which is solved easily.
\end{proof}

\subsection{Local periods}\label{sec:Local period}
For an irreducible admissible unitalizable $K_0(\fp)$-spherical representation $(\pi,V_\pi)$ on $\sG(F)$ having the $(T,\Lambda)$-Bessel model, we fix a pair $(\ell^\pi,\xi^\pi)\in(V_{\pi}^*)^{T,\Lambda} \times V_\pi^{K_0(\fp)}$ satisfying $\ell_\pi(\xi_\pi)=1$. When $\pi$ is $K$-spherical, we choose $(\ell^\pi,\xi^\pi)$ as the unramified Bessel datum $(\ell_0^\pi,\xi_0^\pi)$. Then, there exists a unique $\sG(F)$-invariant inner product $\langle \cdot | \cdot \rangle$ on $V_\pi$ such that $\langle \xi^\pi | \xi^\pi\rangle=1$. For $s\in\C$, $\phi\in\mathcal{S}(L^2)$, $g\in\sG(F)$, and a character $\mu$ on $F^\times$, we define the local period of $\pi$ as
$$
\mathbb{I}_{\pi,K_0(\fp)}^{(s)}(\xi^\pi,\phi,\Lambda,\mu;g)=\sum_{B\in\mathcal{B}((V_\pi^*)^{T,\Lambda},K_0(\fp))} Z^*(\phi,B,s,\mu;g)\overline{B(1_4)},
$$
where 
$$
Z^*(\phi,B,s,\mu;g):=L(s+1/2,\pi,\mu)^{-1}Z(\phi,B,s,\mu;g)
$$
is the normalized zeta integral and $\mathcal{B}((V_\pi^*)^{T,\Lambda},K_0(\fp))$ is an orthonormal basis of $\cB(T,\Lambda)[\pi]^{K_0(\fp)}$ with respect to the inner product on $\cB(T,\Lambda)[\pi]$ via the $\sG(F)$-isomorphism
\begin{align}
V_\pi\owns\xi \mapsto (g \mapsto \ell_\pi(\pi(g)\xi))\in\cB(T,\Lambda)[\pi].\label{Isomorphism}
\end{align}

\subsubsection{Local periods for type I and IIb}\label{sec:Local period type I and IIb}
When $\pi$ is of type I or IIb, we recall that the representation space can be realized as in \S \ref{sec:B(1) for old forms}. We choose $\xi^\pi=\xi_0^\pi=f_K$. Then, since we are dealing with induced representations from unitary characters, the standard inner product $
\langle f | g \rangle:=\int_{K}f(k)\overline{g(k)}dk$ ($f,g\in V_\pi
$) becomes $\sG(F)$-invariant and satisfies $\langle \xi_0^\pi | \xi_0^\pi\rangle=1$. \\For $\bullet\in\{ {\rm I} , {\rm IIb} \}$, the set $\{\langle B_i^{\bullet} | B_i^{\bullet} \rangle^{-\frac{1}{2}} B_i^{\bullet}\}_{1\le i \le \dim{V_\pi^{K_0(\fp)}}}$ forms an orthonormal basis of $\cB(T,\Lambda)[\pi]^{K_0(\fp)}$. By \cite[]{Schmidt}, the values $\langle B_i^{\bullet} | B_i^{\bullet} \rangle$ are given as follows:
{\allowdisplaybreaks\begin{gather*}
\langle B_i^{\rm{I}} | B_i^{\rm{I}} \rangle=q^{i-1}(q+1)\quad (1 \le i \le 4),\\
\langle B_1^{\rm{IIb}} | B_1^{\rm{IIb}} \rangle=(q+1),\quad \langle B_2^{\rm{IIb}} | B_2^{\rm{IIb}} \rangle=q(q+1)^{2},\quad \langle B_3^{\rm{IIb}} | B_3^{\rm{IIb}} \rangle=q^{3}(q+1).
\end{gather*}}By Propositions \ref{local zeta int I, IIb} and \ref{values of B(1_4)}, and the matrix representation of $T_{1,0}$ and $\eta$ given in \cite[Table 3]{PitaleSchmidt2}, we can calculate the local period by a direct calculation. \footnote{We have used \textsc{Mathematica 13}.}

\begin{prop}\label{local period I, IIb}
Let $\pi$ be a smooth admissible irreducible unitary representation of type I or type IIb with trivial central character. Suppose that $\mu : F^\times\to \C^1$ is an unramified character. When $\phi=\cchi_{\cO_L\oplus\cO_L}$, we have
{\allowdisplaybreaks\begin{align*}
\mathbb{I}_{\pi,K_0(\fp)}^{(s)}(\xi_0^\pi,\phi,\Lambda,\mu;\eta)&=\frac{2(q-1)}{q^5(q^2+1)}L(1,\pi;{\rm Std})\notag\\
&\times \left\{\mu(\varpi)^{-1}q^{s+1}+\mu(\varpi) q^{-s+1}-\frac{1}{q+1}\tr(T_{1,0}+q\eta|_{V_\pi^{K_0(\fp)}})\right\},
\end{align*}}where $L(s,\pi;{\rm Std})$ is the standard L-function of $\pi$ defined by
$$
L(s,\pi;{\rm Std})=(1-\alpha q^{-s})^{-1}(1-\beta q^{-s})^{-1}(1-\alpha^{-1} q^{-s})^{-1}(1-\beta^{-1} q^{-s})^{-1}(1-q^{-s})^{-1}.
$$
or
$$
L(s,\pi;{\rm Std})=(1-\alpha q^{-s+\frac{1}{2}})^{-1}(1-\alpha q^{-s-\frac{1}{2}})^{-1}(1-\alpha^{-1} q^{-s+\frac{1}{2}})^{-1}(1-\alpha^{-1} q^{-s-\frac{1}{2}})^{-1}(1-q^{-s})^{-1}.
$$
according as $\pi$ is of type I or IIb.
\end{prop}

\subsubsection{Local period for type IIIa}\label{sec:Local period type IIIa}
Suppose that $\pi$ is of type IIIa. Then, as remarked before, $|\alpha|=|\gamma|=1$ and $\alpha\not=1$. By Lemma \ref{IIIa B(1)}, there exists a basis $\{B_i\}_{i=1,2}$ of $\cB(T,\Lambda)[\pi]^{K_(\fp)}$, unique up to constant, satisfying the relations \eqref{basis1} and $B_2(1_4)=\alpha^{-1}B_1(1_4)\not=0$. Thus, we can uniquely fix $\{B_i^{{\rm IIIa}}\}_{i=1,2}$ by requiring $B_{1}^{{\rm IIIa}}(1_4)=1$ and the datum $(\ell_\pi,\xi^\pi)$ by  $\ell_\pi(\xi^\pi)=B^{{\rm IIIa}}_{1}$ to form \eqref{Isomorphism} and induce the 
 $\sG(F)$-invariant inner product on $\cB(T,\Lambda)[\pi]$ such that $\langle B_{1}^{{\rm IIIa}}|B_{1}^{{\rm IIIa}} \rangle=1$. Since $\langle B_1|B_2\rangle=0$ by \cite[Lemma 2.11]{DPSS}, we then have $\langle B_2^{{\rm IIIa}}|B_2^{{\rm IIIa}}\rangle=1$ and $\langle B_1^{{\rm IIIa}}|B_2^{{\rm IIIa}}\rangle=0$ due to $|\alpha|=1$. The following proposition is immediate from Proposition \ref{zeta int IIIa}. 

\begin{prop}\label{local period IIIa}
Let $\pi$ be a smooth admissible irreducible unitary representation of type IIIa. Suppose that $ : F^\times\to \C^1$ is an unramified character. When $\phi=\cchi_{\cO_L\oplus\cO_L}$, 
\begin{align*}
&\mathbb{I}_{\pi,K_0(\fp)}^{(s)}(\xi^\pi,\phi,\Lambda,\mu;\eta)=\frac{2\Lambda(\varpi)^{-1}\mu(\varpi)^{-1}q^{s+1}}{q^2+1}.
\end{align*}
\end{prop}
%\begin{proof}
%By Proposition \ref{zeta int IIIa} and noting that %$|\alpha|=1$, we have
%\begin{align*}
%\mathbb{I}_{\pi,K_0(\fp)}^{(s)}
%(\xi^\pi,\phi,\Lambda,\mu;\eta)&=\frac{\Lambda(\varpi)^{-1}\mu(\varpi)^{-1}q^{s+1}}{q^2+1}\left\{|B_{1}^{{\rm IIIa}}(1_4)|^2+|B_{2}^{{\rm IIIa}}(1_4)|^2\right\}\\
%&=\frac{2\Lambda(\varpi)^{-1}\mu(\varpi)^{-1}q^{s+1}}{q^2+1}.
%\qedhere
%\end{align*}
%\end{proof}

\subsubsection{Local period for type VIb}\label{sec:Local period type VIb}
Suppose that $\pi$ is of type VIb. We recall that $\dim (V_\pi^{K_0(\fp)})=\dim (\cB(T,\Lambda)^{K_0(\fp)})=1$; thus, by \cite[Theorem 9.3]{PitaleSchmidt}, there exists a unique element $B^{{\rm VIb}}\in\cB(T,\Lambda)^{K_0(\fp)}$ satisfying $B^{{\rm VIb}}(1_4)=1$. We fix the datum $(\ell_\pi,\xi^\pi)$ by $\ell_\pi(\xi^\pi)=B^{{\rm VIb}}$. The following result is immediate from Proposition \ref{zeta int VIb}.

\begin{prop}\label{local period VIb}
Let $\pi$ be a smooth admissible irreducible unitary representation of type VIb. Suppose that $ : F^\times\to \C^1$ is an unramified character. When $\phi=\cchi_{\cO_L\oplus\cO_L}$, 
\begin{align*}
&\mathbb{I}_{\pi,K_0(\fp)}^{(s)}(\xi^\pi,\phi,\Lambda,\mu;\eta)=\frac{\Lambda(\varpi)^{-1}\mu(\varpi)^{-1}q^{s+1}}{q^2+1}.
\end{align*}
\end{prop}

\section{Computation of Gross-Prasad period for ramified characters} \label{sec:CompGP}
We shall add a new case to the local computations in \cite[\S2]{DPSS}. Let $F$, $\fo$, $T$, $L$ and $\Lambda$ be as in \S \ref{sec:CompLZ}. Suppose that $T=\left[\begin{smallmatrix} \sa & \sb/2 \\ \sb/2 & \sc \end{smallmatrix}\right]\in{\rm Sym}_2(F)$, the character $\Lambda$ of $L^\times$ and an irreducible unitary representation $(\pi,V_\pi)$ of $\sG(F)/\sZ(F)$ satisfy
\begin{align}
&\text{
$\sa,\sc\in\fo^\times, \sb\in\fo, \sd:= \sb^2-4\sa\sc\in\fo^\times$,\quad $L/F$ is an unramified field extension
}\label{Assumption} \\
&\text{$\Lambda$ is of conductor $1+\fp_L$,}
\label{AuumptionL}\\
&\text{$(\pi,V_\pi)$ is $K$-spherical.}
 \label{AssumptionPi}  \end{align}
Let $(\cdot\mid \cdot)_\pi$ be the $\sG(F)$-invariant hermitian inner product on $V_\pi$, and $\xi_0\in V_\pi$ be a $K$-fixed unit vector; define a $\sZ(F)K$-bi-invariant function $\Phi_\pi$ on $\sG(F)$ as 
$$
\Phi_\pi(g):={(\pi(g)\xi_0\mid \xi_0)_\pi}
,\quad g\in\sG(F).
$$
Our main concern is the limit 
\begin{align}
J_0^{T}(g):&=\lim_{k\to\infty}\int_{\fp^{-k}\sN(\fo)}\Phi_\pi(b^{-1}ngb)\psi_T(n)\d n\notag\\
&=\lim_{k\to\infty}\int_{(\fp^{-k})^3}\Phi_{\pi}\left(b^{-1}\left[\begin{smallmatrix} 1 &  & x & y \\  & 1 & y & z \\  &  & 1 &  \\  &  &  & 1 \end{smallmatrix}\right]g b\right)\psi(\sa x+\sb y+\sc z) \d x \d y \d z,\quad g\in\sG(F),\label{LocalPeriod1}
\end{align}
where $b:=h(0,1)=\left[\begin{smallmatrix} \varpi^2 &  &  &  \\  & \varpi &  &  \\  &  & 1 &  \\  &  &  & \varpi \end{smallmatrix}\right]\in \sG(F)$. The limit \eqref{LocalPeriod1} makes sense; Indeed, by \cite{Liu}, %when we view the integral as a sequence indexed by $k$,% 
the values of integrals are independent of large enough $k$. The local period $J_{0}^{T,\Lambda}$ and its normalization $J^{T,\Lambda}$ are defined by
$$
J_0^{T,\Lambda}:=\int_{F^\times\backslash L^\times}J_0^{T}(\sm(\tau,\det \tau))\Lambda(\tau)\d \tau, \quad J^{T,\Lambda}:=\frac{L(1,\pi;{\rm Ad})L(1,\chi_{L/F})}{\zeta_F(2)\zeta_F(4)L(\frac{1}{2},\pi\times\mathcal{AI}(\Lambda^{-1}))}J_0^{T,\Lambda}. 
$$
Note that, in our setting that $\pi$ is unramified and $\Lambda$ is ramified, the local $L$-factor $L(s,\pi\times \mathcal{AI}(\Lambda^{-1}))$ is identically $1$. 
\begin{thm}\label{localGPPeriod} Under condition \eqref{Assumption}, we have $J^{T,\Lambda}=q^{-4}$. 
\end{thm}
For the rest of this section, we sketch a proof of this theorem, introducing key lemmas that are used in the proof. Let $v:F\to\Z\cup\{\infty\}$ be the normalized valuation defined by $|x|=q^{-v(x)}$. Recall the double coset decomposition $\sG(F)=\sqcup_{\ell,m\ge 0}\sZ(F)K\cdot h(\ell,m)\cdot K$.
\begin{lem}\label{CartanDecomp}
Let $x,y,z\in F$.  
\begin{itemize}
\item[(i)]
If $\left[\begin{smallmatrix} 1 &  & \varpi^{-2}x & \varpi^{-1}y \\  & 1 & \varpi^{-1}y & z \\  &  & 1 &  \\  &  &  & 1 \end{smallmatrix}\right]\in \sZ(F)K\cdot h(l,m)\cdot K$, then
\begin{align*}
\ell&=2\{\min(0,-2+v(x),-1+v(y),v(z))\\
&\qquad-\min(0,-2+v(xz-y^2),-2+v(x),-1+v(y),v(z))\},\\
m&=\min(0,-2+v(xz-y^2),-2+v(x),-1+v(y),v(z))\\
&\qquad-2\min(0,-2+v(x),-1+v(y),v(z)).
\end{align*}
\item[(ii)]
If $\left[\begin{smallmatrix} \varpi &  & \varpi^{-1}x & y \\  & \varpi^{-1} & \varpi^{-2}y & \varpi^{-1}z \\  &  & \varpi^{-1} &  \\  &  &  & \varpi \end{smallmatrix}\right]\in \sZ(F)K\cdot h(l,m)\cdot K$, then
\end{itemize}
\begin{align*}
\ell&=2\{\min(0,v(x),-1+v(y),v(z))\\
&\qquad-\min(0,v(xz-y^2),v(x),1+v(y),v(z))\}+2,\\
m&=\min(0,v(xz-y^2),v(x),1+v(y),v(z))\\
&\qquad-2\min(0,v(x),-1+v(y),v(z)).
\end{align*}
\end{lem}
\begin{proof}
These formulae are special cases in (\cite[Proposition 3.2]{KL}). (See also \cite[Lemma 2.7]{DPSS}.)
\end{proof}
\begin{lem}\label{LemmaIcal}
For positive integers $m,k\in\Z_{>0}$,
\begin{align}
\int_{\substack{x,y,z\in\fo^\times\\xz-y^2\in \varpi^{k}\fo^\times}}\psi(\varpi^{-m}(\sa x+\sb y+\sc z))\d x\d y\d z=
\begin{cases}
-q^{-k-1}(1-q^{-1})^2\quad &(m=1)\\
0 &(m>1)
\end{cases}.\label{Ical}
\end{align}
\end{lem}

\begin{proof}
By the change of variables $x\mapsto z^{-1}y^2+u$ $(u\in \varpi^{k}\fo^\times)$ and $y\mapsto zy$, the integral \eqref{Ical} equals
\begin{align}
&\int_{u\in\varpi^k\fo^\times}\psi(\varpi^{-m}\sa u)\d u\times\int_{y,z\in\fo^\times}\psi (\varpi^{-m}(\sa y^2+\sb y+\sc)z)\sd y\sd z.\label{Ical1}
\\
&=q^{-k}\begin{cases}
0 \quad&(k\le m-2),\\
-q^{-1} &(k=m-1),\\
1-q^{-1} &(k\ge m),
\end{cases}\quad 
\times\quad  (1-q^{-1})
\begin{cases}
-q^{-1} \quad&(m=1),\\
0 &(m\ge 2).\\
\end{cases}
\notag
\end{align}
To have the evaluation of the second integral in \eqref{Ical}, we use $\sa y^2+\sb y+ \sc \in \fo^\times(\,y\in\fo)$, which is proved by $L/F$ being an unramified field extension. 
\end{proof}
For $\ell,m\ge 0$, we set $H(\ell,m):=\Phi_\pi(h(\ell,m))$. 
%The following lemma explicitly describes special values of $J_0^{T}(g)$.
\begin{lem}\label{J_0(g)}
\begin{align*}
J_0^{T}(1_4)=&q^{-3}H(0,0)-q^{-3}H(0,1)+(1-q^{-1})H(0,2)-(q+q^{-1}-1)H(0,3)\\
&\qquad-(1-q^{-1})H(2,1)-(q+q^{-1}-1)H(2,2),\\
J_0^{T}(s_1)=&(1-q^{-1})H(0,2)-q(1-q^{-1})^2H(0,3)+q^{-1}H(2,0)-2H(2,1)\\
&\qquad +q(1-q^{-1})^2H(2,2)+H(4,0).
\end{align*}
\end{lem}

\begin{proof}
Since $\Phi_\pi$ is bi-$\sZ(F)K$-invariant, the value $J_0^{T}(g)$ is represented by a sum of $H(\ell,m)$. For $r,s,t,k\in\Z$, define the subset of $F^3$ as $\Pi(r,s,t):=\varpi^{r}\fo^\times\times\varpi^{s}\fo^\times\times\varpi^{t}\fo^\times$. Then, we consider the following subsets of $F^3$:
{\allowdisplaybreaks\begin{align*}
&D_{0}:=\fp^{-k}\setminus\fp^{-1}, 
&D_0'&:=\fp^{-k}\setminus\fp^{-1},\\
&D_{1}:=\Pi(-1,-1,-1),
&D'_{1}&:=\Pi(-1,-1,-1),\\
&D_{2}:=\Pi(1,0,-1), 
&D'_{2}&:=\cup_{\substack{r\geq 0 \\ t\geq 0}} \Pi(r,0,t),\\
&D_3:=\cup_{\substack{r\ge 2\\s \ge 1\\t\ge 0}}\Pi(r,s,t), 
&D'_{3}&:=\cup_{\substack{r\ge 0\\s \ge 1\\t\ge 0}}\Pi(r,s,t),\\
&D_4:=\cup_{\substack{s \ge 1\\t\ge 0}}\Pi(1,s,t)\sqcup\cup_{\substack{r \ge 2\\s\ge 1}}\Pi(r,s,-1), 
&D'_{4}&:=\cup_{\substack{s\ge 0\\t\ge 1}}\Pi(-1,s,t)\sqcup\cup_{\substack{s\ge 0\\ t\ge 0}}\Pi(r,s,-1),\\
&D_5:=\cup_{\substack{s \ge 0\\t\ge 0}}\Pi(0,s,t), &D'_{5}&:=\cup_{t\ge 0}\Pi(-1,-1,t)\sqcup\cup_{\substack{r\ge 0\\t\ge -1}}\Pi(r,-1,t),\\
&D_6:=\cup_{\substack{s \ge 1\\t \ge 0}}\Pi(-1,s,t), &D'_{6}&:=\cup_{s\ge 0}\Pi(-1,s,-1),\\
&D_7:=\cup_{s \ge 1}\Pi(1,s,-1)\sqcup \cup_{t \ge 0}\Pi(1,0,t)\sqcup \cup_{\substack{r \ge 2\\t\ge -1}}\Pi(r,0,t),\hspace{-80pt}& & \\
&D_8:=\cup_{s \ge 0}\Pi(0,s,-1),\\
&D_9:=\cup_{s \ge 0}\Pi(-1,s,-1)\sqcup\cup_{t \ge 0}\Pi(-1,-1,t),\\
&D_{10}:=\cup_{\substack{r \ge 0\\t \ge -1}}\Pi(r,-1,t).
\end{align*}}
We have two expressions of $\fp^{-k}$ as disjoint unions: $\fp^{-k}=\sqcup_{0\le i\le10}D_i=\sqcup_{0\le j\le6}D'_j$; the $D_i$-decomposition is used for $J_0^T(1_4)$, while the $D_j'$-decomposition is for $J_{0}^{T}(s_1)$. By means of these, using Lemmas \ref{CartanDecomp}, \ref{LemmaIcal}, we calculate the integrals
\begin{align}
I(D):=\int_{D}\psi(\sa x+\sb y+\sc z)\,\d x\d y\d z,\quad (D\subset\fp^{-k})
\label{Int-psiabc}
\end{align}
as linear combinations of $H(\ell,m)$. It is rather easy to prove $I(D_0)=I(D_0')=0$; for other cases, the coefficients are listed in the following table: 
\begin{align*}
\begin{array}{|c|c|c|}
\hline
 {D} & {(\ell,m)} & \text{coefficient of $H(\ell,m)$} \\
\hline\hline
 {D_{1*}^{-2}} & {(2,2)} & {J+q(1-q^{-1})^2}\\
\hline
 {D_{1\flat}^{-2}} & {(0,3)} & {-q(1-q^{-1})^2}\\
\hline
 {D_{2*}^{0}} & {(2,0)} & {-q^{-1}(1-2q^{-1})(1-q^{-1})}\\
\hline
 {D_{2\flat}^{0}} & {(0,1)} & {-q^{-2}(1-q^{-1})}\\
\hline
 {D_{3}} & {(0,0)} & {q^{-3}}\\
\hline
 {D_{4}} &  {(0,1)} & {q^{-2}(1-2q^{-1})}\\
\hline
 {D_{5}} &  {(0,2)} & {1-q^{-1}}\\
\hline
 {D_{6}} &  {(0,3)} & {-1}\\
\hline
 {D_{7}} &  {(2,0)} & {q^{-1}(1-2q^{-1})(1-q^{-1})}\\
\hline
 {D_{8}} &  {(2,1)} & {-(1-q^{-1})}\\
\hline
 {D_{9}} &  {(2,2)} & {1-J}\\
\hline
 {D_{10}} &  {(4,0)} & {0}\\
\hline
\end{array}
\hspace{0.1cm}
\begin{array}{|c|c|c|}
\hline
 {D} & {(\ell,m)} & \text{coefficient of $H(\ell,m)$} \\
\hline\hline
 {D_{1*}'^{-2}} & {(2,2)} & {J+q(1-q^{-1})^2} \\
\hline
 {D_{1\flat}'^{-2}} & {(0,3)} & {-q(1-q^{-1})^2}\\
\hline
 {D'_{2}} & {(0,2)} & {1-q^{-1}}\\
\hline
 {D'_{3}} & {(2,0)} & {q^{-1}}\\
\hline
 {D'_{4}} & {(2,1)} & {-2}\\
\hline
 {D'_{5}} & {(2,2)} & {-J}\\
\hline
 {D'_{6}} & {(4,0)} & {1}\\
\hline
\end{array}
\end{align*}
where $J:=\int_{\varpi^{-1}\fo^\times}\psi(\sb y)\d y$, and we put $D_{i*}^{u}:=D_1\cap\{(x,y,z)\mid v(xz-y^2)=k\}$, $D_{i\flat}^{u}=D_{i}\setminus D_{i*}^{u}$ for $u\in\Z$ ($D_{j*}'^{u}$ and $D_{j*}'^{u}$ are defined in a similar manner). Having this table, we complete the proof.
\end{proof}

\begin{lem}\label{J_0(kappa)}
For $\kappa\in\GL_2(\fo)$, we have 
$
J_0^{T}(\kappa)=
\begin{cases}
J_0^{T}(1_4)\quad &(\kappa\in{\bf \Gamma}^{0}(\fp)),\\
J_0^{T}(s_1)\quad &(\kappa\notin{\bf \Gamma}^{0}(\fp))
.\end{cases}
$\end{lem}
\begin{proof}
For $g\in\sG(F), \gamma_1,\gamma_2\in {\bm \Gamma}^{0}(\fp):={\bf GL}_2(\fo)\cap\left[\begin{smallmatrix} \fo & \fp \\ \fo & \fo \end{smallmatrix}\right]$, it  holds that
\begin{align}
J_0^{T}(\sm(\gamma_1,\det \gamma_1) g\sm(\gamma_2,\det \gamma_2))=J_0^{T\ss(\gamma_1^{-1})}(g).\label{LocalPeriod2}
\end{align}
By the double coset decomposition: ${\bf GL}_2(\fo)={\bm \Gamma}^{0}(\fp)\sqcup\Gamma^{0}(\fp)\left[\begin{smallmatrix} 0 & 1 \\ 1 & 0 \end{smallmatrix}\right]\Gamma^{0}(\fp)$, and the relation \eqref{LocalPeriod2}, the case of $\kappa\in{\bf \Gamma}^{0}(\fp)$ follows immediately, while we have $J_0^{T}(\kappa)=J_0^{T\ss(\gamma)}(s_1)$ for some $\gamma\in{\bf \Gamma}^{0}(\fp)$ if $\kappa\notin{\bf \Gamma}^{0}(\fp)$. the equation $J_0^{T\ss(\gamma)}(s_1)=J_0^{T}(s_1)$ follows from Lemma \ref{J_0(g)}. Indeed, $J_0^{T}(s_1)$ is independent of $T$, and $T\ss(\gamma)$ also satisfies assumption \eqref{Assumption}.
\end{proof}

\begin{lem}\label{LemmaGP}
\begin{align}
J_0^{T,\Lambda}=
(q+1)^{-1}\Bigl\{q^{-3}H(0,0)&-q^{-3}H(0,1)-H(0,3)-q^{-1}(2,0)\label{GPperiod}\\
&\qquad +(1+q^{-1})H(2,1)+H(2,2)-H(4,0)\Bigr\}.\notag
\end{align}
\end{lem}

\begin{proof} Since $\fo^\times(1+\fp_L)\subset{\bf \Gamma}^{0}(\fp)$, the integral $J_0^{T,\Lambda}$ equals
\begin{align}
[\fo_L^\times:\fo^\times(1+\fp_L)]^{-1}\sum_{\tau}\Lambda(\tau)J_0^{T}(\sm(\tau,\det\tau)),\label{LocalPeriod3}
\end{align}
where $\tau$ runs through a set of complete representative of $\fo_L^\times/\fo^\times(1+\fp_L)$. By Lemma \ref{J_0(kappa)} and the fact $\fo_L^\times \cap{\bf \Gamma}^{0}(\fp)=\fo^\times(1+\fp_L)$, the quantity \eqref{LocalPeriod3} becomes
$$
(q+1)^{-1}\left\{J_0^{T}(1_4)+J_0^{T}(s_1)\sum_{\substack{\tau\\\tau\notin\fo^\times(1+\fp_L)}}\Lambda(\tau)\right\}=(q+1)^{-1}\left\{J_0^{T}(1_4)-J_0^{T}(s_1)\right\}.
$$
Then, Lemma \ref{J_0(g)} yields the equation \eqref{GPperiod}
\end{proof}

\begin{proof}[Proof of Theorem \ref{localGPPeriod}]
The integral $J^{T,\Lambda}$ can be calculated directly by means of Lemmas \ref{J_0(g)}, \ref{LemmaGP}, and (\cite[Proposition 2.10]{DPSS}).
\end{proof}

%%%%%%%%%%%%%%%%%%%%%%%%%%%%%%%%%%%%%%%%%

\end{document}